\documentclass{amsart}

\usepackage{amsmath}
\usepackage{amssymb}
\usepackage{mathrsfs}  
\usepackage[hyperfootnotes=false]{hyperref}
\usepackage{enumitem}
\usepackage[multiple]{footmisc}

\usepackage[style=numeric, doi=false, url=false, isbn=false, maxbibnames=99]{biblatex}
\DeclareFieldFormat[article]{title}{#1}
\DeclareFieldFormat[inbook]{title}{#1}
\DeclareFieldFormat[article,periodical]{volume}{\mkbibbold{#1}}

\usepackage{geometry}
\usepackage{thmtools}
\usepackage{thm-restate}
\usepackage{hyperref}
\usepackage{cleveref}

\newtheorem{theorem}{Theorem}[section]
\newtheorem{lemma}[theorem]{Lemma}
\newtheorem{corollary}[theorem]{Corollary}
\newtheorem{proposition}[theorem]{Proposition}

\theoremstyle{definition}
\newtheorem{definition}[theorem]{Definition}

\newtheorem{remark}[theorem]{Remark}

\newcommand\supp{\mathrm{supp}}

\newcommand{\joinR}{\hspace{-.0pt}}
\newcommand{\RomanI}{\text{I}}
\newcommand{\RomanII}{\mbox{\RomanI\joinR\RomanI}}

\let\Re=\undefined\DeclareMathOperator{\Re}{Re}

\usepackage{scalerel}[2014/03/10]
\usepackage[usestackEOL]{stackengine}
\def\intavg{\,\ThisStyle{\ensurestackMath{
    \stackinset{c}{0\LMpt}{c}{0\LMpt}{\SavedStyle-}{\SavedStyle\phantom{\int}}}
    \setbox0=\hbox{$\SavedStyle\int\,$}\kern-\wd0}\int}

\numberwithin{equation}{section}

\allowdisplaybreaks

\begin{document}

\title[The Benjamin--Ono equation with quasi-periodic data]{The Benjamin--Ono equation with quasi-periodic data}
\author{Hagen Papenburg}
\maketitle

\begin{abstract}
We construct local solutions to the Benjamin--Ono equation for quasi-periodic initial data. 
The solution is unique among limits of smooth solutions and depends continuously on the data. 
Our result applies to a richer class of quasi-periodic functions than previous theorems. 
Central to the argument is an a-priori estimate, the proof of which utilizes Strichartz estimates for quasi-periodic functions obtained recently via decoupling, and a quasi-periodic extension of Tao's gauge transform. 
As a byproduct of our method, we also establish new local wellposedness results in certain anisotropic Sobolev spaces.
\end{abstract}

\section{introduction}
This paper is concerned with the Cauchy problem for the
Benjamin--Ono equation with quasi-periodic data $u_0$:
\begin{equation}\label{BOequation1}
    u_t+Hu_{xx}=\partial_x(u^2), \quad u(0)=u_0.
\end{equation}
Here $H$ is the Hilbert transform, see (\ref{Hilberttransform}).
The Benjamin--Ono equation is an effective model for internal waves in stratified fluids of large depth \cite{Benjamin_1967}. 
The function $u(x,t)$ represents the vertical displacement of the interface between two fluids of differing densities, 
known as the \textit{pycnocline}. Internal waves explain phenomena such as \textit{dead water}, observed for example in fjords, when a ship's propulsion produces internal waves between a fresh surface water layer and a saltier, denser bottom layer, and those waves in turn affect the ship's maneuverability. We refer the reader to~\cite{BONA2008538},~\cite[Chapter 3]{MR4400881}, \cite[Chapter 5]{alma9959979203606533} for a discussion of the modeling aspect of the Benjamin--Ono equation. \\  

The Benjamin--Ono equation has been studied extensively for initial data in Sobolev spaces on the line $\mathbb{R}$ and on the circle $\mathbb{T}=\mathbb{R}/\mathbb{Z}$. It is of physical interest to consider initial data that is neither periodic nor spatially decaying. One step in that direction is to investigate quasi-periodic data. There are few results for initial data that is not a local perturbation of either periodic or decaying data. \\ 

A function $f:\mathbb{R} \rightarrow \mathbb{R}$ is called \textit{quasi-periodic} with two-dimensional frequency module if it can be expressed in the form $f(t)=g(t\omega+\mathbb{Z}^2)$ for a function $g:\mathbb{R}^2/\mathbb{Z}^2 \rightarrow \mathbb{R}$ and a frequency vector $\omega = (\omega_1,\omega_2) \in \mathbb{R}^2$ with incommensurable entries. Such a function admits a Fourier representation $f(t)=\sum_{n \in \mathbb{Z}^2} \hat{f}(n) e^{it(\omega\cdot n)},$
where the quasi-periodic Fourier coefficients $\hat{f}(n)$ agree with the classical periodic Fourier coefficients of the function $g:\mathbb{R}^2/\mathbb{Z}^2 \rightarrow \mathbb{R}$, see Section \ref{sectionnotation}.
By a slight abuse of notation we will often identify a quasi-periodic function $f$ with its representative $g$ on the torus. Taking this perspective, we denote by $\partial_x$ the directional derivative on the two-dimensional torus in the $\omega$-direction in order to match the notation in (\ref{BOequation1}). Equivalently, this is the Fourier multiplier with symbol $\omega\cdot n$. 
Thus, we can regard the Benjamin-Ono equation with quasi-periodic data as an equation on the torus. 

From this viewpoint, the dispersion is degenerate; the dispersive term involves derivatives \textit{in only one direction} on the torus. This causes various traditional dispersive PDE tools to break down for this problem.
\\

We indicate the breakdown of some methods based on the complete integrability of the Benjamin-Ono equation.
The Benjamin--Ono equation has been known to have many attributes of a completely integrable system since the late 1970s:
The existence of infinitely many conserved quantities was first observed by Nakamura in \cite{MR550203}, the equation admits formulation as the compatibility condition of a Lax pair \cite{MR591320, MR3484397}, the dynamics has been identified as a member of an infinite hierarchy of commuting Hamiltonian flows \cite{MR637024}, and an inverse scattering apparatus is under development
\cite{MR1073870, MR686248, MR4275336, MR1627325, MR3738309}.

The complete integrability of the Benjamin--Ono equation plays a fundamental role for the wellposedness theories on the line and circle; the sharp wellposedness result in the hierarchy of Sobolev spaces on the circle has been obtained using a Birkhoff normal form transformation \cite{MR4652410}. On the line, the sharp result has been obtained using the method of commuting flows \cite{MR4743514}. Both methods exploit the complete integrability of the equation.

There are still infinitely many conserved quantities in the quasi-periodic setting, the first few being
\begin{equation}\label{conserved1}
    I_1[u]=\int_{\mathbb{T}^2}u(x_1,x_2) \ dx_2dx_2 = \lim_{M \rightarrow \infty} \frac{1}{2M} \int_{-M}^M u(\omega x) \ dx
\end{equation}
\begin{equation}
    I_2[u]= \int_{\mathbb{T}^2} u^2 = \lim_{M \rightarrow \infty} \frac{1}{2M} \int_{-M}^M u^2(\omega x) \ dx
\end{equation}
\begin{equation}
   I_3[u]=\int_{\mathbb{T}^2} \frac{1}{2}uH\partial_xu + \frac{1}{3}u^3 = \lim_{M \rightarrow \infty} \frac{1}{2M} \int_{-M}^M \big(\frac{1}{2}uH\partial_xu + \frac{1}{3}u^3\big)(\omega x) dx 
\end{equation}
\begin{equation}\label{conserved4}
I_4[u]= \int_{\mathbb{T}^2} u^4 + 6u^2Hu_x - 6[\partial_xu]^2 = \lim_{M \rightarrow \infty} \frac{1}{2M}\int_{-M}^M (u^4 + 6u^2Hu_x - 6[\partial_xu]^2)(\omega x)dx.
\end{equation}
Conservation of the mean $I_1[u]$ is immediate as the right-hand side of the Benjamin--Ono equation is a complete derivative. Conservation of the $L^2$-norm $I_2[u]$ follows from a short calculation using integration by parts and exploiting anti-selfadjointness of the Hilbert transform.
The quantity $I_3[u]$ is, in fact, the Hamiltonian generating the Benjamin--Ono dynamics. The first nontrivial conserved quantity is the higher-order expression $I_4[u]$. In the periodic or spatially decaying setting, the first three conserved quantities have physical interpretations as mass, momentum, and energy, respectively. In the quasi-periodic setting, the mass, momentum and energy are generally infinite, and the conserved quantities (\ref{conserved1})-(\ref{conserved4}) represent averages of the corresponding densities. \\

However, in the quasi-periodic setting, these conserved quantities can no longer be used to control the Sobolev norms. The conserved quantities for quasi-periodic solutions (including (\ref{conserved1})-(\ref{conserved4})) only involve derivatives $\partial_x$ in the $\omega$-direction on the two-dimensional torus; they do not provide any control on joint derivatives $\nabla^su$.
See also the discussion in \cite[Section 5]{AmbroseAitzhan}, where it is shown by scaling considerations that a certain generalization of the Gagliardo-Nirenberg inequality to quasi-periodic functions can not hold; in the setting on the line or on the circle, the Gagliardo-Nirenberg inequality features prominently in the common argument deriving bounds on the Sobolev norms of solutions at half-integer levels by using the conserved quantities  \cite[Exercise 4.33]{Taolocalandglobal}. \\

The lack of control over Sobolev norms presents two major challenges for the study of the Benjamin--Ono equation with quasi-periodic data. First, local wellposedness results for Sobolev spaces of quasi-periodic functions do not automatically become global by iteration. Second, the method of commuting flows is not applicable in this setting due to the central role of a-priori bounds for that approach, for example in establishing equicontinuity properties \cite{MR4743514}.  \\

We now state our main theorem.
\begin{theorem}\label{maintheorem}
    Let $\sigma>7/8$, $\omega \in \mathbb{R}^2$ and denote $Y:=Y^{\sigma}_{\omega}$ as in Definition \ref{Yspace}. There exists an absolute constant $\rho_0$ such that for any initial data $u_0 \in B_{Y}(0,\rho_0)$ there exists a solution to the Benjamin--Ono equation in $C([0,1],H^{1+\tilde{\sigma},\tilde{\sigma}}_{\omega}) \cap L^4([0,1],V^{1,\infty}_{\omega})$ for $\tilde{\sigma}<\sigma$. Further, this solution is unique among limits of smooth solutions and the data-to-solution map is Lipschitz from $B_{Y}(0,\rho_0)$ equipped with the $L^2_{\omega}$-norm to $C([0,1],L^2_{\omega})$. The data-to-solution map is also continuous from $B_Y(0,\rho_0)$ equipped with the $Y$-norm to $C([0,1],H_{\omega}^{1+\tilde{\sigma}, \tilde{\sigma}})$.
\end{theorem}

The following function spaces were used:
\begin{definition}\label{Yspace} Let $\sigma\geq0$, $\omega \in \mathbb{R}^2$. The space $Y^{\sigma}_{\omega}(\mathbb{T}^2)$ is defined as the closure of 
$C^{\infty}(\mathbb{T}^2)$ with respect to the norm \begin{equation}\label{Ynorm}\|u_0\|_Y := \|\langle \nabla \rangle^{\sigma}\langle \partial_x \rangle u_0\|_{L^2(\mathbb{T}^2)} + \|\langle \nabla \rangle^{\sigma}\partial_x^{-1} u_0\|_{L^2(\mathbb{T}^2)}.\end{equation}
Further, we define the associated space of quasi-periodic functions
$$Y^{\sigma}_{\omega}:= \{ x \mapsto f(x)=g(\omega x) \ | \ g \in Y^{\sigma}_{\omega}(\mathbb{T}^2)\} \ \text{ equipped with the norm } \|f\|_{Y^{\sigma}_{\omega}}:=\|g\|_{Y^{\sigma}_{\omega}(\mathbb{T}^2)}.$$
\end{definition}

\begin{definition}[Anisotropic Sobolev spaces]\label{anisotropicSobolev}
Let $p \in [1,\infty]$, $\omega \in \mathbb{R}^2$, $s_1,s_2 \geq 0$. The space $H_{\omega}^{s_1,s_2,p}(\mathbb{T}^2)$ is defined as the closure of $C^{\infty}(\mathbb{T}^2)$ with respect to the norm \begin{equation}\label{Hps1s2norm}\|u\|_{H_{\omega}^{s_1,s_2,p}(\mathbb{T}^2)}:=\|\langle \partial_x \rangle^{s_1} u \|_{L^p(\mathbb{T}^2)} + \|\langle \partial_y \rangle^{s_2} u \|_{L^p(\mathbb{T}^2)}.\end{equation}
    Further, we define the associated space of quasi-periodic functions
\begin{equation*}\begin{split}H_{\omega}^{s_1,s_2,p}:=\{ x \mapsto f(x)=g(\omega x)\  | \ g \in H^{s_1,s_2,p}(\mathbb{T}^2)\} \\ \text{ equipped with the norm } \|f\|_{H^{s_1,s_2,p}_{\omega}}:=\|g\|_{H^{s_1,s_2,p}_{\omega}(\mathbb{T}^2)}.\end{split}\end{equation*}
    In the special cases $s_2=0$ and $p=2$, we denote $$V_{\omega}^{s_1,p}:=H_{\omega}^{s_1,0,p} \quad \text{and} \quad H^{s_1,s_2}_{\omega}:=H^{s_1,s_2,2}_{\omega},$$ respectively.
    The associated norms are denoted accordingly.
    When $p=2$ and $s_1=s_2=:s$, the space $H_{\omega}^{s,s,p}$ reduces to the space $H_{\omega}^s$ consisting of functions of the form $x \mapsto f(x):=g(\omega x)$ for $g$ in the ordinary Sobolev space $H^s(\mathbb{T}^2)$.
\end{definition}

Here, $\partial_x^{-1}$ denotes the multiplier operator that outputs the mean-zero primitive with respect to differentiation in the $\omega$-direction for mean-zero functions; that is, the Fourier multiplier with symbol $\frac{1}{i(\omega\cdot n)}1_{n \neq 0}(n)$. The $\omega$-dependence of the space $Y_{\omega}^{\sigma}(\mathbb{T}^2)$ is implicit in $\partial_x^{-1}$ and $\langle \partial_x \rangle$.
We will often suppress sub- and superscripts and write $Y:=Y_{\omega}^{\sigma}$.
Further, denote by $\partial_y$ the derivative in the normal direction to $\omega$, that is, the multiplier operator on $\mathbb{T}^2$ with symbol $i(\omega^{\perp}\cdot n)$, where $\omega^{\perp}:=(\omega_2,-\omega_1)$. Inhomogeneous fractional derivatives $\langle \partial_x \rangle^{s_1}, \langle \partial_y \rangle^{s_2}, \langle \nabla \rangle^{\sigma}$ are defined accordingly. \\

Given that the equation (\ref{BOequation1}) features tangential derivatives $\partial_x$, it is natural to distinguish between tangential and normal derivatives, $\partial_x$ and $\partial_y$, when considering the regularity of a quasi-periodic function. 
This is captured by the notion of anisotropic Sobolev space (Definition \ref{anisotropicSobolev}).

 Provided that $s_1,s_2>\frac{1}{2}$ with $\frac{1}{s_1}+\frac{1}{s_2}<2$, there is an embedding $ H^{s_1,s_2}_{\omega}(\mathbb{T}^2) \hookrightarrow C(\mathbb{T}^2)$, see Lemma \ref{anisotropicproperties}. In particular, the associated quasi-periodic functions $x \mapsto g(x\omega)$ in $H^{s_1,s_2}_{\omega}$ are well-defined pointwise. 
The notion of solution in Theorem \ref{maintheorem} is in the distributional sense: we say that a function $u \in C([0,1],H_{\omega}^{1+\tilde{\sigma},\tilde{\sigma}})$ solves the Benjamin--Ono equation iff
$$\int_{[0,1]} \int_{\mathbb{R}} u\cdot \partial_t \psi \ dxdt = \int_{[0,1]} \int_{\mathbb{R}} Hu \cdot \partial_{xx}\psi + u^2 \cdot \partial_x\psi \ dx dt$$ holds for any $\psi \in C_c^{\infty}([0,1] \times \mathbb{R}).$ \\

For a frequency vector $\omega=(\omega_1,\omega_2)$ for which the ratio of its entries $\alpha:=\omega_1/\omega_2$ has Roth-Liouville irrationality measure $\mu(\alpha)=2$, we can embed $$H^s_{\omega} \hookrightarrow Y^{\sigma}_{\omega} \quad \textit{for} \ s>15/8 \quad \textit{and} \quad 7/8<\sigma<s-1.$$
These embeddings are discussed in Section \ref{Diophantinesection}.
For such frequency vectors, this result thus goes below the threshold in the known local wellposedness theory for the Benjamin--Ono equation in ordinary Sobolev spaces of quasi-periodic functions $H^s_{\omega}$ \cite{AmbroseAitzhan}, see the discussion below on prior work.
Almost all real numbers $\alpha$ have Liouville-Roth irrationality measure $2$ (!) In particular, this applies to all badly approximable numbers (cf. Remark \ref{badlyapproximable}). Combining the above embedding with our main theorem we formulate the following corollary. 

\begin{theorem}\label{maincorollary}
    Suppose $\omega=(\omega_1,\omega_2) \in \mathbb{R}^2$ is a frequency vector for which the ratio of its entries $\alpha:=\omega_1/\omega_2$ has Roth-Liouville irrationality measure $\mu(\alpha)=2$, and suppose $s>15/8$. There exists an absolute constant $\rho_0$ such that for any initial data $u_0 \in B_{H^s_{\omega}}(0,\rho_0)$ there exists a solution to the Benjamin-Ono equation in $C([0,1],H_{\omega}^{1+\tilde{\sigma},\tilde{\sigma}}) \cap L^4([0,1],V^{1,\infty}_{\omega})$ for any $\tilde{\sigma}<s-1$. Further, this solution is unique among limits of smooth solutions and the data-to-solution map is Lipschitz from $B_{H^s_{\omega}}(0,\rho_0)$ equipped with the $L^2_{\omega}$-norm to $C([0,1],L^2_{\omega})$.
\end{theorem}

For a detailed account of the history of work on wellposedness of the Benjamin--Ono equation on the line and the circle, we refer the reader to \cite{MR4743514, MR4400881}. Here, we shall only mention three aspects as relevant to our paper:

First, early results on wellposedness relied on energy method arguments; see for example \cite{MR1044731, MR385355}. Such techniques carry over to the quasi-periodic setting, as was demonstrated recently in \cite{AmbroseAitzhan}. There, it is shown that the Cauchy problem for the Benjamin--Ono equation is locally wellposed in the quasi-periodic Sobolev space $H^s_{\omega}$ for $s>2$ (the space consisting of functions of the form $x \mapsto g(x\omega)$ for $g \in H^s(\mathbb{T}^2)$, with  $\omega \in \mathbb{R}^2$ being a fixed frequency vector, see Definition \ref{anisotropicSobolev}). 
We also mention the paper \cite{zhao2024localwellposednessdispersiveequations}, which carries out an energy method argument for a different space (termed uniform local Sobolev space by the author) which includes some quasi-periodic functions. 

Second, it is known that the data-to-solution map for the Benjamin--Ono equation fails to be uniformly continuous on any neighbourhood of the origin \cite{MR2172940}. This rules out constructing solutions directly via contraction mapping, as such approaches would automatically yield analytic dependence of solutions on data. See also \cite{MR1885293} for earlier results in that direction. 

Third, the paper \cite{MR2052470} introducing a gauge transform was very influential in the study of Benjamin--Ono. For example, before the method of commuting flows, the record for wellposedness on the line was long held by \cite{MR2291918}, which combines the gauge transform with $X^{s,b}$-space techniques. The gauge transform from \cite{MR2052470} is also a key ingredient in this paper. \\

We turn to the Cauchy problem with quasi-periodic data.
Besides the already mentioned works \cite{AmbroseAitzhan, zhao2024localwellposednessdispersiveequations}, the result in \cite{MR4936333} is also applicable to the Benjamin--Ono equation and considers a class of analytic quasi-periodic data. More generally, dispersive equations with quasi-periodic data have recently received increased attention. This is motivated in part by a conjecture made by Deift \cite{MR2411922, MR3622647}. While this conjecture in its strongest interpretation has been disproven \cite{MR4757532}, many questions remain open; the wellposedness theory for quasi-periodic data is not nearly as well understood as in the setting on the real line or on the circle.

For example, on the line or circle, the question of global wellposedness of the KdV equation, which might be the best studied dispersive equation and can serve as a benchmark for progress, is completely settled for initial data in the hierarchy of Sobolev spaces by the method of commuting flows \cite{MR3990604}. This relies of course on the complete integrability structure of KdV. 
For quasi-periodic data, global wellposedness is (to the best of the author's knowledge) only known for initial data in a comparatively restrictive class (which includes quasi-periodic data with small exponentially decaying Fourier coefficients and with a frequency vector satisfying a Diophantine condition) \cite{MR3859361, MR3486173}. The proof relies on spectral properties of the Schr\"odinger operator associated with the initial data.\footnote{For this class of data the Deift conjecture has been answered affirmatively \cite{MR3859361}}.

Given the additional complexity in the quasi-periodic setting, there has been considerable activity in studying completely integrable dispersive equations with quasi-periodic data. Besides the already mentioned works, the KdV equation with quasi-periodic data is also considered in \cite{MR3023416}, where a local result is obtained by $X^{s,b}$-space techniques. The nonlinear Schr\"odinger equation (NLS) and the derivative NLS with quasi-periodic data are addressed in \cite{damanik2024existenceuniquenessasymptoticdynamics, MR3328142, MR4918730, xu2024weaklynonlinearschrodingerequation}, \cite{DamanikLiFeiNLSII, MR4936333}, respectively. Other dispersive equations with quasi-periodic data, including general models, are considered in \cite{MR4751185, MR4668338, MR4936333, zhao2024localwellposednessdispersiveequations}. The method in \cite{AmbroseAitzhan} also applies to more general models. \\

Our result goes, as already mentioned, below the threshold in the known local wellposedness theory for the Benjamin--Ono equation in ordinary Sobolev spaces of quasi-periodic functions $H^s_{\omega}$ \cite{AmbroseAitzhan}. The threshold $s>2$ obtainable by the energy method approach in \cite{AmbroseAitzhan} is dictated by the need to control the quantity $\|\partial_xu\|_{L^{\infty}(\mathbb{T}^2)}$ by $\|u\|_{H^s(\mathbb{T}^2)}$ via Sobolev embedding on the two-dimensional torus.

In this paper, we get an a-priori bound on $\|\partial_xu\|_{L^1([0,1],L^{\infty}_{\omega})}$ on the uniform time interval $[0,1]$ for sufficiently regular solutions --- provided the initial data is small in the low regularity norm $\|\cdot\|_Y$! The predicate `sufficiently regular' is specified by requiring the solution to be a continuous map into a certain space $X=X_{\omega}^{s_1,s_2}$ (for an appropriate choice of $s_1,s_2$), see Definition \ref{definitionX} below.

\begin{restatable}[a-priori bound]{theorem}{aprioriformaltheorem}\label{aprioriformal}
    There exist an absolute constant $\rho_0>0$ such that the following holds: For every $0<\rho<\rho_0$ there exists $M(\rho)>0$ such that for any $0<T\leq 1$, any solution $u\in C([0,T],X)$ to the Benjamin--Ono equation with initial data satisfying $\|u(0)\|_{Y} \leq \rho$ obeys the a-priori estimate \begin{equation}\label{aprioriformalequation}\|u\|_{L^4([0,T],L^{\infty}(\mathbb{T}^2))}+\|\partial_xu\|_{L^4([0,T],L^{\infty}(\mathbb{T}^2))} \leq M(\rho).\end{equation}
\end{restatable}

As already indicated, the gauge transform introduced in \cite{MR2052470} is of importance in this paper. A central step in the paper \cite{MR2052470} about Benjamin-Ono on the line is to establish a bound 
\begin{equation}
\|\partial_x u\|_{L^1([0,1],L^{\infty}(\mathbb{R}))} \lesssim 1
\end{equation}
on a uniform time interval $[0,1]$ for smooth solutions to the Benjamin--Ono equation provided the initial data is small in the low-regularity space $H^1(\mathbb{R})$.
The approach using the gauge transform encounters several difficulties in the quasi-periodic setting compared to the setting on the line, the following two chief among them:
\begin{itemize}
    \item The Strichartz estimate for the Schr\"odinger propagator incurs a derivative loss in the quasi-periodic setting. The following estimate was recently obtained via decoupling \cite{MR4918730}:
    \begin{equation}\label{Strichartz1}
\|e^{it\partial_{xx}}u_0\|_{L^4([0,T],L^4(\mathbb{T}^2))} \lesssim T^{1/8}\|\langle \nabla \rangle^{\kappa}u_0\|_{L^2(\mathbb{T}^2)}, \quad \text{for} \quad \kappa>1/4, \ T \in (0,1]
    \end{equation}
    \item There is no global high-regularity wellposedness theory in the quasi-periodic setting. If we take an approximating sequence of smooth initial data, the existence times of solutions provided by local theories \cite{AmbroseAitzhan, ,MR4936333, zhao2024localwellposednessdispersiveequations} might shrink to zero. 
\end{itemize}

We give a heuristic outline of the proof of Theorem \ref{aprioriformal}, thereby explaining how the first difficulty manifests and how it is overcome. This requires us to touch upon the argument in \cite{MR2052470}.  Afterwards, we address the second difficulty. \\

As mentioned in \cite{MR2052470}, the gauge transform is motivated by the Cole-Hopf transform which can be used to solve Burger's equation $u_t-u_{xx}+\frac{1}{2}(u^2)_x=0, u(0)=u_0$ on the real line. Burger's equation only differs from the Benjamin--Ono equation in the absence of the Hilbert transform. If $u$ solves Burger's equation, then the spatial primitive $v(t,x)=\int_{-\infty}^x u(t,y)dy$ satisfies a second order parabolic PDE with quadratic nonlinearity: $$v_t-v_{xx} + \frac{1}{2}v_x^2=0, \quad v(0,x)=\int_{- \infty}^x u_0(y)dy.$$ Making the Ansatz $w:=\phi(v)$ for some smooth function leads to the equation $$w_t=w_{xx} -[\phi''(v)+\phi'(v)]w_x^2, \quad w(0)=\phi(v(0)).$$
Taking $\phi(v)=\exp(-v)$, this becomes a linear equation: in fact, the wave equation, whose solution can be represented by an explicit formula. Inverting the transformation $\phi$ and differentiating gives a formula for the solution $u$ of Burger's equation. See \cite[Section 4.4.1]{MR1625845}.\\

Now suppose $u$ is a spatially quasi-periodic solution to the Benjamin-Ono equation with zero mean, and denote by $F=\partial_x^{-1}u$ the mean-zero primitive 
with respect to differentiation in the $\omega$-direction. Similar to the Cole-Hopf transform, taking the transformation $\phi$ to be $P_{+hi}(\exp(-i(\cdot)))$ 
leads to the gauge transformed function \begin{equation}w:=\phi(v)=P_{+hi}(e^{-iF})\end{equation} satisfying a simpler equation, see Section \ref{partIbootstrap1}, compare \cite{MR2052470}:
\begin{equation}\label{wequation}
    w_t - iw_{xx}=-2P_{+hi}(P_-(F_{xx})w) -2P_{+hi}(P_-(F_{xx})P_{lo}(e^{-iF})).
\end{equation}
Here, $P_{+hi}$ ($P_-$) denote certain Littlewood-Paley projections to positive high (negative) frequencies with respect to the $\omega$-direction, see Section \ref{sectionnotation}. Note that the assumption about vanishing mean of $u$ is not very restrictive, due to the Galilean symmetry $v(t,x)=u(t,x-ta)-a$ of the equation. 

While far from being linear, the nonlinearity in this equation is much better behaved compared to Benjamin--Ono.
There is no high-low frequency interaction with a derivative falling on the high frequency present anymore.
This is the kind of frequency interaction known to cause failure of uniform continuity of the data-to-solution map for Benjamin--Ono, see \cite{MR2172940}. The second term is low frequency due to all the projections involved and will be regarded as an error term.

While we no longer have an explicit formula for solving the $w$-equation, the equation can be used to gain control over the quantity $\|w_{xx}\|_{L^{4}([0,T],L^{\infty}(\mathbb{T}^2))}$ via bootstrapping as we outline now. This in turn will give us control over the quantity $\|u_x\|_{L^4([0,T], L^{\infty}(\mathbb{T}^2)} = \|F_{xx}\|_{L^4([0,T],L^{\infty}(\mathbb{T}^2))}$ 
appearing in the statement of Theorem \ref{aprioriformal} by `inverting' the gauge transform (Section \ref{partIIbootstrap}). \\

Differentiating equation (\ref{wequation}) twice in the tangential direction, it becomes a semi-linear equation for $q:=w_{xx}$ which schematically looks like a quadratic nonlinear Schr\"odinger equation without derivative in the nonlinearity
\begin{equation}\label{approximateequation}q_t -iq_{xx} \cong q^2.\end{equation}
Indeed, the term $-2P_{+hi}(P_-(F_{xx})P_{lo}(e^{-iF}))$ can be regarded as an error term and is shown to be small in the appropriate sense in Lemma \ref{BestimateforI}. Further, Lemmas \ref{paraproductone} and \ref{paraproducttwo} assert paraproduct estimates which allow us to let the derivatives $\partial_{xx}$ effectively fall on the second factor of $-2P_{+hi}(P_-(F_{xx})w)$. 
Finally, in Section \ref{partIIbootstrap} the gauge transform is `inverted' allowing us to regard the term $P_-(F_{xx}) = \overline{P_+(F_{xx})}$ as behaving similar to $w_{xx}=q$, in the sense that the latter controls the former in suitable function spaces. \\

In the setting on the real line, Strichartz estimates tell us that a solution to the quadratic nonlinear Schr\"odinger equation $q_t + iq_{xx} = q^2$ with small initial data $q_0 = O_{L^2(\mathbb{R})}(\varepsilon^2)$ satisfies 
\begin{align*}\|q\|_{S([0,T]\times\mathbb{R})} & \lesssim \|q_0\|_{L^2(\mathbb{R})}+\|q^2\|_{L^1([0,T],L^2(\mathbb{R}))} \\ & \leq \varepsilon^2+ T^{3/4}\|q\|_{L^4([0,T],L^{\infty}(\mathbb{R}))}\|q\|_{L^{\infty}([0,T],L^2(\mathbb{R}))} \lesssim \varepsilon^2+ T^{3/4}\|q\|_{S([0,T]\times \mathbb{R})}^2.\end{align*}
Here, the Strichartz norm is
$\|q\|_{S([0,T]\times\mathcal{D})}:=\|q\|_{L^{\infty}([0,T],L^2(\mathcal{D}))}+\|q\|_{L^4([0,T],L^{\infty}(\mathcal{D}))},$ for $\mathcal{D}=\mathbb{R}, \mathbb{T}^2$.
This allows one to bootstrap: $\|q\|_{S([0,T]\times \mathbb{R})} \lesssim \varepsilon \ \text{implies} \ \|q\|_{S([0,T]\times \mathbb{R})} \lesssim \varepsilon^2.$ In the quasi-periodic setting, however, the Strichartz estimate (\ref{Strichartz1}) incurs a derivative loss. At best, we can only say that 
\begin{align*}\|q\|_{S([0,T]\times \mathbb{T}^2)} & \lesssim \|q_0\|_{L^2(\mathbb{T}^2)}+\|\langle \nabla \rangle^rq^2\|_{L^1([0,T],L^2(\mathbb{T}^2))}\\ & \lesssim \varepsilon^2+ T^{3/4} \|q\|_{L^4([0,T],L^{\infty}(\mathbb{T}^2))}\|\langle \nabla \rangle^rq\|_{L^{\infty}([0,T],L^2(\mathbb{T}^2))}, \ r>3/4.\end{align*}
This follows from the Strichartz estimate (\ref{Strichartz1}) combined with Sobolev embedding $\|u\|_{L^{\infty}(\mathbb{T}^2)} \lesssim \|\langle \nabla \rangle^hu\|_{L^4(\mathbb{T}^2)}$ for $h>1/2$; 
see Section \ref{partIbootstrap1}. \\

In order to be able to bootstrap, \begin{equation*}\|q\|_{S([0,T]\times \mathbb{T}^2)} \lesssim \varepsilon \implies \|q\|_{S([0,T]\times \mathbb{T}^2)} \lesssim \varepsilon^2,\end{equation*} we will separately establish the bound $\|\langle \nabla \rangle^rq\|_{L^{\infty}([0,T],L^2(\mathbb{T}^2))} \lesssim \varepsilon$ by a Gronwall 
argument using equation (\ref{wequation}).  
This is done in Section \ref{sectionenergyestimatepartI}. \\

Of course, equation (\ref{approximateequation}) is only a schematic approximation. In actuality, one of the terms we need to control after applying the Strichartz estimate to equation (\ref{wequation}) twice differentiated, is
$$\| \partial_{xx}P_{+hi}((\langle \nabla \rangle^rP_-(F_{xx}))w)\|_{L^1([0,T],L^2(\mathbb{T}^2))} + \| \partial_{xx}P_{+hi}(P_-(F_{xx})(\langle \nabla \rangle^rw))\|_{L^1([0,T],L^2(\mathbb{T}^2))}$$
(see \eqref{eq19}, \eqref{eq20}, \eqref{eq21}). For both summands, we want to let the derivatives $\partial_{xx}$ fall on the higher frequency factor $w$ and put the factor which gets hit by $\langle \nabla\rangle^r$ into the norm $\|\cdot\|_{L^{\infty}([0,T],L^2(\mathbb{T}^2))}$ so that we can hope to control this term by a Gronwall argument, while the other factor goes into $\|\cdot\|_{L^1([0,T],L^{\infty}(\mathbb{T}^2))}$ and will be controlled by the boostrap hypothesis. To do so, we will establish two paraproduct estimates, Lemmas \ref{paraproductone} and \ref{paraproducttwo}.
The first paraproduct estimate is similar to a corresponding statement in the setting on the real line \cite[Lemma 3.2]{MR2052470}. The second paraproduct estimate, however, addresses the new challenge of putting the high-frequency factor, on which we want\hspace{-0.1em} the\hspace{-0.1em} derivatives\hspace{-0.1em} $\partial_{xx}$\hspace{-0.1em} to\hspace{-0.1em} fall, into $L^{\infty}(\mathbb{T}^2)$-norm.
This estimate (Lemma \ref{paraproducttwo}) is more delicate,~its proof relying on tools from paradifferential calculus (and transference techniques). See Section \ref{paraproductestimatessection}. \\

We return to the details of Theorem \ref{maintheorem}; 
Given arbitrary initial data in $B_Y(0,\rho_0)$, one would like to approximate by smooth initial data and then use the a-priori estimate on $\|\partial_xu\|_{L^4([0,T],L^{\infty}(\mathbb{T}^2))}$ in Theorem \ref{aprioriformal} together with a Gronwall argument to conclude that the sequence of emanating solutions existing on a uniform time interval is Cauchy in $C([0,1],L^2(\mathbb{T}^2))$.

In the setting on the line or on the circle, the functions $u$ and $w$ for regularized data exist a priori globally due to the high-regularity global wellposedness theory. In our setting, however, the solutions provided by the high regularity wellposedness theory in $H^s_{\omega}$ \cite{AmbroseAitzhan} exist only on a time interval depending on the size of the initial data in the \text{high}-regularity norm $H^{s}_{\omega}$, and the $H^s_{\omega}$-norm might blow up as we approximate the given initial data in $Y$-norm. 
We would like the solutions for regularized data to exist on a uniform time interval. \\

One might hope to overcome this issue by using the bound (\ref{aprioriformalequation}) to control the growth of the $H^{s}_{\omega}$-norm and iterate the local $H^{s}_{\omega}$ wellposedness theory. In analogy to a well-known estimate controlling the growth of the Sobolev norm for solutions to KdV on the line or circle \cite[equation (9.10)]{MR3308874}, there is the estimate \begin{equation}\label{symmetricSobolevnormgrowthcontrol}\|u(t)\|_{H^{s}_{\omega}} \leq \|u(0)\|_{H^s_{\omega}} \cdot \exp \bigg( c \int_0^t \|\nabla u(s)\|_{L^{\infty}} ds \bigg),\end{equation} obtained by a Gronwall argument involving the Kato-Ponce commutator estimate, see Remark \ref{remarksymmetricSobolevnormgrowthcontrol}. However, using this estimate to iterate the $H^{s}_{\omega}$-wellposedness theory would require control over $\|\langle \nabla \rangle u\|_{L^1([0,T],L^{\infty})}$ and not just $\|\partial_xu\|_{L^4([0,T],L^{\infty})}$. As we will see, controlling derivatives $\partial_y$ in the normal direction is hard. \\ 

Thus, for our purposes, the local wellposedness theory in $H^{s}_{\omega}$ is not quite adequate. We are looking for a space defined by a norm that is weak enough so that its growth can be controlled using only `tangential' derivatives $\partial_x$,
yet strong enough so that wellposedness in $X$ provides enough regularity to justify the calculations in the proof of Theorem \ref{aprioriformal}.
We discovered that the following function space (with appropriate choices of parameters $s_1,s_2$) does the job!

\begin{restatable}{definition}{Xspace}\label{definitionX}
Let $\omega \in \mathbb{R}^2$ and $s_1,s_2 \geq 0$. 
Define the space $X_{\omega}^{s_1,s_2}(\mathbb{T}^2)$ to be the closure of $C^{\infty}(\mathbb{T}^2)$ with respect to the norm \begin{equation*}\|u\|_X:= \|[\langle \partial_x \rangle^{s_1} + \langle \partial_y \rangle^{s_2}]u\|_{L^2} + \|\langle \partial_y \rangle^{s_2}\partial_x^{-1}u\|_{L^2}.\end{equation*} 
We define the associated space of quasi-periodic functions $$X=X^{s_1,s_2}_{\omega}:=\{x \mapsto f(x)=g(x\omega)  | \ g \in X_{\omega}^{s_1,s_2}(\mathbb{T}^2)\} \ \text{equipped with the norm} \ \|f\|_X:=\|g\|_{X_{\omega}^{s_1,s_2}(\mathbb{T}^2)}.$$
\end{restatable}
The parameters $\omega \in \mathbb{R}^2$, $s_1,s_2 \geq 0$ will be fixed, so that we often suppress subscripts and simply write $X$.

    \begin{restatable}{theorem}{Xwellposedness}\label{wellposednessinX}
        Let $\omega \in \mathbb{R}^2$, suppose $1/2<s_2<1$ and $s_1$ is sufficiently large depending on $s_2$. 
        The Benjamin--Ono equation is locally wellposed in the space $X=X_{\omega}^{s_1,s_2}$. 
   \end{restatable}

\begin{restatable}{theorem}{Hswellposedness}
\label{wellposednessinHs1s2}
    Let $\omega \in \mathbb{R}^2$, suppose $1/2<s_2<1$ and $s_1$ is sufficiently large depending on $s_2$. 
    The Benjamin--Ono equation is locally wellposed in the space $H^{s_1,s_2}_{\omega}$.
\end{restatable}
    
    The notion of wellposedness here is the classical one due to Hadamard. Given $R>0$ there exist $T:=T(R)$ such that for any initial data $u_0 \in B_{X}(0,R)$ there exists a unique classical solution $u \in C([0,T],X)$ to the Benjamin--Ono equation. 
    Further, the data-to-solution map $B_{X}(0,R) \rightarrow C([0,T],X): u_0 \mapsto u$ is continuous. \\

We explain the motivation leading us to work with the space $X$. 
A key observation is that \begin{equation}\label{crucialestimate}\|\partial_xu\|_{L^{\infty}} \lesssim \|u\|_{H^{s_1,s_2}_{\omega}},\end{equation}
provided $s_2>1/2$ and $s_1>3s_2/(2s_2-1)$, see Lemma \ref{anisotropicproperties}.
This suggests that energy arguments could be employable to establish wellposedness in the space $H^{s_1,s_2}_{\omega}$. This is indeed the case, and we will prove Theorem \ref{wellposednessinHs1s2} by using Galerkin approximation and energy arguments. This is done in Section \ref{wellposednessanisotropicspace}.

Theorem \ref{wellposednessinX} is a variation of this result. The first part of the norm defining the space $X$ is precisely the anisotropic Sobolev norm. The second part of the norm tells us that the function $\partial_x^{-1}u =F$ featured in the gauge transform $w=P_{+hi}(e^{iF})$ also belongs to the anisotropic Sobolev space $H^{s_1,s_2}_{\omega}$. In fact, we can write $$\|u\|_{X_{\omega}^{s_1,s_2}} \sim \|u\|_{H^{s_1,s_2}_{\omega}} + \|\partial_x^{-1}u\|_{H^{s_1,s_2}_{\omega}}.$$
Compared to Theorem \ref{wellposednessinHs1s2}, the proof of Theorem \ref{wellposednessinX} essentially only requires us to obtain additional energy bounds on the second part $\|\langle \partial_y \rangle^{s_2}\partial_x^{-1}u\|_{L^2}$ of the norm. This can be done by using the dynamical equation for $F=\partial_{x}^{-1}u$ obtained from the Benjamin--Ono equation by applying the operator $\partial_x^{-1}$. See Section \ref{sectionwellposednessinX}. \\

We will show that continuity $u:[0,T] \rightarrow X$ provides enough regularity to justify the calculations in the proof of Theorem \ref{aprioriformal}, see Section \ref{sectionaprioriproposition}. 
For initial data satisfying $\|u_0\|_Y \leq \eta$, this a-priori estimate together with the following growth bound on the norm $\|\cdot\|_X$ will allow us to iterate the local wellposedness theory in $X$ on the time interval $[0,1]$.

    \begin{restatable}[growth bound for $X$-norm]{proposition}{Xnormgrowth}\label{normgrowthX}
        The solution provided by the wellposedness theory in $X$ (Theorem \ref{wellposednessinX}) satisfies the bound $$\|u(t)\|_{X} \leq \|u_0\|_{X} \exp \bigg( c  \int_0^t \|u(s)\|_{L^{\infty}} + \|\partial_xu(s)\|_{L^{\infty}} ds \bigg).$$ Here $c$ is some fixed absolute constant.
    \end{restatable}

Theorem \ref{maintheorem} now follows by approximating arbitrary initial data in $B_Y(0,\rho_0)$ by a sequence of data in $X$. Combining Theorems \ref{wellposednessinX}, \ref{aprioriformal}, and Proposition \ref{normgrowthX}, the emanating solutions exist on the uniform time interval $[0,1]$, and they can be shown to form a Cauchy sequence in $C([0,1],H_{\omega}^{1+\tilde{\sigma},\tilde{\sigma}})$ by a Gronwall argument using the a-priori estimate in Theorem \ref{aprioriformal} and interpolation. That the limit belongs to $L^4([0,1],V^{1,\infty}_{\omega})$ follows from weak* compactness of the unit ball of this space and the a-priori estimate (\ref{aprioriformalequation}). Uniqueness and continuity of the data-to-solution map follow by Gronwall arguments using the a-priori bound (\ref{aprioriformalequation}). \\ 

As the main purpose of Theorem \ref{wellposednessinHs1s2} and the related Theorem \ref{wellposednessinX} in this paper is to provide a wellposedness theory at a suitable level of regularity so that the a-priori estimate in Theorem \ref{aprioriformal} can be used for iteration, we do not need to quantify the hypothesis in Theorems \ref{wellposednessinX} and \ref{wellposednessinHs1s2} that $s_1$ is sufficiently large depending on $s_2$. However, we provide the following example of a function in the space $H^{s_1,s_2}_{\omega}$ which does not belong to the space of initial data $H_{\omega}^s, s>2$ considered in the previous work \cite{AmbroseAitzhan} (nor to the space considered in \cite{MR4936333}):
\begin{align*}g(x)=\sum_{k \in \mathbb{Z}^2} \frac{1}{1+(\omega \cdot k)^{2s_1}+(\omega^{\perp} \cdot k)^{2s_2}} e^{2 \pi i x \omega \cdot k}, \text{ with } 1/2<s_2<1 \text{ and } s_1 \text{ sufficiently large}.\end{align*}
The space $H_{\omega}^{s}(\mathbb{T}^2), s>2$ considered in \cite{AmbroseAitzhan} and the space $H^{s_1,s_2}_{\omega}$ appearing in Theorem \ref{wellposednessinHs1s2} overlap, but neither space contains the other.

The space $Y=Y_{\omega}^{\sigma}, \ \sigma>7/8$ appearing in the statement of Theorem \ref{maintheorem}, however, can be taken to be strictly larger than $H_{\omega}^{s}(\mathbb{T}^2), \ s>2$, by making an appropriate choice of $1>\sigma>7/8$, provided the ratio $\alpha=\omega_1/\omega_2$ of the entries of the frequency vector $\omega=(\omega_1,\omega_2)$
has Roth-Liouville irrationality measure $\mu(\alpha)<17/8$.
See Section \ref{Diophantinesection}. The following example is a function that belongs to the space $H^{1+\sigma}_{\omega} \hookrightarrow Y=Y_{\omega}^{\sigma}$, but does not belong to any of the spaces of initial data under consideration in previous work on Benjamin--Ono with quasi-periodic data \cite{AmbroseAitzhan, MR4936333, zhao2024localwellposednessdispersiveequations}: 
$$g(x)= \sum_{N \in 2^{\mathbb{N}}} N^{-1-\sigma-\varepsilon} e^{2\pi i x \nu N}, \  \varepsilon >0, \ 7/8<\sigma<1, \ \nu \in \omega_1 \mathbb{Z} + \omega_2 \mathbb{Z}: |\nu| \sim 1.$$

\begin{remark} Just like the classical energy method, the energy method in anisotropic Sobolev spaces $H^{s_1,s_2}_{\omega}$ does not use the specific dispersion relation. It applies just as well to other dispersive equations, including the Korteweg--de Vries equation (KdV) and the derivative nonlinear Schr\"odinger equation (dNLS)
\begin{equation*}
    \partial_tu+\partial_{x}^3u = u\partial_xu \ \ \text{(KdV)}, \quad \quad \partial_tu-i\partial_{xx}u=u\partial_xu \ \ \text{(dNLS)}.
\end{equation*}
Only the anti-selfadjointness of the linear dispersive parts, $\partial_{x}^3$ and $-i\partial_{x}^2$, respectively, is used.
\begin{corollary}\label{KdVandNLS}
    The KdV equation and the dNLS equation are locally wellposed in $H^{s_1,s_2}_{\omega}$ for $1/2<s_2<1$ and $s_1$ sufficiently large depending on $s_2$.
\end{corollary}
\end{remark}

\subsection{Organization of paper}
In Section \ref{sectionnotation}, notations are fixed. In Section \ref{preliminariessection}, we prove some preliminary results, 
including estimate (\ref{crucialestimate}) and commutator estimates for tangential and normal inhomogeneous derivatives (Lemma \ref{KatoPonceVegaQuasih}). 

In Section \ref{wellposednessanisotropicspace}, we prove both Theorems \ref{wellposednessinX} and \ref{wellposednessinHs1s2}. We first prove Theorem \ref{wellposednessinHs1s2} in Sections \ref{regularizedBO}-\ref{sectioncontinuitydatatosolution} by Galerkin approximation and energy method arguments; see the beginning of Section \ref{wellposednessanisotropicspace} for an outline of the proof. In Section \ref{sectionwellposednessinX}, the necessary adaptations to prove Theorem \ref{wellposednessinX} are discussed. In Section \ref{sectionnormgrowthcontrol} we obtain Proposition \ref{normgrowthX} essentially as a corollary of the previously presented proof.

In Section \ref{sectionaprioriproposition}, we prove Theorem \ref{aprioriformal}.
A heuristic outline of the proof was already given in the introduction. We refer the reader to Section \ref{sectionaprioriproposition} for the structure of the proof as presented in Sections \ref{setup}-\ref{sectionerrorterms} of this paper.
In Section \ref{sectionfinisihingtheproof} we put everything together and prove Theorem \ref{maintheorem}.

In Section \ref{Diophantinesection}, we discuss Diophantine conditions allowing us to embed Sobolev spaces of quasi-periodic functions $H^{s}_{\omega}$ into the space $Y=Y_{\omega}^{\sigma}$ appearing in our main theorem. We consider badly approximable numbers and Roth-Liouville irrationality measure. As a corollary of our main theorem these considerations yield Theorem \ref{maincorollary}.

\subsection{Acknowledgements} The author would like to thank Rowan Killip and Monica Vi\c{s}an for many helpful discussions and for providing feedback on earlier versions of the paper.

\section{notation and basic properties}\label{sectionnotation}
We use the standard notation $A\lesssim B$ or $A=O(B)$ to indicate that there exists an absolute constant $C$ such that $A\leq CB$. The constant can change from line to line. If the constant fails to be universal, then possible dependencies are indicated by subscripts. We write $A \sim B$ if both $A \lesssim B$ and $B \lesssim A$ hold. We write $f(x)=o(g(x))$ if $f(x)/g(x)\rightarrow 0$ as $x\rightarrow \infty$. For example, $o(1)$ is a placeholder for a sequence converging to zero. 
We denote by $\cong_{\mathbb{R}}$ isomorphism of vector spaces over real numbers. \\

Given an interval $I=[0,T], \ T>0$ and a function $u:[0,T] \times \mathbb{T}^2 \rightarrow \mathbb{C}$ we use the following notation for mixed Lebesgue spacetime norms
$$\|u\|_{L^p([0,T],L^q(\mathbb{T}^2))}:=\big\|\|u(t,x)\|_{L_x^q(\mathbb{T}^2)} \big\|_{L^p_t([0,T])} = \bigg[ \int_{[0,T]} \Big(\int_{\mathbb{T}^2} |u(t,x)|^q dx \Big)^{p/q} dt\bigg]^{1/p}.$$ 
When writing $\|\cdot\|_{L^p([0,T],L^q)}$ the domain for the spatial variables is understood to be the torus $\mathbb{T}^2$.
We often abbreviate this further to $\|\cdot\|_{L^pL^q}$ when the time interval $[0,T]$ is clear from context. We do so throughout Sections \ref{partIbootstrap3}, \ref{partIIbootstrap}, \ref{sectionerrorterms}.

Further, we denote the Strichartz norm
\begin{equation*}
\|u\|_{S([0,T])}:=\|u\|_{L^4([0,T],L^{\infty})}+\|u\|_{L^{\infty}([0,T],L^2)}.
\end{equation*} 
Regarding ordinary Lebesgue spaces, the domain is understood to be $\mathbb{T}^2$ when writing $\|\cdot\|_{L^q}$, unless otherwise indicated. \\

We denote by $C([0,T],X)$ the space of continuous functions on the interval $[0,T]$ into the metric space $X$, equipped with the norm $\|u\|_{C([0,T],X)}:= \sup_{t \in [0,T]}\|u(t)\|_X$. For a function $u$ of space and time we use the  hat-notation $\hat{u}$ to denote the Fourier transform in the spatial variables only. \\

Our convention for the Fourier transform is as follows: The Fourier transform $\hat{u}:\mathbb{Z}^2 \rightarrow \mathbb{C}$ of a function $u:\mathbb{T}^2 \rightarrow \mathbb{C}$ is $$\hat{u}(n_1,n_2):= \kern-2pt \int_{\mathbb{T}^2}\kern-2pt u(x_1,x_2)e^{-2\pi i(n_1x_1+n_2x_2)}dx_1dx_2, \  \text{so} \ u(x_1,x_2)= \kern-2pt \sum_{n \in \mathbb{Z}^2} \hat{u}(n_1,n_2)e^{2\pi i (n_1x_1+n_2x_2)}.$$
With this convention, the Fourier transform is a unitary operator $L^2(\mathbb{T}^2) \rightarrow l^2(\mathbb{Z}^2)$. \\

We define directional Littlewood-Paley projections.
Fix a frequency vector $\omega\in \mathbb{R}^2$. Let $\psi_{+hi}$ be a smooth function taking values in $[0,1]$ with $\psi_{+hi}(x)=1$ for $x\geq 2$ and $\psi_{+hi}(x)=0$ for $x \leq 1/2$. Further, define
\begin{itemize}
    \item $\psi_{-hi}(x):=\psi_{+hi}(-x)$, $\psi_{lo}(x):=1-\psi_{+hi}(x) - \psi_{-hi}(x)$
    \item $\psi_{+HI}(x):=\psi_{+hi}(x/4)$, $\psi_{-HI}(x):=\psi_{-hi}(x/4)$, and $\psi_{LO}(x):=\psi_{lo}(x/4)$
\end{itemize}
We define the Littlewood-Paley projection to positive high frequencies in the $\omega$-direction as the multiplier operator on the two dimensional torus $\mathbb{T}^2$ with symbol $\psi_{+hi}(\omega \cdot n)$. That is, $$\widehat{P_{+hi}f}(n)=\psi_{+hi}(\omega \cdot n)\hat{f}(n).$$
Analogously, we define the Littlewood-Paley projections $P_{lo},P_{-hi},P_{+HI},P_{LO},P_{-HI}$ as multiplier operators with symbols $\psi_{\Box}(\omega \cdot n)$ for $\Box=-hi,lo,+HI,LO,-HI$, respectively. Note that
$$1=P_{+hi}+P_{lo}+P_{-hi} \quad \text{and} \quad 1=P_{+HI}+P_{LO}+P_{-HI}.$$
Further, $P_+$ and $P_-$ denote the multiplier operators with symbols $1_{\{\omega \cdot n > 0\}}(n)$ and $1_{\{\omega \cdot n < 0\}}(n)$. The quasi-periodic Hilbert transform in the $\omega$-direction can be expressed as \begin{equation}\label{Hilberttransform}H=-i[P_+-P_-].\end{equation}
The following lemma captures the basic boundedness properties on $L^p$-spaces of these Fourier projections.

\begin{lemma}
    The directional Littlewood-Paley projections $P_{+hi},P_{lo},P_{-hi},P_{+HI},P_{LO},P_{-HI}$ as well as $H,P_+,P_-$ are bounded on $L^p(\mathbb{T}^2)$ for $1<p<\infty$. Further, $P_{lo}$ and $P_{LO}$ are also bounded on $L^{\infty}$.
\end{lemma}

This can be established easily using the transference principle stated below. The associated boundedness statement for the corresponding Fourier projections on the plane follow from classical Littlewood-Paley theory and a rotational dilation of the plane. Using transference circumvents the lack of rotational and dilational invariance of the torus.

\begin{proposition}[De Leeuw]\label{deLeeuw}\cite[Chapter VII: Theorem 3.8]{MR304972}
    Let $m(\xi)$ be an $L^{\infty}(\mathbb{R}^2)$-function. Let $T$ be the associated multiplier operator defined by $$\widehat{Tf}(\xi)=m(\xi)\hat{f}(\xi), \ \xi \in \mathbb{R}^2, \ \text{ for functions $f$ on the plane $\mathbb{R}^2$},$$ and let $\tilde{T}$ be its periodization defined by $$\widehat{\tilde{T}f}(n)=m(n)\hat{f}(n), \ n \in \mathbb{Z}^2, \ \text{for functions $f$ on the torus $\mathbb{R}^2 / \mathbb{T}^2$}.$$ If every point of the lattice $\mathbb{Z}^2$ is a Lebesgue point of $m(\xi)$ and if $T$ is bounded on $L^p(\mathbb{R}^2)$, then $\tilde{T}$ is bounded on $L^p(\mathbb{T}^2)$.
\end{proposition}

Further, we denote the following Fourier projections to finitely many frequencies: \begin{equation}\label{generalprojection}P_{q,r}u(x)=\sum_{k \in \mathbb{Z}^2: |\omega\cdot k|\leq q, |\omega^{\perp} \cdot k| \leq r} \hat{u}(n) e^{ik\cdot x}, \ \  q,r \in \mathbb{R}_+, \quad \text{and} \quad P_q:=P_{q,q}.\end{equation} 
For the definition of the multiplier operators $\partial_x$, $\langle \partial_x \rangle^{s_1}$, $\partial_x^{-1}$, $\partial_y$, $\langle \partial_y \rangle^{s_2}$, $\langle \nabla \rangle^{\sigma}$, see the discussion following the statement of Theorem \ref{maintheorem} and the discussion of quasi-periodic functions in the third paragraph of the introduction. In the quasi-periodic setting, the Schr\"odinger propagator $e^{it\partial_{xx}}$ is defined as the multiplier operator with symbol $e^{-it(\omega \cdot n)^2}$.

\section{preliminaries}\label{preliminariessection}
We record some properties of anisotropic Sobolev spaces (Definition \ref{anisotropicSobolev} with $p=2$) that we will use in the body of this paper. We establish commutator estimates for quasi-periodic functions (Lemma \ref{KatoPonceVegaQuasih}) that we will need in Section \ref{wellposednessanisotropicspace}. This requires discussing transference techniques for bilinear multiplier operators.

\subsection{Properties of anisotropic Sobolev spaces}
\begin{lemma}[Properties of anisotropic Sobolev spaces]\label{anisotropicproperties}
    We have the following: \\

\noindent $\bullet$ Anisotropic Sobolev embedding: For $s_1,s_2 > \frac{1}{2}$ with $\frac{1}{s_1} + \frac{1}{s_2} <2$, we have \begin{equation}\label{anisotropicembedding} \|u\|_{L^{\infty}} \lesssim \|u\|_{H^{s_1,s_2}_{\omega}}.\end{equation}

\noindent $\bullet$ Algebra property: For $s_1,s_2 > \frac{1}{2}$ satisfying $\frac{1}{s_1} + \frac{1}{s_2} <2$ we have
\begin{equation}\label{algebraproperty}
        \|uv\|_{H^{s_1,s_2}_{\omega}}\lesssim \|u\|_{H^{s_1,s_2}_{\omega}}\|v\|_{L^{\infty}} + \|u\|_{L^{\infty}} \|v\|_{H^{s_1,s_2}_{\omega}} \lesssim \|u\|_{H^{s_1,s_2}_{\omega}} \|v\|_{H^{s_1,s_2}_{\omega}}. \end{equation}
Rescaling the norm appropriately, $H_{\omega}^{s_1,s_2}$ becomes a Banach algebra. \\

\noindent $\bullet$ Crucial estimate on $\partial_xu$: Suppose $s_2>\frac{1}{2}$ and $s_1>\frac{3s_2}{2s_2-1}$. We have \begin{equation}\label{crucialequation}\|\partial_x u\|_{L^{\infty}} \lesssim \|u\|_{H^{s_1,s_2}_{\omega}}.\end{equation}

\noindent $\bullet$ Interpolation: Suppose $0<\tilde{s_1}<s_1$, $0<\tilde{s_2}<s_2$ and $\theta \in (0,1)$ with $\tilde{s_1} \leq \theta s_1$ and $\tilde{s_2} \leq \theta s_2$. Then \begin{equation}\label{interpolationestimate}\|u\|_{H^{\tilde{s_1},\tilde{s_2}}_{\omega}} \leq \|u\|_{H^{s_1,s_2}_{\omega}}^{\theta} \cdot \|u\|_{L^2}^{1-\theta}.\end{equation} In particular, this applies to $\theta = \max (\frac{\tilde{s_1}}{s_1}, \frac{\tilde{s_2}}{s_2})$. \\

\noindent $\bullet$ Boundedness properties of derivative operators: The following maps are bounded linear:
\begin{equation}\label{boundedness1} \partial_{x}^k: H^{s_1,s_2}_{\omega} \rightarrow H^{s_1-k,\frac{s_1-k}{s_1}s_2}_{\omega} \text{  for  } k<s_1 \end{equation}
\begin{equation}\label{boundedness2} \langle \nabla \rangle^{\eta}:H^{s_1,s_2}_{\omega} \rightarrow H_{\omega}^{s_1-\eta \frac{s_1}{s_2}, s_2-\eta} \text{  for  } 0 \leq \eta \leq s_2 \leq s_1. \end{equation}
\end{lemma}

\begin{proof}
To prove estimate \eqref{anisotropicembedding}, it suffices to 
show that $\big\|\frac{1}{1+|\omega\cdot n|^{s_1} + | \omega^{\perp} \cdot n|^{s_2}}\big\|_{l^2_n(\mathbb{Z}^2)}$ is finite.
The integral over $\mathbb{R}^2$ and the sum over the lattice $\mathbb{Z}^2$ of the function $f: \mathbb{R}^2 \rightarrow \mathbb{R}_{+}$ taking values 
\begin{equation*}f(\xi_1,\xi_2) =\frac{1}{1+ |\omega \cdot (\xi_1,\xi_2)|^{2s_1} + |\omega^{\perp} \cdot (\xi_1,\xi_2)|^{2s_2}} \sim \frac{1}{[1+|\omega \cdot (\xi_1,\xi_2)|^{s_1} + |\omega^{\perp} \cdot (\xi_1,\xi_2)|^{s_2}]^2}\end{equation*} 
are comparable:
\begin{equation*}\sum_{(n_1,n_2) \in \mathbb{Z}^2} f(n_1,n_2) \sim \int_{\mathbb{R}^2}f(\xi_1,\xi_2) \ d\xi_1 d\xi_2.\end{equation*}
Indeed, the function $f$ is slowly varying on rectangles in the sense that for any rectangle $Q$ centered at the origin it holds that
$$\forall (x,y) \in \mathbb{R}^2 \ \forall (a,b)\in Q: f(x+a, y+b) \sim_Q f(x,y).$$
After applying a rotational dilation of the plane, we conclude by invoking the following fact (which can be established by calculation): Suppose $s,t>1$. Then \begin{equation}\label{helpfulcalculation} \int_0^{\infty} \int_0^{\infty} \frac{1}{1+x^s+y^t} dxdy < \infty \quad \iff \quad \frac{1}{s} + \frac{1}{t} < 1. \end{equation} 

The algebra property (\ref{algebraproperty}) follows from (\ref{anisotropicembedding}) together with Plancherel and subadditivity (up to a constant factor) of $n \mapsto \langle \omega \cdot n \rangle^{2s_1}$ and $n \mapsto \langle \omega^{\perp} \cdot n \rangle^{2s_2}$. \\

By the same reasoning as for (\ref{anisotropicembedding}), estimate (\ref{crucialequation}) amounts to $$\int_{\mathbb{R}} \int_{\mathbb{R}} \frac{|\xi_1|^2}{1+|\xi_1|^{2s_1} + |\xi_2|^{2s_2}} d\xi_1 d\xi_2 < \infty.$$ But the latter follows again from (\ref{helpfulcalculation}) after a change of variable $\xi_1'=\xi_1^3$, provided the hypotheses $s_1>\frac{3s_2}{2s_2-1}$ and $s_2>\frac{1}{2}$ are satisfied.

The interpolation inequality (\ref{interpolationestimate}) follows by Plancherel and H\"older just as for standard Sobolev spaces. \\

The boundedness of the derivative operators \eqref{boundedness1} and \eqref{boundedness2} can be established using Young's inequality for products (and Plancherel). 
For example, boundedness of \eqref{boundedness1} follows from $$\xi_1^k(\langle \xi_1 \rangle^{s_1-k}+\langle \xi_2 \rangle^{\frac{s_1-k}{s_1}s_2}) \lesssim \langle \xi_1 \rangle^{s_1} + \langle \xi_2 \rangle^{s_2}.$$
Boundedness of \eqref{boundedness2} can be established similarly.
\end{proof}
Controlling the quantity $\|\partial_x u\|_{L^{\infty}}$ appearing in \eqref{crucialequation} is central to the energy method for dispersive equations, as carried out in Section \ref{wellposednessanisotropicspace} for anisotropic Sobolev spaces.

\subsection{Tools for dispersive PDE in the quasi-periodic setting: commutator estimates}

We prove two commutator estimates for fractional inhomogeneous tangential and normal derivatives. The estimates will play a similar role in Section \ref{wellposednessanisotropicspace} as the familiar Kato-Ponce commutator estimate does in the setting on the line. See the discussion following Definition \ref{anisotropicSobolev} for the definition of directional fractional inhomogenous derivatives $\langle \partial_x \rangle^{s_1}$ and $\langle \partial_y \rangle^{s_2}$. 

In the following, we call the bilinear operator defined by \begin{equation}\label{paraproduct}\pi(f,g)(x):=\int_{\mathbb{R}^2} \int_{\mathbb{R}^2} m(\xi_1,\xi_2)\hat{f}(\xi_1) \hat{g}(\xi_2)e^{2\pi i(\xi_1+\xi_2).x}d\xi_1 d\xi_2\end{equation} the paraproduct (or bilinear multiplier operator) $\pi$ associated to the symbol $m(\xi_1,\xi_2)$. 
A classical result by Coifman and Meyer \cite{MR511821, CoifmanMeyerAlternative} says that if the symbol $m$ is smooth and satisfies the bound
\begin{equation}\label{CoifmanMeyer}
|\partial_{\xi_1}^{\alpha} \partial_{\xi_2}^{\beta}m(\xi_1,\xi_2)| \lesssim_{\alpha,\beta} (|\xi_1|+|\xi_2|)^{-|\alpha|-|\beta|} \quad \text{for every} \ \alpha, \beta \in \mathbb{N}^2,
\end{equation}
then $\pi$ extends from Schwartz functions to a bounded bilinear map $L^{p_1}(\mathbb{R}^2) \times L^{p_2}(\mathbb{R}^2) \rightarrow L^p(\mathbb{R}^2)$ for $1<p_1,p_2 \leq \infty$ with $\frac{1}{p_1}+\frac{1}{p_2}=\frac{1}{p}$ and $0<p<\infty$. Such symbols are called Coifman-Meyer symbols. 

While the Coifman-Meyer Theorem features prominently in the original proof of the Kato-Ponce commutator estimate \cite{MR0951744}, we obtain the commutator 
estimates in Lemma \ref{KatoPonceVegaQuasih} by transference arguments. However, we still need the Coifman-Meyer theorem in proving the paraproduct estimates in Lemma \ref{paraproducttwo}. \\

Given a paraproduct $\pi$ of the form (\ref{paraproduct}), its periodization $\tilde{\pi}$ is defined by taking the restriction of its symbol $m(\xi_1,\xi_2)$ to the integer lattice $\mathbb{Z}^2$ as the symbol of $\tilde{\pi}$. That is, for functions $f,g:\mathbb{T}^2 \rightarrow \mathbb{C}$ define \begin{equation}\label{periodization}\tilde{\pi}(f,g)(x):= \sum_{n_1,n_2 \in \mathbb{Z}^2} \pi(n_1,n_2)\hat{f}(n_1) \hat{g}(n_2)\exp(2\pi i x.(n_1+n_2)).\end{equation}
In the proof of one of the commutator estimates we use the following transference principle:

\begin{proposition}\cite[Theorem 3]{MR1808390}\label{transferenceparaproduct}
    Suppose $m(\xi_1,\xi_2)$ is a $L^{\infty}(\mathbb{R}^2\times \mathbb{R}^2)$-function and let $\pi$ be the paraproduct defined by formula (\ref{paraproduct}).
    If the paraproduct $\pi$ satisfies the bound $$\|\pi(f,g)\|_{L^p(\mathbb{R}^2)} \lesssim \|f\|_{L^{p_1}(\mathbb{R}^2)}\|g\|_{L^{p_2}(\mathbb{R}^2)}$$ for some $p,p_1,p_2 \in [1,\infty]$ with $\frac{1}{p}=\frac{1}{p_1}+\frac{1}{p_2}$ and each point of the lattice $\mathbb{Z}^2$ is a Lebesgue point of $m$, then the periodization $\tilde{\pi}$ also satisfies the bound $$\|\tilde{\pi}(f,g)\|_{L^p(\mathbb{T}^2)} \lesssim \|f\|_{L^{p_1}(\mathbb{T}^2)}\|g\|_{L^{p_2}(\mathbb{T}^2)}.$$
\end{proposition}

\begin{lemma}[commutator estimates]\label{KatoPonceVegaQuasih} We have the following two estimates:

Quasi-periodic Kato-Ponce commutator estimate: Suppose $s_1\geq1$ and $2 \leq p<\infty$. Then \begin{equation}\label{commutatorsderivatives}\|[\langle \partial_x \rangle^{s_1},f]g\|_{L^p(\mathbb{T}^2)} \lesssim \|\langle \partial_x \rangle^{s_1}f\|_{L^p(\mathbb{T}^2)} \|g\|_{L^{\infty}(\mathbb{T}^2)} + \|\partial_xf\|_{L^{\infty}(\mathbb{T}^2)}\|\langle \partial_x \rangle^{s_1-1}g\|_{L^p(\mathbb{T}^2)}.\end{equation}

Quasi-periodic Kenig-Ponce-Vega commutator estimate: Suppose $0 \leq s_2 \leq 1$, and $p \in [1, \infty]$. Then \begin{equation}\label{KenigPonceVegaQuasi}\|[\langle \partial_y \rangle^{s_2},f]g\|_{L^p(\mathbb{T}^2)} \lesssim \|\langle \partial_y \rangle^{s_2}f\|_{L^p(\mathbb{T}^2)} \|g\|_{L^{\infty}(\mathbb{T}^2)}.\end{equation}
\end{lemma}

\begin{proof}[Proof of (\ref{KenigPonceVegaQuasi})]
    Estimate (\ref{KenigPonceVegaQuasi}) can be reformulated as $L^p(\mathbb{T}^2) \times L^{\infty}(\mathbb{T}^2) \rightarrow L^{p}(\mathbb{T}^2)$ boundedness of the bilinear multiplier operator with symbol $$[\langle \omega^{\perp} \cdot (n_1+n_2) \rangle^{s_2} - \langle \omega^{\perp} \cdot n_2 \rangle^{s_2}]\langle \omega^{\perp} \cdot n_1\rangle^{-s_2}.$$
    In this form, Proposition \ref{transferenceparaproduct} 
    is applicable. The corresponding estimate on the plane $\mathbb{R}^2$ follows from the classical Kenig-Ponce-Vega estimate on the line \cite[Theorem 1.1]{MR3914540} and a rotational dilation of the plane. 
\end{proof}

For estimate (\ref{commutatorsderivatives}), Proposition \ref{transferenceparaproduct} is not applicable. However, we can achieve transference by adapting the strategy of the proof of De Leeuw's theorem (Lemma \ref{deLeeuw}) in \cite[Section VII: Theorem 3.8]{MR304972}. 
We begin by proving two lemmas, the first one (Lemma \ref{PQR}) being in analogy with \cite[Section VII: Lemma 3.11]{MR304972}. We then conclude estimate (\ref{commutatorsderivatives}).

   \begin{lemma}\label{PQR}
        Let $P,Q$,$R$ be trigonometric polynomials. Suppose $\pi$ is a paraproduct whose symbol $m(\xi_1,\xi_2)$ is continuous at all points of the integer lattice $\mathbb{Z}^2$ and further grows at most exponentially $m(\xi_1,\xi_2)=O(\exp |(\xi_1,\xi_2)|).$ 
        Let $\tilde{\pi}$ be its periodization defined by (\ref{periodization}), and denote for $\delta>0$ the Gaussian $w_{\delta}(y) = e^{-\pi \delta |y|^2}$. Whenever $\alpha, \beta, \gamma >0$ with $\alpha+\beta+\gamma=1$ we have that \begin{equation}\label{PQRequation}\lim_{\varepsilon \rightarrow 0} \varepsilon^{} \int_{\mathbb{R}^2} \pi(Pw_{\alpha \varepsilon},Qw_{\beta \varepsilon})(x)\overline{R(x)}w_{\gamma \varepsilon}(x)dx = \int_{[-1/2, 1/2]^2} \tilde{\pi}(P,Q)(x)\overline{R(x)}dx.\end{equation}
    \end{lemma}

\begin{proof}[Proof of Lemma \ref{PQR}]
    By linearity, it suffices to show the statement when $P,Q,R$ are characters; so assume $P=e^{2\pi i j \cdot x }, Q = e^{2\pi i k \cdot x}, R = e^{2\pi i l \cdot x}$. Observe that the right hand side in (\ref{PQRequation}) is 
\begin{equation}
\left\{
    \begin{array}{lr} 
      \ \ \ \ 0 & j + k \neq l \\
      m(j,k) & j+k =l
   \end{array}
\right\}.
\end{equation}

By Plancherel and the formula for the Fourier transform of a Gaussian, the expression inside the limit on the left hand side of (\ref{PQRequation}) is
  \begin{align}\label{expression1}\varepsilon^{} \int_{\mathbb{R}^2} \int_{\mathbb{R}^2} m(\xi_1, \xi_2) e^{-\pi|\xi_1 -m|^2/\alpha \varepsilon} (\alpha \varepsilon)^{-1} e^{- \pi |\xi_2-l|^2/ \varepsilon^2} (\beta\varepsilon)^{-1}  e^{-\pi|\xi_1 + \xi_2 -k|^2/\gamma \varepsilon}(\gamma \varepsilon)^{-1}d\xi_1 d\xi_2.
\end{align}
Heuristically, if the centers of the three Gaussians coincide in the sense that the planes $\xi_1=j, \xi_2=k, \xi_1+\xi_2=l$ intersect in a point, then the product of the three Gaussians together with the combined prefactor $(\varepsilon\alpha\beta\gamma)^{-1}$ acts as a Dirac delta as $\varepsilon \rightarrow 0$. Rigorously, if $j+k=l$, then (\ref{expression1}) can be expressed as
\begin{align}\label{expression2}
    \int_{\mathbb{R}^2} \int_{\mathbb{R}^2} m(\xi_1, \xi_2) \cdot (\varepsilon^2 \alpha \beta \gamma)^{-1} G_{\varepsilon}(\xi_1-j, \xi_2-k) \ d\xi_1 d\xi_2 \end{align}
where we denoted 
\begin{equation*}G_{\varepsilon}(\xi_1,\xi_2):=\exp(-(\pi/\varepsilon) \cdot  (\xi_1,\xi_2)A_{\alpha,\beta,\gamma}(\xi_1,\xi_2)^t)\end{equation*} with \begin{equation*}A_{\alpha, \beta, \gamma} :=   \begin{bmatrix}
    (1/\alpha +1/\gamma)Id_2       & (1/\gamma) Id_2 \\
    (1/{\gamma})Id_2       & (1/\gamma + 1/\beta)Id_2 \\
\end{bmatrix},\end{equation*}
and the kernel $(\varepsilon^2 \alpha \beta \gamma)^{-1} G_{\varepsilon}(\xi_1-j, \xi_2-k)$ is an approximate identity translated to the point $(j,k)$. In fact, it is a generalized Gauss-Weierstrass kernel due to $\det(A_{\alpha,\beta,\gamma})=(\alpha\beta\gamma)^{-2}$. The calculation of the determinant uses the hypothesis that $\alpha+\beta+\gamma=1$. 
The assumption that the symbol $m$ is continuous at $(j,k)$ and grows at most exponentially, provides enough regularity to guarantee that (\ref{expression2}), the integral of $m$ against the Gauss-Weierstrass kernel, converges to $m(j,k)$ as $\varepsilon \rightarrow 0$.  \\

Next, consider the case $j+k\neq l$. The kernel against which $m$ is integrated in (\ref{expression1}) is the four-dimensional Gaussian $$\tilde{K}_{\varepsilon}(\xi_1,\xi_2)=(\varepsilon^2\alpha\beta\gamma)^{-1}\exp(-({\pi}/{\varepsilon})Q(\xi_1,\xi_2))= \exp(-\pi Q_{\min} / \varepsilon) \cdot (\varepsilon^2\alpha\beta\gamma)^{-1} \cdot  G_{\varepsilon}(\xi_1-\xi_{1c},\xi_2-\xi_{2c}),$$
where we denoted the positive definite quadratic form $$Q(\xi_1,\xi_2):=|\xi_1-j|^2/\alpha + |\xi_2-k|^2/\beta + |\xi_1+\xi_2-l|^2/{\gamma},$$
its minumin value and the coordinates where the minumum is attained being
$$Q_{min}:=\frac{|j+k-l|^2}{\alpha+\beta+\gamma}, \quad \xi_{1c}=\frac{(\beta+\gamma)j+\alpha(l-k)}{\alpha+\beta+\gamma} \quad \text{and} \quad \xi_{2c}=\frac{(\alpha+\gamma)k+\beta(l-j)}{\alpha+\beta+\gamma}.$$
The expression inside the limit on the left hand side of \eqref{PQRequation} can thus be written as
\begin{align}\label{expression3}
\exp(-\pi Q_{\min} /\varepsilon) \cdot (\varepsilon^2 \alpha \beta\gamma)^{-1} \cdot \int_{\mathbb{R}^2} \int_{\mathbb{R}^2} m(\xi_1,\xi_2) G_{\varepsilon}(\xi_1-\xi_{1c}, \xi_2 - \xi_{2c}) d\xi_1 d\xi_2
\end{align}
and is bounded by
$$\exp(-\pi Q_{\min}/\varepsilon) \sup_{\xi_1,\xi_2} \big| m(\xi_1,\xi_2)G_{\varepsilon}(\xi_1-\xi_{1c}, \xi_2-\xi_{2c})^c \big| \cdot (\varepsilon^2 \alpha \beta\gamma)^{-1} \cdot \kern-0.1em  \int_{\mathbb{R}^4}  G_{\varepsilon}(\xi_1-\xi_{1c}, \xi_2 - \xi_{2c})^{1-c} d\xi_1 d\xi_2,$$
with $0<c<1$ a fixed constant. Due to the symbol $m(\xi_1,\xi_2)$ growing at most exponentially, the expression $ m(\xi_1,\xi_2)G_{\varepsilon}(\xi_1-\xi_{1c}, \xi_2-\xi_{2c})^c $ is bounded in $\xi_1,\xi_2$, and the bound is uniform as $\varepsilon \rightarrow 0$. The last factor is the integral of a Gaussian:
$$\int_{\mathbb{R}^2}\int_{\mathbb{R}^2} G_{\varepsilon}(\xi_1-\xi_{1c},\xi_2-\xi_{2c})^{1-c} d\xi_1 d\xi_2 = \pi^2 \det((1-c)(\pi/\varepsilon)A_{\alpha,\beta,\gamma})^{-1/2} = \varepsilon^2 (1-c)^{-2} \det(A_{\alpha,\beta,\gamma})^{-1/2}.$$
If $j+k \neq l$, then $Q_{\min}>0$. Therefore, 
$$\varepsilon^{} \int_{\mathbb{R}^2}\pi(Pw_{\alpha \varepsilon}, Qw_{\beta\varepsilon})\overline{Rw_{\gamma \varepsilon}} \lesssim \exp(-\pi Q_{\min}/\varepsilon) \rightarrow 0 \quad \text{as} \ \varepsilon \rightarrow 0$$
whenever $j+k \neq l$. Note that if $j+k=l$, then $Q_{\min}=0$, $\xi_{1c}=j, \xi_{2c}=k$, and equation (\ref{expression3}) reduces to expression (\ref{expression2}).
\end{proof}

\begin{remark}
    The assumptions on the symbol $m$ in Lemma \ref{PQR} can be weakened. It suffices to assume that every lattice point is a Lebesgue point and the growth bound can be weakened to $\log |m(\xi_1,\xi_2)| = O(|\xi|^2)$. 
\end{remark}

\begin{lemma}\label{sderivativesPw}
Let $P$ be a trigonometric polynomial. Let $2\leq p \leq \infty$, and suppose $s \geq 0$, $\alpha>0$. We have 
\begin{equation}\label{trigpolyconvergence}\varepsilon^{1/p} \cdot \|\langle \partial_x \rangle^s(Pw_{\alpha \varepsilon})\|_{L^p(\mathbb{R}^2)} \rightarrow (p\alpha)^{-1/p} \cdot \|\langle \partial_x \rangle^{s}P\|_{L^p(\mathbb{T}^2)} \quad \text{as} \ \  \varepsilon \rightarrow 0.\end{equation}
Here we interpret $(p\alpha)^{-1/p}=1$ if $p = \infty$. If $s=0$, then (\ref{trigpolyconvergence}) holds for all $1 \leq p \leq \infty$, $\alpha >0$. Analogous statements hold when the inhomogeneous derivatives $\langle \partial_x \rangle^s$ are replaced by homogeneous derivatives $|\partial_x|^s$ or ordinary derivatives $\partial_x^s, s \in \mathbb{N}$. 
\end{lemma}

\begin{proof}[Proof of Lemma \ref{sderivativesPw}]
    Take $N$ large enough so that the Fourier support of the trigonometric polynomial $P$ is contained in $([-N,N]\cap \mathbb{Z}) \times ([-N,N]\cap \mathbb{Z})$. The Fourier transform of $Pw_{\varepsilon/3}$ is a linear combination of shifted Gaussians $\sum_{m \in [-N,N]^2} \hat{P}(m)(\alpha \varepsilon)^{-1}e^{-\pi|\xi-m|^2/(\alpha \varepsilon)}$, resembling an increasingly sharp spike around each $m$ in $\supp\hat{P}$ as $\varepsilon$ goes to zero. This suggests approximating
    $$\mathcal{F}_{x \rightarrow \xi}(\langle \partial_x \rangle^s(Pw_{\alpha \varepsilon}))=\sum_{m \in [-N,N]^2} \langle \xi \rangle^s \hat{P}(n) (\alpha \varepsilon)^{-1}e^{-\pi|\xi-m|^2/(\alpha \varepsilon)}$$ by $$\mathcal{F}_{x \rightarrow \xi}((\langle \partial_x \rangle^sP)w_{\alpha \varepsilon})= \sum_{m \in [-N,N]^2}\langle m \rangle^s \hat{P}(m)(\alpha \varepsilon)^{-1}e^{-\pi|\xi-m|^2/(\alpha \varepsilon)}.$$
    Indeed, the error \begin{align*}E(\varepsilon):&=\varepsilon^{1/p}\mathcal{F}_{x \rightarrow \xi}(\langle \partial_x \rangle^s(Pw_{\alpha \varepsilon})-(\langle \partial_x \rangle^sP)w_{\alpha \varepsilon}) \\ & = \varepsilon^{1/p} \sum_{m \in [-N,N]^2}[\langle \xi \rangle^s - \langle m \rangle^s]\hat{P}(m)(\alpha \varepsilon)^{-1}e^{-\pi|\xi-m|^2/(\alpha \varepsilon)}\end{align*} is small in the sense that $\|E(\varepsilon)\|_{L^{p'}(\mathbb{R})}=o_{\varepsilon \rightarrow 0}(1)$ for any $1 \leq p' < \infty$, 
    so that Hausdorff-Young inequality gives $$\varepsilon^{1/p} \cdot \|\langle \partial_x \rangle^s(Pw_{\alpha \varepsilon}) - (\langle \partial_x \rangle^sP)w_{\alpha \varepsilon}\|_{L^p(\mathbb{R}^2)} \lesssim \|E(\varepsilon)\|_{L^{p'}(\mathbb{R}^2)} \rightarrow 0 \quad \text{as} \quad \varepsilon \rightarrow 0,$$ for $2 \leq p \leq \infty$.
    Verifying the assertion about the error term is a simple calculation:
    \begin{align*}
        \|E(\varepsilon)\|_{L^{p'}} \leq  \sum_{m \in [-N,N]^2} \Big[  \sup_{\xi \in \mathbb{R}^2} \big[|\langle \xi \rangle^s - \langle m \rangle^s|^{}e^{- \pi |\xi-m|^2/(2\alpha \varepsilon)} \big] \hat{P}(m)^{} \Big] \\ \cdot \Big[  \varepsilon^{p'/p}  \int_{\mathbb{R}^2} (\alpha \varepsilon)^{-p'} e^{-p' \pi |\xi-m|^2/(2\alpha \varepsilon)} d\xi \Big]^{1/p'} 
    \end{align*}
    and the latter factor involving the integral evaluates to a constant $(2/p')^{1/p'}\alpha^{1/p' -1}$ independent of $\varepsilon$, while the first factor involving the supremum goes to zero as $\varepsilon \rightarrow 0$.
    Hence, we have reduced the lemma to showing 
    \begin{equation*}
        \varepsilon^{1/p}\cdot\|(\langle \partial_x \rangle^sP)w_{\alpha \varepsilon}\|_{L^p(\mathbb{R}^2)} \rightarrow (p\alpha)^{-1/p} \cdot \|\langle \partial_x \rangle^sP\|_{L^p(\mathbb{T}^2)} \quad \text{as} \quad \varepsilon \rightarrow 0.
    \end{equation*}
    But $\langle \partial_x \rangle^sP$ is just another trigonometric polynomial and in particular continuous, so that \cite[Lemma 3.9]{MR304972} gives
\begin{align}\label{calculationsimilarStein}
    \begin{split}
        \varepsilon \|(\langle \partial_x \rangle^sP)w_{\alpha \varepsilon} \|_{L^p}^p & = (p\alpha)^{-1} \Big[ (p\alpha\varepsilon) \int_{\mathbb{R}^2}|\langle \partial_x \rangle^sP(x)|^p e^{- \pi(p\alpha \varepsilon)|x|^2} dx \Big] \\  \rightarrow & (p\alpha)^{-1} \int_{[0,1]^2}|\langle \partial_x \rangle^sP(x)|^pdx =\big[ (p\alpha)^{-1/p} \|\langle \partial_x \rangle^sP\|_{L^p(\mathbb{T}^2)}\big]^p \quad \text{as} \ \varepsilon \rightarrow 0.
        \end{split}
    \end{align}
    The proof for homogeneous derivatives or ordinary derivatives is verbatim the same upon replacing $\langle \xi \rangle^s$, $\langle m \rangle^s$ by $|\xi|^s$, $|m|^s$, or $\xi^s$, $m^s$, $s \in \mathbb{N}$ respectively. For $s=0$, the statement for more general $p \in [1,\infty]$ follows directly as in (\ref{calculationsimilarStein}).
\end{proof}

\begin{proof}[Proof of (\ref{commutatorsderivatives})]
We have the following Kato-Ponce commutator estimate for directional derivatives on the plane:
\begin{equation}\label{directionalKatoPonce}\|[\langle \partial_{\omega} \rangle^s,f]g]\|_{L^p(\mathbb{R}^2)} \lesssim \|\langle \partial_{\omega}\rangle^sf\|_{L^p(\mathbb{R}^2)} \|g\|_{L^{\infty}(\mathbb{R}^2)} + \|\partial_{\omega}f\|_{L^{\infty}(\mathbb{R}^2)} \|\langle \partial_{\omega}\rangle^{s-1}g\|_{L^p(\mathbb{R}^2)}.\end{equation}
This is obtained by the same argument as for (\ref{KenigPonceVegaQuasi}), using the Kato-Ponce commutator estimate on the line \cite{MR0951744} and a rotational dilation of the plane. 
Using the previous two Lemmas and a duality argument we can transfer this to the periodic setting.
Lemma \ref{PQR} and H\"older tell us that
\begin{equation*}
    \big\langle \tilde{\pi}(P,Q),R \big\rangle_{\mathbb{T}^2} = \lim_{\varepsilon \rightarrow 0} \varepsilon\int_{\mathbb{R}^2} \pi(Pw_{\varepsilon/3},Qw_{\varepsilon/3})\overline{Rw_{\varepsilon/3}} \leq \lim_{\varepsilon \rightarrow 0} \varepsilon^{} \|\pi(Pw_{\varepsilon/3},Qw_{\varepsilon/3})\|_{L^p(\mathbb{R}^2)} \|Rw_{\varepsilon/3}\|_{L^{p'}(\mathbb{R}^2)},
\end{equation*}
where we denoted the paraproduct $\pi(f,g)=[\langle\partial_{\omega}\rangle^s,f]g$ on the plane and its periodization $\tilde{\pi}(f,g)=[\langle \partial_x \rangle^s,f]g$.
Now estimate \eqref{directionalKatoPonce} bounds the expression inside the limit on the right hand side by
\begin{align*}
   &\varepsilon ^{1/p'}\|Rw_{\varepsilon/3}\|_{L^{p'}(\mathbb{R}^2)} \\ & \cdot  \Big[ \varepsilon^{1/p}\|\langle \partial_x \rangle^s(Pw_{\varepsilon/3})\|_{L^p(\mathbb{R}^2)} \cdot \|Qw_{\varepsilon/3}\|_{L^{\infty}(\mathbb{R}^2)} + \|\partial_x(Pw_{\varepsilon/3})\|_{L^{\infty}(\mathbb{R}^2)} \cdot \varepsilon^{1/p}\|\langle \partial_x \rangle^{s-1}(Qw_{\varepsilon/3})\|_{L^p(\mathbb{R}^2)} \Big]. 
\end{align*}
By Lemma \ref{sderivativesPw}, this converges up to a constant factor to 
\begin{equation*}
    \|R\|_{L^{p'}(\mathbb{T}^{2})} \cdot \big[\|\langle \partial_x \rangle^sP \|_{L^p(\mathbb{T}^2)} \|Q\|_{L^{\infty}(\mathbb{T}^2)} + \|\partial_xP\|_{L^{\infty}(\mathbb{T}^2)} \|\langle \partial_x \rangle^{s-1}Q\|_{L^p(\mathbb{T}^2)} \big]
\end{equation*}
as $\varepsilon \rightarrow 0$. We conclude by taking the supremum over all trigonometric polynomials $R$ with $\|R\|_{L^{p'}(\mathbb{T}^2)}=1$ and using duality.
\end{proof}

We have proven estimate (\ref{commutatorsderivatives} for trigonometric polynomials. By density arguments, the estimate is valid on larger classes of functions. In any case, we will apply this estimate to approximate solutions $u_n, u_m$ in Section \ref{wellposednessanisotropicspace}, which are, in fact, trigonometric polynomials.

\begin{remark}
    The transference principle in Proposition \ref{transferenceparaproduct}
    can also be obtained from Lemma \ref{PQR} and \ref{sderivativesPw} (with $s=0$) by a duality argument similar to the proof of estimate (\ref{commutatorsderivatives}).
    To obtain the endpoint case of Proposition \ref{transferenceparaproduct} for operators of type $L^p(\mathbb{R}^2) \times L^{\infty}(\mathbb{R}^2) \rightarrow L^p(\mathbb{R}^2)$, $1<p<\infty$, one further needs a variant of Lemma \ref{PQR} with $\beta=\varepsilon$. That is, \begin{equation*}\lim_{\varepsilon \rightarrow 0} \varepsilon^{} \int_{\mathbb{R}^2} \pi(Pw_{\alpha \varepsilon},Qw_{\varepsilon^2})(x)\overline{R(x)}w_{\gamma \varepsilon}(x)dx = \int_{[-1/2, 1/2]^2} \tilde{\pi}(P,Q)(x)\overline{R(x)}dx,\end{equation*}
    whenever $\alpha, \gamma >0$ with $\alpha+\gamma=1$. This can be shown analogously to the proof of Lemma \ref{PQR}.
\end{remark}

\begin{remark}
    The Kato-Ponce commutator estimate on the real line is proved in \cite{MR0951744} by using Bony's paraproduct decomposition and applying the Coifman-Meyer theorem to each part. One could also prove (\ref{commutatorsderivatives}) by performing a similar paraproduct decomposition with respect to the tangential frequency $\omega.\xi$ and then applying the bilinear transference theorem \cite[Theorem 3]{MR1808390} to the individual parts.
    There is a similar discussion in \cite{MR4840270}
    about the fractional product rule.

    Another approach is to work directly on the torus, to perform an appropriate paraproduct decomposition with respect to the tangential frequency $\omega \cdot n$ and to try establishing a version of the Coifman-Meyer theorem applicable in that setting. While there is a version of the Coifman-Meyer theorem on the torus, see \cite[Problem 3.4]{schlag2012classical}, it is also necessary to address any complications caused by the fact that the paraproduct decomposition is with respect to the projection $\omega \cdot n$.
    The proof presented here benefits from the fact that rotational dilations do not cause any difficulties in the Euclidean setting.
\end{remark}

In Sections \ref{wellposednessanisotropicspace}-\ref{sectionfinisihingtheproof} we will frequently use the following version of Gronwall's inequality:
\begin{equation}\label{Gronwall}u(t) \leq \alpha(t) + \int_0^t \beta(s)u(s)ds \implies u(t) \leq \alpha(t) \exp \big(\int_0^t \beta(s) ds \big),\end{equation}
where $\alpha, \beta$ and $u$ are continuous real-valued functions defined on an interval $[0,T]$ with $\beta$ non-negative and $\alpha$ non-decreasing. 
In Section \ref{sectionaprioriproposition}, we will make use of Hille's Theorem asserting that the application of closed linear operators commutes with the Bochner integral provided the expressions involved are integrable.

\begin{proposition}\label{Hillestheorem}[Hille's Theorem \cite{MR89373, MR4802867}] Suppose $(X,\Sigma,\mu)$ is a measure space, $T:B \rightarrow B'$ is a closed linear operator between Banach spaces $B$ and $B'$, and both $f:X \rightarrow B$ and $Tf:X \rightarrow B'$ are Bochner integrable. Then $$\int_E Tf = T \int_E f$$ for all measurable $E \in \Sigma$. 
\end{proposition}

\section{wellposedness in anisotropic Sobolev spaces and in $X_{\omega}^{s_1,s_2}$}\label{wellposednessanisotropicspace}

In this section, we show that the Benjamin--Ono equation is wellposed in the Sobolev-type space $H^{s_1,s_2}_{\omega}$ of quasi-periodic functions, as well as in the related space $X=X_{\omega}^{s_1,s_2}$, for $\frac{1}{2}<s_2<1$ and $s_1$ sufficiently large. In other words, we prove Theorems \ref{wellposednessinX} and \ref{wellposednessinHs1s2} stated in the introduction. \\

  The proof of Theorem \ref{wellposednessinHs1s2} proceeds by Galerkin approximation and energy method arguments:
  First, a family of finite-dimensional approximations to the problem can be solved by classical ODE theory. The regularization that we use is a restriction to a finite number of frequencies. This has the property that the $L^2$ norm is still a conserved quantity and coercive, thus yielding global solutions to these approximate equations. This is done in Section \ref{regularization}. Second, we establish an a-priori bound on the $H^{s_1,s_2}_{\omega}$-norm on a uniform time interval independent of the regularization parameter. This is done in Section \ref{uniformaprioriestimate}. Third, we show that the sequence of approximate solutions is Cauchy in $L^{\infty}([0,T],H^{s_1,s_2}_{\omega})$. This step crucially relies on separate initial data-regularization; see section \ref{sectionCauchybound}. Fourth, the limit is shown to be a solution to the Benjamin--Ono equation in Section \ref{sectionsolution}. Fifth, uniqueness is proven by a Gronwall argument in Section \ref{sectionuniqueness}. Sixth, continuity of the data-to-solution map is established by compactness arguments in Section \ref{sectioncontinuitydatatosolution}. \\

Theorem \ref{wellposednessinX} is a slight variation of Theorem \ref{wellposednessinHs1s2} and its proof essentially only requires additional Cauchy bounds on the $\|\langle \partial_y \rangle^{s_2} \partial_x^{-1}u\|_{L^2}$-part of the norm defining the space $X_{\omega}^{s_1,s_2}$, see Definition \ref{definitionX}.
This is done in Section \ref{sectionwellposednessinX} using the dynamical equation for $F=\partial_x^{-1}u$ and a Gronwall argument.

As an offshoot of the proof of Theorem \ref{wellposednessinHs1s2}, we obtain the following bound controlling the growth of the $H^{s_1,s_2}_{\omega}$-norm. It is a consequence of the a-priori estimate on solutions to the regularized equation in Section \ref{uniformaprioriestimate}.
 \begin{proposition}\label{Hnormgrowth}
        The solution provided by the $H^{s_1,s_2}_{\omega}$ wellposedness theory (Theorem \ref{wellposednessinHs1s2}) satisfies the estimate $$\|u(t)\|_{H^{s_1,s_2}_{\omega}} \leq \|u_0\|_{H^{s_1,s_2}_{\omega}} \exp \big( c  \int_0^t \|u(s)\|_{L^{\infty}} + \|\partial_xu(s)\|_{L^{\infty}} ds \big).$$ Here $c$ is some fixed absolute constant.
    \end{proposition}
A similar argument using Gronwald's inequality controls the growth of the other part of the $X$-norm. We thus establish a bound controlling the growth of the $X$-norm as stated in Proposition  \ref{normgrowthX} in the introduction. Both Propositions \ref{Hnormgrowth} and \ref{normgrowthX} are proven in Section \ref{sectionnormgrowthcontrol}. \\

\subsection{Regularized equation}\label{regularization}
In this section, we construct global smooth solutions for a family of regularizations of the Benjamin--Ono equation. \\

We consider the regularization \begin{equation}\label{regularizedBO}u_t =- P_nH\partial_{xx}u + P_n \partial_x(P_nu P_nu)\end{equation}
where $P_n$ is projection to Fourier modes $k$ with $|\omega.k|, | \omega^{\perp} . k|\leq n$, see \eqref{generalprojection}.
On the Fourier side, this becomes a system of finitely many coupled ODEs, which locally has smooth solutions by the Picard-Lindel\"of theorem. In fact, the Cauchy–Kovalevskaya theorem is applicable and the solutions are even analytic. \\

Analogous to the actual Benjamin--Ono equation, the regularized equation conserves the $L^2$-norm:  
\begin{equation}\label{L2conservation}
    \frac{d}{dt} \|u(t)\|_{L^2}^2 = \int_{\mathbb{T}^2} -u(t) P_nH\partial_{xx}u(t) + u(t)P_n[\partial_x(P_nu(t)P_nu(t))] dt = 0.
\end{equation}
The contribution of the first term in the integral vanishes by anti-selfadjointness of $-P_nH\partial_{xx}$. The vanishing of the contribution of the second summand in the integral can be seen by moving the projection $P_n$ onto the first factor by Plancherel and then using integration by parts and product rule. 
By Plancherel, the $l^2$-norm of the finite number of Fourier modes thus serves as a coercive constraint for the system of ODEs, allowing us to construct global smooth solutions. \\

None of this is surprising; in fact, the dynamics of the regularized equation (\ref{regularizedBO}) is generated by the Hamiltonian \begin{equation}\label{Hamiltonian}H[u]:= \int_{\mathbb{T}^2} -\frac{1}{2} (P_nu) H\partial_x(P_nu) + \frac{1}{3}(P_nu)^3\end{equation} with respect to the Poisson structure 
$$\{F(u),G(u)\}:=\langle \partial_xD_uF(u),  D_uG(u)\rangle_{L^2(\mathbb{T}^2)}$$
on the subspace of $C^{\infty}(\mathbb{T}^2)$ consisting of real-valued functions with Fourier support contained in $$I:=\{k \in \mathbb{Z}^2: |\omega.k|, |\omega^{\perp}.k|\leq n\}.$$
Here $D_uF$ is the functional derivative characterized by $$F(u+\phi)=F(u)+\langle D_uF(u),\phi \rangle_{L^2\mathbb{T}^2} + o(\|\phi\|_{L^2\mathbb{T}^2})$$ for all $\phi$ in the phase space just described. \\

The Hamiltonian (\ref{Hamiltonian}) is a truncated version of the Hamiltonian for the ordinary Benjamin--Ono equation, differing only in the presence of the Fourier projections $P_n$. The phase space is restricted and becomes finite-dimensional. 
Equation (\ref{regularizedBO}) now reads as $\partial_tu = \{H,u\}.$ 
The vanishing Poisson bracket $$ \Big\{ \int_{\mathbb{T}^2}u^2,H\Big\}=0$$ exhibits the $L^2$-norm as a conserved quantity; note that vanishing of the Poisson bracket amounts to the same calculation as (\ref{L2conservation}).
A finite-dimensional Hamiltonian system with a coercive conserved quantity automatically has global smooth solutions. \\

We can describe the phase space in frequency domain concretely as $$\mathcal{M}=\{\hat{u}=[\hat{u}(k)]_{k \in I}: \hat{u}(-k)=\overline{\hat{u}(k)} \in \mathbb{C} \} \cong_{\mathbb{R}} \mathbb{C}^{\frac{|I|-1}{2}} \oplus \mathbb{R} \cong_{\mathbb{R}} \mathbb{R}^{|I|}.$$
The Hamiltonian at a state $(\hat{u}(k))_{k \in I}$ can by Plancherel be written as 
$$H[(\hat{u}(k))_{k \in I}]= -\frac{1}{2} \sum_{i,j \in I: i+j=0} \hat{u}(i) \text{sgn}(\omega.j) (\omega.j)\hat{u}(k) + \frac{1}{3} \sum_{i,j,l\in I: \ i+j+l=0} \hat{u}(i)\hat{u}(j)\hat{u}(l),$$
and the Poisson bracket of two functions $F,G \in C^{\infty}(\mathcal{M})$ is given by $$\{F,G\}= \sum_{k \in I} \overline{\partial_{\hat{u}(k)}F(\hat{u}) i(\omega \cdot k)} \partial_{\hat{u}(k)}G(\hat{u}).$$
Hamilton's equation $\partial_t \hat{u} = \{H,\hat{u}\}$ is the mentioned system of coupled ODEs associated with (\ref{regularizedBO}).
\\

Throughout Sections \ref{uniformaprioriestimate}-\ref{sectionnormgrowthcontrol}, we let $u_n$ be the global smooth solution to equation (\ref{regularizedBO}) with regularized initial data \begin{equation}\label{initialdata}u_0^{\delta(n)}= P_{\delta(n)^{1/s_1}, \delta(n)^{1/s_2}}u_0, \end{equation} the parameter $\delta(n) \leq n$ an increasing function of $n$ to be specified later. Here $P_{\delta(n)^{1/s_1}, \delta(n)^{1/s_2}}$ is the projection to frequencies with $|\omega.k| \leq  \delta(n)^{1/s_1}$ and $|\omega^{\perp}.k|\leq  \delta(n)^{1/s_2}$, see (\ref{generalprojection}).

\subsection{Uniform a-priori estimate}\label{uniformaprioriestimate}
We show that the solutions $u_n$ to the regularized equation \eqref{regularizedBO} satisfy an a-priori estimate on the $H^{s_1,s_2}_{\omega}$-norm, uniformly in the regularization parameter $n$.

\begin{proposition}[$H_{\omega}^{s_1,s_2}$ bound uniform in regularization parameter]\label{Hshbound}
    Suppose $1/2 < s_2 \leq 1$ and $s_1 > \frac{3s_2}{2s_2-1}$. There exists a time $T:=T(\|u_0\|_{H^{s_1,s_2}_{\omega}})$ such that the solutions $u_n$ satisfy the bound $$\|u_n(t)\|_{L^{\infty}([0,T],H_{\omega}^{s_1,s_2})} \leq 2 \|u_0\|_{H^{s_1,s_2}_{\omega}}$$
    uniformly in $n$. Further, $T$ can be taken uniformly for initial data from a compact set of initial data in $H^{s_1,s_2}_{\omega}$.
\end{proposition}

\begin{proof}
    The key observation is that $$\|\partial_x u\|_{L^{\infty}} \lesssim \|u\|_{H^{s_1,s_2}_{\omega}},$$ see Lemma \ref{anisotropicproperties}. 
    We calculate that
    \begin{align*}\label{xsderivative}
        \frac{d}{dt}\|\langle \partial_x \rangle^{s_1}u_n\|_{L^2}^2 & =\int_{\mathbb{T}^2} -\langle \partial_x \rangle^{s_1} u_n H\partial_{xx} \langle \partial_x \rangle^{s_1} u_n + \frac{1}{2}\int_{\mathbb{T}^2} \langle \partial_x \rangle^{s_1}u_n \partial_x \langle \partial_x \rangle^{s_1} u_n^2 \\ & = \int_{\mathbb{T}^2} \langle \partial_x \rangle^{s_1}u_n \big([\langle \partial_x \rangle^{s_1}, u_n]\partial_x u_n + u_n \partial_x\langle \partial_x \rangle^{s_1} u_n\big)  \\ 
        & \lesssim \|\langle \partial_x \rangle^{s_1} u_n\|_{L^2} \|[\langle \partial_x \rangle^{s_1},u_n] \partial_x u_n\|_{L^2} + \Big|\frac{1}{2} \int_{\mathbb{T}^2} \langle \partial_x \rangle^{s_1} u_n \langle \partial_x \rangle^{s_1} u_n \partial_x u_n \Big| \\ &
        \lesssim \|\langle \partial_x \rangle^{s_1} u_n\|_{L^2}^2 \|\partial_x u_n\|_{L^{\infty}}.
    \end{align*}
    In the first line, we used equation (\ref{regularizedBO}) and the fact that the projections $P_n$ in the nonlinear term of equation (\ref{regularizedBO}) can be dropped by an application of Plancherel. In the second line, we used anti-selfadjointness of $-H\partial_{xx}$. In the third line, we used integration by parts and H\"older, and in the fourth line we used the quasi-periodic Kato-Ponce commutator estimate for tangential derivatives (\ref{commutatorsderivatives}).

    Analogously, using the quasi-periodic Kenig-Ponce-Vega commutator estimate for normal derivatives \eqref{KenigPonceVegaQuasi}, 
    we have 
    \begin{equation}\label{yhderivative}
    \begin{split}
        \frac{d}{dt} \|\langle \partial_y \rangle^{s_2} u_n\|_{L^2}^2 &= \int_{\mathbb{T}^2} \langle \partial_y \rangle^{s_2} u_n [\langle \partial_y \rangle^{s_2} , u_n]\partial_x u_n + \int_{\mathbb{T}^2} \langle \partial_y \rangle^{s_2} u_n \cdot u_n \langle \partial_y \rangle^{s_2} \partial_x u_n \\ &\leq \|\langle \partial_y \rangle^{s_2}u_n\|_{L^2}\|[\langle \partial_y \rangle^{s_2}, u_n]\partial_xu_n\|_{L^2} + \frac{1}{2}\Big|\int_{\mathbb{T}^2} \langle \partial_y \rangle^{s_2} u_n \langle \partial_y \rangle^{s_2} u_n \partial_x u_n \Big| \\ & \lesssim \|\langle \partial_y \rangle^{s_2} u_n\|_{L^2}^2 \|\partial_x u_n\|_{L^{\infty}}.
    \end{split}
    \end{equation}
    Gronwall's inequality \eqref{Gronwall} gives that
    \begin{align}\label{normgrowthprelim}
\|u_n(t)\|_{H^{s_1,s_2}_{\omega}} & \leq \|u_0^{\delta(n)}\|_{H^{s_1,s_2}_{\omega}} \exp \Big( c \int_0^t \|\partial_x u_n(s)\|_{L^{\infty}} ds \Big)  \leq \|u_0\|_{H^{s_1,s_2}_{\omega}} \exp \big( ct\|u_n\|_{L^{\infty}([0,t],H^{s_1,s_2}_{\omega})} \big),
    \end{align}
    where we used the key observation $\|\partial_xu_n \|_{L^{\infty}} \lesssim \|u\|_{H^{s_1,s_2}_{\omega}}$ in the last inequality.
    Upon taking $T$ such that $\exp(T \cdot 4c\|u_0\|_{H^{s_1,s_2}_{\omega}}) =2$, we can use the above estimate to bootstrap with $$\Omega_1 := \{t\in[0,T]: \sup_{s \in [0,t]} \|u(s)\|_{H^{s_1,s_2}_{\omega}} \leq 4 \|u_0\|_{H^{s_1,s_2}_{\omega}}\}$$ and $$\Omega_2 := \{ t \in [0,T]: \sup_{s \in [0,t]} \|u(s)\|_{H^{s_1,s_2}_{\omega}} \leq 2 \|u_0\|_{H^{s_1,s_2}_{\omega}} \}$$
    as bootstrap hypothesis and conclusion, respectively. This yields $\Omega_2 =[0,T]$. The choice of $T$ was independent of $n$. 
    Given a compact set $\mathcal{K} \subseteq H^{s_1,s_2}_{\omega}$ of initial data, we take $T$ such that \begin{equation*}\exp(T \cdot 4c \cdot \sup_{u_0 \in \mathcal{K}}\|u_0\|_{H^{s_1,s_2}_{\omega}})=2. \qedhere \end{equation*}
\end{proof}

Observe that the above calculations yielding estimate (\ref{normgrowthprelim}) are valid for any $s_1\geq0$, $1\geq s_2\geq0$. This will be used in the next subsection and we record the following statement.
\begin{lemma}[Control on norm growth, regularized equation]\label{normgrowthregularized}
    Suppose $1\geq s_2>1/2$ and $s_1>\frac{3s_2}{2s_2-1}$. Suppose we are given $\tilde{s}_1\geq 0$ and $1 \geq \tilde{s}_2 \geq 0$. Then the following estimate holds: $$\|u_n(t)\|_{H^{\tilde{s}_1.\tilde{s}_2}_{\omega}} \lesssim \|u_0\|_{H^{\tilde{s}_1,\tilde{s}_2}_{\omega}} \exp \big( \int_0^t \|\partial_x u_n(s)\|_{L^{\infty}} ds \big) \lesssim \|u_0\|_{H^{\tilde{s}_1,\tilde{s}_2}_{\omega}} \exp \big(\int_0^t \|u_n(s)\|_{H^{s_1,s_2}_{\omega}}ds\big).$$
\end{lemma}
When $\tilde{s}_1 \geq s_1$ and $\tilde{s}_2 \geq s_2$ this can be interpreted as a blow-up alternative for the $H^{s_1,s_2}_{\omega}$ and $H^{\tilde{s}_1,\tilde{s}_2}_{\omega}$ norms. Of course, we already know that the regularized solutions exist globally. But in the limit as $n\rightarrow \infty$ this does give a blow-up alternative for the actual Benjamin--Ono equation (once we have established that the approximate solutions $u_n$ converge to the solutions given by the wellposedness theories in $H^{s_1,s_2}_{\omega}$ and $H^{\tilde{s}_1,\tilde{s}_2}_{\omega}$).
\begin{lemma}[Blow-up alternative]
    Given initial data $u_0 \in H^{\tilde{s}_1,\tilde{s}_2}_{\omega} \subseteq H^{s_1,s_2}_{\omega}$, the $C([0,T_1],H^{\tilde{s}_1,\tilde{s}_2}_{\omega})$ solution to the Benjamin--Ono equation exist on the same time interval as the $C([0,T_2],H^{s_1,s_2}_{\omega})$ solution, that is, $T_1=T_2$.
\end{lemma}

\begin{remark}\label{remarksymmetricSobolevnormgrowthcontrol}
    Estimate (\ref{symmetricSobolevnormgrowthcontrol}) in the introduction controlling the growth of ordinary Sobolev norms for sufficiently regular solutions to the Benjamin--Ono equation follows by the same reasoning as in this section, but with the Kato-Ponce commutator estimate for joint inhomogeneous derivatives \begin{equation}\label{KatoPoncetorus}\|[\langle \nabla \rangle^s,f]g\|_{L^2(\mathbb{T}^2)} \lesssim \|\langle \nabla \rangle^s f\|_{L^2(\mathbb{T}^2)}\|g\|_{L^{\infty}(\mathbb{T}^2)}+\|\langle \nabla \rangle f\|_{L^{\infty}(\mathbb{T}^2)}\|\langle \nabla \rangle^{s-1}g\|_{L^2(\mathbb{T}^2)}\end{equation} (applied to $f=u$ and $g = \partial_xu$) instead of the commutator estimates for tangential and normal inhomogeneous derivatives in Lemma \ref{KatoPonceVegaQuasih}. 
    The above commutator estimate \eqref{KatoPoncetorus} can be obtained from the classical Kato-Ponce commutator estimate on the Euclidean plane \cite{MR0951744} by the same transference arguments as in the proof of estimate (\ref{commutatorsderivatives}) using Lemmas \ref{PQR} and \ref{sderivativesPw}.
\end{remark}

\subsection{Cauchy bound}\label{sectionCauchybound}
In this section, we show that the sequence of solutions $u_n$ to the regularized equation \eqref{regularizedBO} with initial data \eqref{initialdata} is Cauchy in $L^{\infty}([0,T],H^{s_1,s_2}_{\omega})$.

\begin{proposition}\label{HshCauchy}
    The sequence $(u_n)_{n \in \mathbb{N}}$ is uniformly Cauchy in $L^{\infty}([0,T],H^{s_1,s_2}_{\omega})$ for initial data from a compact set $\mathcal{K} \subseteq H^{s_1,s_2}_{\omega}$. That is, $$\|u_n-u_m\|_{L^{\infty}([0,T],H^{s_1,s_2}_{\omega})} = o_{m \rightarrow \infty}(1)$$ uniformly for $u_0 \in \mathcal{K}$.
\end{proposition}
The key step in showing Proposition \ref{HshCauchy} is the following estimate:
\begin{equation}\label{keyHsCauchy}
    \frac{d}{dt} \|(u_n-u_m)(t)\|_{H^{s_1,s_2}_{\omega}} \lesssim \|(u_n-u_m)(t)\|_{H^{s_1,s_2}_{\omega}} [\|\partial_xu_n(t)\|_{L^{\infty}} + \|\partial_xu_m(t)\|_{L^{\infty}}] + o_{m \rightarrow \infty}(1).
\end{equation}
Then, Gronwall's inequality bounds \begin{equation*}
    \|(u_n-u_m)(t)\|_{H^{s_1,s_2}_{\omega}} \leq [\|(u_n-u_m)(0)\|_{H^{s_1,s_2}_{\omega}}+t \cdot o_{m \rightarrow \infty}(1) ] \exp \Big( c \int_0^t\|\partial_xu_n(s)\|_{L^{\infty}} + \|\partial_xu_m(s)\|_{L^{\infty}} ds \Big).
\end{equation*}
Using the a-priori bound in Proposition \ref{HshCauchy} together with estimate \eqref{crucialequation} in Lemma \ref{anisotropicproperties}, the latter factor involving the exponential is uniformly bounded in $n,m \in \mathbb{N}$ and $u_0 \in \mathcal{K}$, and the overall expression goes to zero uniformly in $t\in[0,T]$ and $u_0 \in \mathcal{K}$ as $n>m \rightarrow \infty$. \\

The proof of the key estimate (\ref{keyHsCauchy}) requires first establishing quantitative Cauchyness of the sequence of approximate solutions in the space $L^{\infty}([0,T],L^2)$, which is the content of the following proposition.

\begin{proposition}\label{L2Cauchy}
    Let $0<q<s_2$ and $s_1$ sufficiently large depending on $q$ and $s_2$. 
    The sequence $\{u_n\}_{n \in \mathbb{N}}$ is Cauchy in $L^{\infty}([0,T],L^2)$ with  $$\|u_n-u_m\|_{L^{\infty}([0,T],L^2)} \lesssim [m^{-q}\cdot T + \delta(m)^{-1} \cdot o_{m \rightarrow \infty}(1)],
     \quad n>m.$$
    The implicit constant can be chosen uniformly for all initial data from a given compact set $\mathcal{K} \subseteq H^{s_1,s_2}_{\omega}$.
\end{proposition}

\begin{proof}
By anti-selfadjointness of $-H\partial_{xx}$ we have
\begin{equation}\label{derivativeL2}\frac{d}{dt} \|(u_n-u_m)\|^2_{L^2} = \int_{\mathbb{T}^2} [u_n-u_m]\partial_x[P_n(u_n^2) - P_m(u_m^2)].\end{equation}
    We write the contribution from the nonlinearity as a difference of squares, while picking up another term whose frequency support is bounded away from zero: $$\partial_x(P_n(u_n^2)-P_m(u_m^2)) = \partial_x((P_n - P_m)u_m^2) + P_n\partial_x[(u_n-u_m)(u_n+u_m)].$$
The contribution from the difference of squares can be estimated as \begin{equation}\label{firstterm}\int_{\mathbb{T}^2} (u_n-u_m) P_n\partial_x[(u_n-u_m)(u_n+u_m)] \lesssim \|u_n-u_m\|_{L^2}^2 [\|\partial_xu_n\|_{L^{\infty}} + \|\partial_x u_m\|_{L^{\infty}}].\end{equation} 
Here, we used Plancherel to move the projection $P_n$ to the first factor, where it can be dropped, followed by the product rule and integration by parts.
Next, we estimate the contribution from the first summand. Using that the frequency support is bounded away from zero we gain decay $m^{-q}$ at the cost of paying extra derivatives $\langle \nabla \rangle^q$.
\begin{align}\label{secondterm2}\begin{split}
    \int_{\mathbb{T}^2} (u_n-u_m) \partial_x((P_n-P_m)u_m^2) &  \leq   \|u_n-u_m\|_{L^2} \|\langle \nabla \rangle^q\partial_x  (P_n-P_m)u_m^2\|_{L^2} \cdot m^{-q} 
    \\ & \lesssim m^{-q} \|u_n-u_m\|_{L^2} \| u_m\|_{H^{s_1,s_2}_{\omega}}^2 \lesssim m^{-q} \|u_n-u_m\|_{L^2}.
    \end{split}
\end{align}
In the second line, we use boundedness of $\langle \nabla \rangle^q \partial_x: H^{s_1,s_2}_{\omega} \rightarrow L^2$ due to $$\xi_1\langle|\xi|\rangle^q \lesssim \langle \xi_1 \rangle^{1+q} + \langle \xi_1 \rangle \langle \xi_2 \rangle^q \lesssim \langle \xi_1 \rangle^{\max(1+q,\frac{1}{1-q/s_2})} + \langle \xi_2 \rangle^{s_2} \lesssim \langle \xi_1\rangle^{s_1} + \langle \xi_2 \rangle^{s_2}$$ which follows from Young's inequality provided that $0<q<s_2, \max(\frac{1}{1-q/s_2}, 1+q) < s_1$. We also used the algebra property of $H^{s_1,s_2}_{\omega}$ and the a-priori bound $\|u_m\|_{H^{s_1,s_2}_{\omega}} \lesssim \|u_0\|_{H^{s_1,s_2}_{\omega}} \lesssim 1$ in Lemma \ref{Hshbound}.
Adding \eqref{firstterm} and \eqref{secondterm2} and combining with \eqref{derivativeL2} gives
$$\frac{d}{dt} \|(u_n-u_m)(t)\|_{L^2} \lesssim \|(u_n -u_m)(t)\|_{L^2}[\|\partial_xu_n\|_{L^{\infty}} + \|\partial_xu_m\|_{L^{\infty}}] + m^{-q}.$$
Upon integrating in time and applying Gronwall's inequality \eqref{Gronwall}, this yields 
$$\|u_n-u_m\|_{L^{\infty}([0,T],L^2)} \lesssim [cTm^{-q} + \|u_0^{\delta(n)} - u_0^{\delta(m)}\|_{L^2}] \exp \Big(c \int_0^t \|\partial_xu_n(s)\|_{L^{\infty}} + \|\partial_xu_m(s)\|_{L^{\infty}}ds \Big).$$
Finally, we have by initial data regularization (\ref{initialdata}) that $$\|u_0^{\delta(n)} - u_0^{\delta(m)}\|_{L^2} \leq \delta(m)^{-1}\|u_0^{\delta(n)} - u_0^{\delta(m)}\|_{H^{s_1,s_2}_{\omega}} = \delta(m)^{-1} \cdot o_{m \rightarrow \infty}(1).$$ 
In the last line we used that $\|u_0^{\delta(n)} - u_0^{\delta(m)}\|_{H^{s_1,s_2}_{\omega}} = o_{m \rightarrow \infty}(1)$ uniformly for all $u_0 \in \mathcal{K}$ due to equicontinuity of the compact set $\mathcal{K} \subseteq H^{s_1,s_2}_{\omega}$ of initial data (compare Riesz-Kolmogorov).
Together with estimate (\ref{crucialequation}) and the a-priori estimate in Lemma \ref{Hshbound} this completes the proof.
\end{proof}

We return to the key estimate \eqref{keyHsCauchy} and break it into two pieces:
\begin{equation}\label{growthxderivatives}
    \frac{d}{dt} \|\langle \partial_x \rangle^{s_1}(u_n-u_m)\|_{L^2} \lesssim \|\langle \partial_x \rangle^{s_1}(u_n-u_m)\|_{L^2} [\|\partial_xu_n\|_{L^{\infty}} + \|\partial_xu_m\|_{L^{\infty}}] + o_{m \rightarrow \infty}(1).
\end{equation}
and
\begin{equation}\label{growthyderivatives}
    \frac{d}{dt} \|\langle \partial_y \rangle^{s_2}(u_n-u_m)\|_{L^2} \lesssim \|\langle \partial_y \rangle^{s_2}(u_n-u_m)\|_{L^2} [\|\partial_xu_n\|_{L^{\infty}} + \|\partial_xu_m\|_{L^{\infty}}] + o_{m \rightarrow \infty}(1).
\end{equation}
The remainder of this section is concerned with proving these two estimates.
\begin{proof}[Proof of estimate \eqref{growthyderivatives}]
By anti-selfadjointness of $-H\partial_{xx}$ we have
\begin{equation}\label{yderivatives}\frac{d}{dt}\|\langle \partial_y \rangle^{s_2}(u_n-u_m)\|_{L^2}^2 = \int_{\mathbb{T}^2} \langle \partial_y \rangle^{s_2}[u_n-u_m] \langle \partial_y \rangle^{s_2} \partial_x [P_n(u_n^2) - P_m(u_m^2)]. \end{equation}
We decompose the difference of the nonlinearities as the sum of two terms, one of which is a difference of squares and the other one has frequency support bounded away from zero:
\begin{align*}
         \langle \partial_y \rangle^{s_2} \partial_x (P_n(u_n^2) - P_m(u_m^2)) = & \underbrace{\langle \partial_y \rangle^{s_2} \partial_x ((P_n - P_m)u_m^2)}_{=:I} + \underbrace{P_n\langle \partial_y \rangle^{s_2} \partial_x ((u_n-u_m)(u_n+u_m))}_{=:I\kern-1.5pt I} .
\end{align*}
When considering the contribution of $I\kern-1.5pt I$ to \eqref{yderivatives}, the projection $P_n$ can be dropped by Plancherel, and we decompose this term further as 
\begin{equation}\label{II2}
\begin{split}
    I\kern-1.5pt I = \underbrace{[\langle \partial_y \rangle^{s_2}, u_n-u_m] \partial_x (u_n+u_m)}_{=:I\kern-1.5pt I_1} + \underbrace{(u_n-u_m) \langle \partial_y \rangle^{s_2} \partial_x(u_n+u_m) }_{=:I\kern-1.5pt I_2} \\ + \underbrace{[\langle \partial_y \rangle^{s_2}, (u_n+u_m)] \partial_x (u_n-u_m)}_{=:I\kern-1.5pt I_3} + \underbrace{(u_n+u_m) \langle \partial_y \rangle^{s_2} \partial_x(u_n-u_m)}_{=:I\kern-1.5pt I_4}.
    \end{split}
\end{equation}

The rearrangement in \eqref{II2} is strategic: When all the derivatives $\langle \partial_y \rangle^{s_2}$ fall on one factor, the resulting terms $I\kern-1.5pt I_2$ and $I\kern-1.5pt I_4$ can be estimated by using integration by parts and by exploiting separate initial data regularization, respectively.
The difference between these terms and $I\kern-1.5pt I$ gives rise to commutator terms  $I\kern-1.5pt I_1$ and $I\kern-1.5pt I_3$, which are amenable to be handled by the Kenig-Ponce-Vega commutator estimate (\ref{KenigPonceVegaQuasi}). We estimate the contribution to (\ref{yderivatives}) of each of these terms: \\

\item First, the contribution by $I\kern-1.5pt I_4:$ We estimate using integration by parts that $$\int_{\mathbb{T}^2} \langle \partial_y \rangle^{s_2}(u_n-u_m) \cdot (u_n+u_m) \langle \partial_y \rangle^{s_2} \partial_x (u_n-u_m) \lesssim \|\langle \partial_y \rangle^{s_2}(u_n-u_m)\|_{L^2}^2 [\|\partial_xu_n\|_{L^{\infty}}+\|\partial_xu_m\|_{L^{\infty}}].$$ \\

\item Second, the contribution by $I\kern-1.5pt I_3:$ By quasi-periodic Kenig-Ponce-Vega commutator estimate (\ref{KenigPonceVegaQuasi}) we have that
\begin{align*}
    & \int_{\mathbb{T}^2} \langle \partial_y \rangle^{s_2}(u_n-u_m) [\langle \partial_y \rangle^{s_2}, u_n+u_m] \partial_x(u_n - u_m) \\ \lesssim & \|\langle \partial_y \rangle^{s_2}(u_n-u_m)\|_{L^2} \|\langle \partial_y \rangle^{s_2}(u_n+u_m)\|_{L^2} \|\partial_x(u_n-u_m)\|_{L^{\infty}}
\end{align*}
The second factor is bounded for initial data $u_0$ from a given compact set by Lemma \ref{normgrowthregularized}: 
\begin{align*}\|\langle \partial_y \rangle^{s_2} u_m(t)\|_{L^2} \lesssim \| \langle \partial_y \rangle^{s_2}u_0 \|_{L^2} \exp\Big(\int_0^t \|\partial_x u_m(s)\|_{L^{\infty}}ds \Big) \lesssim 1.\end{align*} We gain decay from the last factor by using anisotropic Sobolev embedding \eqref{anisotropicembedding} and interpolating between boundedness in $H^{s_1,s_2}_{\omega}$ (Lemma \ref{Hshbound}) and Cauchy decay for the $L^{\infty}([0,T],L^2)$-norm (Lemma \ref{L2Cauchy}):
\begin{equation}\label{partialxinLinfinity}\|\partial_x(u_n-u_m)\|_{L^{\infty}} \lesssim \|u_n-u_m\|_{H^{s_1-\varepsilon, s_2-\varepsilon}_{\omega}} \lesssim \|u_n-u_m\|_{H^{s_1,s_2}_{\omega}}^{p(\varepsilon)} \|u_n-u_m\|_{L^2}^{p(\varepsilon)'} = o_{m \rightarrow \infty}(1).\end{equation} Here $\varepsilon>0$ is some small constant so that estimate (\ref{crucialequation}) holds and $p(\varepsilon), p(\varepsilon)'$ is a pair of H\"older conjugate exponents appropriate for above interpolation, see \eqref{interpolationestimate} in Lemma \ref{anisotropicproperties}. \\

\item Third, the contribution by $I\kern-1.5pt I_1:$ Applying the quasi-periodic Kenig-Ponce-Vega commutator estimate \eqref{KenigPonceVegaQuasi} yields $$\int_{\mathbb{T}^2} \langle \partial_y \rangle^{s_2}(u_n-u_m)[\langle \partial_y \rangle^{s_2}, u_n-u_m]\partial_x(u_n+u_m) \lesssim \|\langle \partial_y \rangle^{s_2}(u_n-u_m)\|_{L^2}^2 (\|\partial_x u_n\|_{L^{\infty}}+\|\partial_x u_m\|_{L^{\infty}}).$$ \\

\item Fourth, the contribution by $I\kern-1.5pt I_2:$ We begin estimating by H\"older that \begin{align*} & \int_{\mathbb{T}^2} \langle \partial_y \rangle^{s_2}(u_n-u_m)\cdot (u_n-u_m) \langle \partial_y \rangle^{s_2} \partial_x (u_n+u_m) \\ & \lesssim \|\langle \partial_y \rangle^{s_2} (u_n-u_m)\|_{L^2}\|u_n-u_m\|_{L^{\infty}} \|\langle \partial_y \rangle^{s_2} \partial_x (u_n+u_m)\|_{L^2}.\end{align*}
We control the last factor at the cost of a growth factor in terms of the initial data regularization: Given any $\varepsilon>0$ we have boundedness $\langle \partial_y \rangle^{s_2} \partial_x: H^{s_1,s_2(1+\varepsilon)}_{\omega} \rightarrow L^2$ provided $s_1$ is sufficiently large depending on $s_2$ and $\varepsilon$. This amounts to the estimate $\langle \xi_2 \rangle^{s_2} \xi_1 \lesssim \langle \xi_1 \rangle^{s_1} + \langle \xi_2 \rangle^{s_2 (1+ \varepsilon)}$ valid due to Young's inequality. By Lemma \ref{normgrowthregularized} and initial data regularization (\ref{initialdata}) we bound \begin{equation*}
    \|\langle \partial_y \rangle^{s_2} \partial_x u_m\|_{L^2} \lesssim \|u_m\|_{H^{(s_1,s_2(1+\varepsilon))}_{\omega}} \lesssim \|u_0^{\delta(m)}\|_{H^{(s_1,s_2(1+\varepsilon))}_{\omega}} \lesssim (1+\delta(m))^{\varepsilon} \|u_0\|_{H^{(s_1,s_2)}_{\omega}}.
\end{equation*}
An analogous estimate holds for $\|\langle \partial_y \rangle^{s_2}\partial_x u_n\|_{L^2}$. The idea is that we can get enough decay from the second factor $\|u_n-u_m\|_{L^{\infty}}$ to outweigh this growth in $\delta(m)$. By anisotropic Sobolev embedding \eqref{anisotropicembedding},  interpolation \eqref{interpolationestimate}, $H^{s_1,s_2}_{\omega}$-bound on regularized solutions (Lemma \ref{Hshbound}), and invoking $L^{\infty}([0,T],L^2)$ Cauchy decay (Lemma \ref{L2Cauchy}), we have
\begin{align}\label{Linfinitydecay}\begin{split}\|u_n-u_m\|_{L^{\infty}} & \lesssim \|u_n-u_m\|_{H^{(s_1/2, 1/4+s_2/2)}_{\omega}} \\ & \lesssim \|u_n-u_m\|_{L^2}^{\theta} \|u_n-u_m\|_{H^{s_1,s_2}_{\omega}}^{1-\theta} \lesssim [Tm^{-q} + \delta(m)^{-1} \cdot o_{m \rightarrow \infty}(1)]^{\theta}.\end{split}\end{align} Here $\theta \in (0,1)$ is an appropriate parameter dictated by interpolation \eqref{interpolationestimate}.
We can take $\theta$ to be non-decreasing in $s_1$. Given $\theta$ choose $0<\varepsilon < \theta$ and $s_1$ sufficiently large depending on $\varepsilon.$ Provided \begin{equation}\label{deltabound1}
    \delta(m) \lesssim m^{q\theta/\varepsilon} \cdot o_{m \rightarrow \infty}(1)
\end{equation}
we have that $[Tm^{-q}+\delta(m)^{-1} o _{m \rightarrow \infty}(1)]^{\theta} (1+\delta(m))^{\varepsilon} = o_{m \rightarrow \infty}(1)$ and thus $$\int_{\mathbb{T}^2} \langle \partial_y \rangle^{s_2}(u_n-u_m)\cdot (u_n-u_m) \langle \partial_y \rangle^{s_2} \partial_x (u_n+u_m) \lesssim \|\langle \partial_y \rangle^{s_2} (u_n-u_m)\|_{L^2} \cdot o_{m \rightarrow \infty}(1).$$
\\

\item Fifth, the contribution by $I$: Due to the frequency support of $I$ being bounded away from zero, we can gain decay $m^{-\mu}$ at the cost of paying extra derivatives $\langle \nabla \rangle^{\mu}$. We estimate
\begin{align*}
    & \quad \int_{\mathbb{T}^2} \langle \partial_y \rangle^{s_2} (u_n-u_m) \langle \partial_y \rangle^{s_2} \partial_x ((P_n - P_m)u_m^2) \\ & \lesssim \|\langle \partial_y \rangle^{s_2} (u_n-u_m)\|_{L^2} \|\langle \nabla \rangle^{\mu} \langle \partial_y \rangle^{s_2} \partial_x ((P_n-P_m)u_m^2)\|_{L^2} \cdot m^{-\mu} \\ & \lesssim \|\langle \partial_y \rangle^{s_2} (u_n-u_m)\|_{L^2} \|u_m\|_{H^{s,1}_{\omega}}^2 \cdot m^{-\mu}
\end{align*}
In the second line, we used boundedness of $\langle \nabla \rangle^{\mu} \langle \partial_y \rangle^{s_2} \partial_x:H^{s_1,1}_{\omega} \rightarrow L^2$, which amounts to \begin{align*}\langle |\xi|\rangle^{\mu} \langle \xi_2 \rangle^{s_2} \xi_1 & \lesssim \langle \xi_1 \rangle^{1+\mu} \langle \xi_2 \rangle^{s_2} \kern-0.3em + \kern-0.2em \langle \xi_1 \rangle \langle \xi_2 \rangle^{s_2+\mu} \kern-0.2em \lesssim \langle \xi_1 \rangle^{\max((1+\mu)/(1-s_2),1/(1-s_2-\mu))} \kern-0.3em + \kern-0.2em \langle \xi_2 \rangle^1 \lesssim \langle \xi_1\rangle^{s_1} \kern-0.3em + \kern-0.2em \langle \xi_2 \rangle^1\end{align*} due to Young's inequality provided \begin{equation*}
    s_1>\max((1+\mu)/(1-s_2),1/(1-s_2-\mu)).
\end{equation*} Further, we used the algebra property of $H^{(s_1,1)}_{\omega}$. Next, by norm growth control (Lemma \ref{normgrowthregularized}) and initial data regularization \eqref{initialdata} \begin{align*}\|u_m(t)\|_{H^{s_1,1}_{\omega}} & \lesssim \|u_0^{\delta(m)}\|_{H^{s_1,1}_{\omega}} \exp \big(\int_0^t \|\partial_x u_m(s)\|_{L^{\infty}} ds \big) \\ & \lesssim (1+\delta(m)^{(1-s_2)/s_2})\|u_0\|_{H^{s_1,s_2}_{\omega}} \lesssim (1+\delta(m)^{(1-s_2)/s_2}).\end{align*}
Thus, once we impose the bound \begin{equation}\label{deltabound2}
    \delta(m) \lesssim m^{s_2\mu /2(1-s_2)} \cdot o_{m \rightarrow \infty}(1)
\end{equation}
on the initial data regularization, we have that 
$$\int_{\mathbb{T}^2} \langle \partial_y \rangle^{s_2} (u_n-u_m) \langle \partial_y \rangle^{s_2} \partial_x ((P_n - P_m)u_m^2) \lesssim \|\langle \partial_y \rangle^{s_2} (u_n-u_m)\|_{L^2} \cdot o_{m \rightarrow \infty}(1).$$
Taking $\delta(m)$ to be an increasing function in $m$ satisfying (\ref{deltabound1}) and (\ref{deltabound2}) completes the proof.
\end{proof}

\begin{remark}The fifth step bounding the contribution by $I$ is the reason why we imposed the condition $s_2<1$ in the hypothesis of Theorem \ref{wellposednessinHs1s2}: we still want to have norm growth control (Lemma \ref{normgrowthregularized}) after applying extra derivatives $\langle \nabla \rangle^{\eta}$.

The point here is that we may gain decay $m^{-\eta}$ in terms of the regularization parameter at the cost of some additional joint derivatives $\langle \nabla \rangle^{\eta}$. Using norm growth control (Lemma \ref{normgrowthregularized}) we can push these additional derivatives on the initial data, where we may eliminate them again at the cost of growth in terms of the initial data regularization $\delta(m)$. 
In order to be able to do this, we need the norm growth control at slightly higher regularity than the regularity of the original space; thus, we impose $s_2<1$.

The purpose of separate initial data regularization is that the decay in terms of the regularization parameter of the equation outweighs the growth in terms of the regularization of the initial data.
\end{remark}

We turn to estimate \eqref{growthxderivatives}. The proof is mostly analogous to the proof of estimate \eqref{growthyderivatives}.

\begin{proof}[Proof of estimate \eqref{growthxderivatives}]
Again, we have by anti-selfadjointness of $-H\partial_{xx}$ that
\begin{equation}\label{eq101}\frac{d}{dt} \|\langle \partial_x \rangle^{s_1} (u_n-u_m)\|^2_{L^2} = \int_{\mathbb{T}^2} \langle \partial_x \rangle^{s_1} (u_n-u_m)\langle \partial_x \rangle^{s_1} \partial_x(P_n(u_n^2) - P_m(u_m^2))\end{equation}
  and we write the contribution from the nonlinearity as a difference of squares, while picking up another term whose frequency support is bounded away from zero: 
$$\langle \partial_x \rangle^{s_1}\partial_x(P_n(u_n^2)-P_m(u_m^2)) = \langle \partial_x \rangle^{s_1}\partial_x((P_n - P_m)u_m^2) + \langle \partial_x \rangle^{s_1}P_n\partial_x[(u_n-u_m)(u_n+u_m)].$$
    When considering the contribution of the latter term to (\ref{eq101}) we can drop the projection $P_n$ by Plancherel and using Leibniz rule and commutators we write 
    \begin{align*}\langle \partial_x \rangle^{s_1} \partial_x [(u_n - u_m)(u_n+u_m)] & =  [\langle \partial_x \rangle^{s_1}, u_n-u_m]\partial_x(u_n+u_m) + (u_n-u_m)\langle \partial_x \rangle^{s_1} \partial_x(u_n+u_m) \\ & + [\langle \partial_x \rangle^{s_1}, u_n+u_m]\partial_x(u_n-u_m) +  (u_n+u_m)\langle \partial_x \rangle^{s_1} \partial_x(u_n-u_m).\end{align*}
    This leaves us with five terms which we estimate now. \\
    
 First, we estimate the contribution of the term whose frequency support is bounded away from zero. This condition on the frequency support allows us to gain decay in terms of the regularization parameter $m^{-\mu}$ at the cost of paying extra derivatives $\langle \nabla \rangle^{\mu}$. We estimate
    \begin{align*}
    \int_{\mathbb{T}^2} \langle \partial_x \rangle^{s_1} (u_n-u_m) \langle \partial_x \rangle^{s_1} \partial_x ((P_m-P_n)u_m^2) & \lesssim \|\langle \partial_x \rangle^{s_1}(u_n-u_m)\|_{L^2} \|\langle \nabla \rangle^{\mu} \langle \partial_x \rangle^{s_1+1} u_m^2\| \cdot m^{-\mu} \\ & \lesssim \|\langle \partial_x \rangle^{s_1}(u_n-u_m)\|_{L^2} \|u_m\|_{H^{2s_1,s_2}_{\omega}}^2 \cdot m^{-\mu} \\ 
   & \lesssim \|\langle \partial_x \rangle^s(u_n-u_m)\|_{L^2} \cdot m^{-\mu}\delta(m)^2.
\end{align*}
In the second line, we used boundedness of $\langle \nabla \rangle^{\mu} \langle \partial_x \rangle^{s_1+1}: H^{2s_1,s_2}_{\omega} \rightarrow L^2$ which amounts to $$\langle |\xi|\rangle^{\mu} \langle \xi_1 \rangle^{s_1+1} \leq \langle \xi_1 \rangle^{s_1+1+\mu} + \langle \xi_1 \rangle^{s_1+1} \langle \xi_2 \rangle^{\mu} \lesssim \langle \xi_1 \rangle^{\max(s_1+1+\mu, (s_1+1)/(1-\mu/s_2))} + \langle \xi_2 \rangle^{s_2} \lesssim \langle \xi_1 \rangle^{2s_1} + \langle \xi_2 \rangle^{s_2}$$ due to Young's inequality, provided $\mu$ is sufficiently small. Also, the algebra property of $H^{2s_s,s_2}_{\omega}$ is used. In the third line we used norm control for $H^{2s_1,s_2}_{\omega}$ and initial data regularization \eqref{initialdata}: $$\|u_m(t)\|_{H^{2s_1,s_2}_{\omega}} \lesssim \|u_0^{\delta(m)}\|_{H^{2s_1,s_2}_{\omega}} \exp (\int_0^t \|\partial_x u_m(s) \|_{L^{\infty}} ds) \lesssim \delta(m) \|u_0\|_{H^{s_1,s_2}_{\omega}} \lesssim \delta(m).$$
Thus, it suffices to impose the growth condition \begin{equation}\label{delta3}
    \delta(m) \lesssim m^{\mu/2} \cdot o_{m \rightarrow \infty}(1)
\end{equation}
on the initial data regularization parameter. \\

\item Second, we estimate by quasi-periodic Kato-Ponce commutator estimate \eqref{commutatorsderivatives} \begin{align*}
    \int_{\mathbb{T}^2} \langle \partial_x \rangle^{s_1} (u_n-u_m) [\langle \partial_x \rangle^{s_1}: u_n-u_m]\partial_x (u_n+u_m) \\ \lesssim \|\langle \partial_x \rangle^{s_1}(u_n-u_m)\|_{L^2} \cdot  \big[ & \|\langle \partial_x \rangle^{s_1} (u_n+u_m)\|_{L^2} \|\partial_x(u_n-u_m)\|_{L^{\infty}} \\ & + (\|\partial_x u_m\|_{L^{\infty}}+\|\partial_xu_n\|_{L^{\infty}}) \|\langle \partial_x \rangle^{s_1} (u_n-u_m)\|_{L^2} \big].
\end{align*}
As in \eqref{partialxinLinfinity} we have
$$\|\partial_x(u_n-u_m)\|_{L^{\infty}} = o_{m \rightarrow \infty}(1),$$
and by Proposition \ref{uniformaprioriestimate} we have 
$$\|\langle \partial_x \rangle^{s_1}u_n(t)\|_{L^2}+\|\langle \partial_x \rangle^{s_1} u_m(t)\|_{L^2} \lesssim \|u_0\|_{H^{s_1,s_2}_{\omega}} \lesssim 1$$ uniformly for initial data from a compact set $\mathcal{K} \subseteq H^{s_1,s_2}_{\omega}$. \\

\item
Third, we bound the contribution of the second to last term
\begin{align*}
    \int_{\mathbb{T}^2} \langle \partial_x \rangle^{s_1} (u_n-u_m) (u_n-u_m)\langle \partial_x \rangle^{s_1} \partial_x (u_n+u_m) \\ \ \ \ \lesssim \|\langle \partial_x \rangle^{s_1} (u_n-u_m)\|_{L^2} \|u_n-u_m\|_{L^{\infty}} [\|\langle \partial_x \rangle^{s_1+1}u_n\|_{L^2} + \|\langle \partial_x \rangle^{s_1+1}u_m\|_{L^2}].
\end{align*}
Using norm growth control (Lemma \ref{normgrowthregularized}) and initial data regularization \eqref{initialdata} we estimate $$\|\langle \partial_x \rangle^{s_1+1}u_m\|_{L^2}  \lesssim  \|\langle \partial_x \rangle^{s_1+1}u_0^{\delta(m)}\|_{L^2} \exp (\int_0^t \|\partial_xu_n(s)\|_{L^{\infty}} ds)  \lesssim \delta(m)^{1/s_1}\|u_0\|_{H^{s_1,s_2}_{\omega}},$$
and analogously
$\|\langle \partial_x \rangle^{s_1+1}u_n\|_{L^2}   \lesssim \delta(n)^{1/s_1}\|u_0\|_{H^{s_1,s_2}_{\omega}}.$
We can get enough decay out of the term $\|u_n-u_m\|_{L^{\infty}}$ to dominate the growth factor $\delta(m)^{1/s_1}$. As in \eqref{Linfinitydecay}, we have 
\begin{align*}\|u_n -u_m\|_{L^{\infty}} & \lesssim [Tm^{-q} + \delta(m)^{-1} o_{m \rightarrow \infty}(1)]^{\theta}\end{align*}
for a suitable $\theta \in (0,1)$ and $0<q<s_2$. Thus, 
\begin{align*}\delta(m)^{1/s_1}\|u_n -u_m\|_{L^{\infty}} & 
\lesssim \delta(m)^{1/s_1} m^{-q\theta}+ \delta(m)^{1/s_1-\theta}o_{m\rightarrow \infty}(1) = o_{m \rightarrow \infty}(1),\end{align*}
provided $s_1$ is sufficiently large and $\delta(m)$ is an increasing function of $m$ satisfying
\begin{equation}\label{delta4} \delta(m) \lesssim m^{s_1q\theta} \cdot o_{m \rightarrow \infty}(1).\end{equation}
\\

\item Fourth, the contribution by $[\langle \partial_x \rangle^{s_1}, u_n+u_m]\partial_x(u_n-u_m)$ can be estimated similarly to the second point. By quasi-periodic Kato-Ponce commutator estimate \eqref{commutatorsderivatives},
\begin{align*}
   \int_{\mathbb{T}^2} \langle \partial_x \rangle^{s_1} (u_n-u_m) [\langle \partial_x \rangle^{s_1}: u_n-u_m]\partial_x (u_n+u_m)  \\ \lesssim   \|\langle \partial_x \rangle^{s_1}(u_n-u_m)\|_{L^2}  \big[ & \|\langle \partial_x \rangle^{s_1} (u_n+u_m)\|_{L^2} \|\partial_x(u_n-u_m)\|_{L^{\infty}} \\ & \  +  (\|\partial_x (u_n)\|_{L^{\infty}} + \|\partial_x (u_m)\|_{L^{\infty}}) \|\langle \partial_x \rangle^{s_1} (u_n-u_m)\|_{L^2} \big]. 
\end{align*}
As in \eqref{partialxinLinfinity} we have
$\|\partial_x(u_n-u_m)\|_{L^{\infty}} = o_{m \rightarrow \infty}(1),$
and by Proposition \ref{uniformaprioriestimate}
$$\|\langle \partial_x \rangle^{s_1}u_n(t)\|_{L^2}+\|\langle \partial_x \rangle^{s_1} u_m(t)\|_{L^2} \lesssim \|u_0\|_{H^{s_1,s_2}_{\omega}} \lesssim 1.$$ 
\\

\item Fifth, the contribution by the last summand can be bounded using integration by parts: 
\begin{align*}\int_{\mathbb{T}^2} \langle \partial_x \rangle^{s_1}(u_n-u_m) (u_n+u_m) \langle \partial_x \rangle^{s_1} \partial_x(u_n-u_m) \lesssim \|\langle \partial_x \rangle^{s_1} (u_n-u_m)\|_{L^2}^2 \|\partial_x(u_n+u_m)\|_{L^{\infty}}.\end{align*}

Take $\delta(m)$ an increasing function of $m$ satisfying estimates \eqref{delta3} and \eqref{delta4}.
Adding the above estimates and dividing by $\|\langle \partial_x \rangle^s(u_n-u_m)\|_{L^2}$ completes the proof of estimate \eqref{growthxderivatives}.
\end{proof}

\subsection{Solution to Benjamin--Ono equation}\label{sectionsolution}

We verify that the limit of the solutions $u_n$ to~the regularized equation, $u:= H^{s_1,s_2}_{\omega}\text{-}\lim_{n \rightarrow \infty} u_n$, is a classical solution to the Benjamin--Ono equation.

\begin{proof} For a function $u \in C([0,T],H^{s_1,s_2}_{\omega})$, being a classical solution is equivalent to requiring that the associated integral equation \begin{equation}\label{solvesintegralequation}
    u(t)=u_0 + \int_0^t -H\partial_{xx}u(s) + \partial_x(u(s))^2 ds
\end{equation}
holds pointwise in space and time, by Lemma \ref{anisotropicproperties}.
    We want to pass to the limit $n\rightarrow \infty$ in \begin{equation}\label{solvesregularizedBO}u_n(t)=u_0^{\delta(n)} + \int_0^t -H\partial_{xx}u_n(s) + P_n\partial_x(u_n(s))^2 ds.\end{equation}
    Convergence of the initial data term is clear. The linear term can easily be handled using convergence $u_n \rightarrow u$ in $C([0,T],H^{s_1,s_2}_{\omega})$, anisotropic Sobolev embedding \eqref{anisotropicembedding}, and boundedness property \eqref{boundedness1} (with $k=2$) in Lemma \ref{anisotropicproperties}.
    This yields $$\int_0^t -H\partial_{xx}u_n(s)ds \rightarrow \int_0^t -H \partial_{xx}u(s)ds \quad \text{pointwise on } [0,T]\times \mathbb{R}.$$
    Next, we consider the nonlinear term and write 
    \begin{align*} & \int_0^t P_n\partial_x(u_n(s))^2ds - \int_0^t \partial_x(u(s)^2)ds  = \int_0^t \partial_x P_n (u_n(s)^2 - u(s)^2)ds  - (1 - P_n)\partial_x\int_0^t u(s)^2 ds. \end{align*}
    We estimate the first term using anisotropic Sobolev embedding, the algebra property of $H^{s_1,s_2}_{\omega}$, and boundedness property \eqref{boundedness1} (with $k=1$) in Lemma \ref{anisotropicproperties},
    yielding \begin{align*} & \big\|\partial_x \big(\int_0^t P_n(u_n(s)^2 - u(s)^2) ds \big)\big\|_{C([0,T],C(\mathbb{T}^2))} \\ \lesssim & \ \big\|P_n \int_0^t(u_n(s)-u(s))(u_n(s)+u(s))ds\big\|_{C([0,T],H^{s_1,s_2}_{\omega})} \\ \lesssim &\ T \|u_n-u\|_{C([0,T],H^{s_1,s_2}_{\omega})} [\|u_n\|_{C([0,T],H^{s_1,s_2}_{\omega})} + \|u\|_{C([0,T],H^{s_1,s_2}_{\omega})}].\end{align*}
    Due to the a-priori bound in Lemma \ref{Hshbound}, the factor $\|u_n\|_{C([0,T],H^{s_1,s_2}_{\omega})} + \|u\|_{C([0,T],H^{s_1,s_2}_{\omega})}$ is bounded uniformly in $n$, so that the overall expression goes to zero as $n \rightarrow \infty$. 
    To estimate the second term, we use again anisotropic Sobolev embedding and boundedness property \eqref{boundedness1} in Lemma \ref{anisotropicproperties}, 
    giving
    \begin{equation}\label{secondterm}\big\|(1-P_n) \partial_x\int_0^t u(s)^2 ds \big\|_{C([0,T] \times \mathbb{T}^2)} \lesssim \|(1-P_n) \int_0^t u(s)^2 ds\|_{C([0,T],H^{s_1,s_2}_{\omega})}.\end{equation}
By the algebra property of $H^{s_1,s_2}_{\omega}$ the function $t \mapsto G(t):= \int_0^t u(s)^2 ds$ is continuous and the image $$\big\{G(t)=\int_0^t u(s)^2 ds \ | \ t\in[0,T]\big\}$$ thus compact and equicontinuous in $H^{s_1,s_2}_{\omega}$. In particular, we have  equitightness on the Fourier side (cf. Riesz-Kolmogorov), that is, the tails on the Fourier side become small uniformly in $t$:
    $$\sup_{t \in [0,T]} \sum_{|n|>r} |\hat{G}(t,n)|^2[ \langle \omega \cdot n \rangle^{2s_1} + \langle \omega^{\perp} \cdot n \rangle^{2s_2}] \rightarrow 0 \text{ as } r \rightarrow \infty.$$
    Thus, the expression in (\ref{secondterm}) goes to zero as $n \rightarrow \infty$. Taking $n \rightarrow \infty$ in (\ref{solvesregularizedBO}) thus yields (\ref{solvesintegralequation}) pointwise in space and time.
\end{proof}

\subsection{Uniqueness}\label{sectionuniqueness} We show that the solution is unique in $C([0,T],H^{s_1,s_2}_{\omega})$.
Suppose that $u, v \in C([0,T],H^{s_1,s_2}_{\omega})$ are solutions to the Benjamin--Ono equation. Then $w=u-v$ satisfies the equation $$w_t = -Hw_{xx} + \partial_x(w(u+v))$$
so that using integration by parts $$\frac{d}{dt} \|w\|^2_{L^2} \leq \frac{1}{2} \|w\|_{L^2}^2 [\|\partial_xu\|_{L^{\infty}} + \|\partial_xv\|_{L^{\infty}}].$$
Gronwall's inequality \eqref{Gronwall} and estimate (\ref{crucialequation}) yields $$\|w\|_{C([0,T],L^2)} \lesssim \|w(0)\|_{L^2} \exp \big(\frac{1}{2}T[\|u\|_{L^{\infty}([0,T],H^{s_1,s_2}_{\omega})} + \|v\|_{L^{\infty}([0,T],H^{s_1,s_2}_{\omega})}]\big) = 0.$$
This gives equality almost everywhere on every time slice and by continuity of $u(t),v(t) \in H^{s_1,s_2}_{\omega} \subseteq C(\mathbb{T}^2)$ (Lemma \ref{anisotropicproperties}) thus equality pointwise.

\subsection{Continuity of data-to-solution map}\label{sectioncontinuitydatatosolution}
We show continuity of the data-to-solution map as a mapping $H^{s_1,s_2}_{\omega} \rightarrow C([0,T],H^{s_1,s_2}_{\omega})$.

Suppose $u_0^k \rightarrow u_0$ in $H^{s_1,s_2}_{\omega}$. Denote by $u_n^k,u_n$ the emanating solutions to the regularized Benjamin Ono equation (\ref{regularizedBO}) with regularized intial data (\ref{initialdata}), and by $u^k,u$ the emanating solutions to the actual Benjamin--Ono equation that we just constructed. By triangle inequality,
\begin{equation}\label{continuitydatatosolution}
    \|u - u^k\|_{L^{\infty}H^{s_1,s_2}_{\omega}} \leq \|u-u_n\|_{L^{\infty}H^{s_1,s_2}_{\omega}} + \|u_n-u^k_n\|_{L^{\infty}H^{s_1,s_2}_{\omega}} +\|u^k_n - u^k\|_{L^{\infty}H^{s_1,s_2}_{\omega}}.
\end{equation}
Note that $\|u^k_n - u^k\|_{L^{\infty}H^{s_1,s_2}_{\omega}} \rightarrow 0$ as $n \rightarrow \infty$ uniformly in $k$, as the $L^{\infty}([0,T],H^{s_1,s_2}_{\omega})$ Cauchyness in Lemma \ref{HshCauchy} is uniform for initial data from compact set. By taking $n$ large, we can thus make the first and third term in (\ref{continuitydatatosolution}) arbitrarily small uniformly in $k$. Now, for fixed $n$, the second term $\|u_n - u_n^k\|_{L^{\infty}H^{s_1,s_2}_{\omega}}$ goes to zero as $k \rightarrow \infty$ due to continuity of the data to solution map for the regularized Benjamin--Ono equation. This completes the proof that $$\|u-u^k\|_{L^{\infty}([0,T],H^{s_1,s_2}_{\omega})} \rightarrow 0 \quad \text{as } k \rightarrow \infty.$$ 

Subsections \ref{regularization}-\ref{sectioncontinuitydatatosolution} together establish Theorem \ref{wellposednessinHs1s2}. The same arguments apply verbatim to prove Theorem \ref{KdVandNLS} upon replacing the linear term $H\partial_{x}^2$ by $\partial_x^3$ or $-i\partial_x^2$, respectively.

\subsection{Wellposedness in $X$}\label{sectionwellposednessinX}
In this section, we prove Theorem \ref{wellposednessinX}. 
It suffices to show that the sequence $\{u_n\}_{n \in \mathbb{N}}$ of emanating solutions is Cauchy with respect to the contribution of $\|\langle \partial_y \rangle^q\partial_x^{-1}u\|_{L^2}$ to the $X$-norm, so that the limit $u=\lim_{n \rightarrow \infty} u_n$ is in fact in $C([0,T],X)$. By section \ref{sectionsolution}, this is still a solution in the sense that the integral equation (\ref{solvesintegralequation}) holds pointwise in space and time, and uniqueness of solutions is clear as $C([0,T],X) \hookrightarrow C([0,T],H^{s_1,s_2}_{\omega})$. The proof of continuity of the data-to-solution map in section \ref{sectioncontinuitydatatosolution} adapts without difficulty, changing the $H^{s_1,s_2}_{\omega}$-norm to $X$-norm at appropriate places.

\begin{lemma}
    We have $$\frac{d}{dt} \|\langle \partial_y \rangle^{s_2} \partial_x^{-1}(u_n-u_m)\|_{L^2} = o_{m \rightarrow \infty}(1).$$
    Thus, $$\|\langle \partial_y\rangle^{s_2} \partial_x^{-1}(u_n-u_m)\|_{L^2} \leq \|\langle \partial_y\rangle^{s_2} \partial_x^{-1}(u_0^{\delta(n)}-u_0^{\delta(m)})\|_{L^2} +t \cdot o_{m \rightarrow \infty}(1)$$
    goes to zero uniformly in $t\in[0,T]$ and uniformly for initial data $u_0$ from a given compact set $\mathcal{K} \subseteq X$ as $m \rightarrow \infty$. Hence, the function $u:= \lim_{n \rightarrow \infty} u_n$ is in $C([0,T],X)$.
\end{lemma}

\begin{proof}
    By anti-selfadjointness of $-H\partial_{xx}$ we have $$\frac{d}{dt} \|\langle \partial_y \rangle^{s_2} \partial_x^{-1}(u_n-u_m)\|_{L^2} = \int_{\mathbb{T}^2} \langle \partial_y \rangle^{s_2} \partial_x^{-1}(u_n-u_m) \langle \partial_y \rangle^{s_2}[P_n(u_n^2) - P_m(u_m^2) -(\widehat{u_n^2}(0) - \widehat{u_m^2}(0))].$$ We pick up the term  $\widehat{u_n^2}(0) - \widehat{u_m^2}(0)$ due to the identity $\partial_x^{-1} \circ \partial_x g = \partial_x \circ \partial_x^{-1}g = g  - \intavg g$. Analogously to the proofs in section \ref{sectionCauchybound}, we rearrange
    $$\langle\partial_y\rangle^{s_2} [P_n(u_n^2) - P_m(u_m^2)] = \langle \partial_y \rangle^{s_2}(P_n-P_m)u_m^2 + \langle \partial_y \rangle^{s_2}P_n((u_n-u_m)(u_n+u_m)).$$
    We estimate: \\
    
    First, $$\int_{\mathbb{T}^2} \langle \partial_y \rangle^{s_2} \partial_x^{-1}(u_n-u_m) \cdot ((\widehat{u_n^2-u_m^2})(0))\lesssim \|\langle \partial_y \rangle^{s_2} \partial_x^{-1}(u_n-u_m)\|_{L^2} \big|(\widehat{u_n^2-u_m^2})(0) \big|$$
        Now $$\big|(\widehat{u_n^2-u_m^2})(0) \big| = \int_{\mathbb{T}^2}u_n^2 -u_m^2 \lesssim \|u_n-u_m\|_{L^2}[\|u_n\|_{L^2}+\|u_m\|_{L^2}] = o_{m\rightarrow \infty}(1),$$
        where we used Cauchy property in $C([0,T],L^2)$ (Lemma \ref{L2Cauchy}) and the uniform bound (Lemma \ref{Hshbound}) in the last line. \\

     Second, \begin{align*}
            \int_{\mathbb{T}^2} \langle \partial_y \rangle^{s_2} \partial_x^{-1}(u_n-u_m) \langle \partial_y \rangle^{s_2}[(P_n-P_m)u_m^2] & \lesssim \|\langle \partial_y \rangle^{s_2} \partial_x^{-1}(u_n-u_m)\|_{L^2} \|\langle \partial_y \rangle^{s_2} (P_n-P_m) u_m^2\|_{L^2}. 
        \end{align*}
        Now \begin{align*}\|\langle \partial_y \rangle^{s_2} (P_n-P_m) u_m^2\|_{L^2} & \lesssim \|\langle \nabla \rangle^{\mu}\langle \partial_y \rangle^{s_2} u_m^2\|_{L^2} \cdot m^{-\mu}  \\ &\lesssim \|u_m\|_{H^{s_1,s_2+\mu}_{\omega}} \|u_m\|_{L^{\infty}} \cdot m^{-\mu}  \\ & \lesssim m^{-\mu}\delta(m)^{\mu /s_2} \|u_0^{\delta(m)}\|_{H^{s_1,s_2}_{\omega}} \|u_m\|_{H^{s_1,s_2}_{\omega}}  = o_{m\rightarrow \infty}(1)\end{align*} provided that \begin{equation}\delta(m) \lesssim m^{s_2} \cdot o_{m\rightarrow \infty}(1).\end{equation}
        In the second line we used boundedness of $\langle \nabla \rangle^{\eta} \langle \partial_y \rangle^{s_2}:H^{s_1,s_2+\mu}_{\omega}\rightarrow L^2$ due to $$\langle |\xi| \rangle^{\mu}\langle \xi_2 \rangle^{s_2} \leq \langle \xi_1 \rangle^{\mu } \langle \xi_2 \rangle^{s_2} + \langle \xi_2 \rangle^{\mu+s_2}\lesssim \langle \xi_1 \rangle^{\mu +s_2} + \langle \xi_2 \rangle^{\mu +s_2} \lesssim \langle \xi_1 \rangle^{s_1} + \langle \xi_2 \rangle^{\mu +s_2}$$ by Young's inequality provided $\mu +s_2 \leq s_1$, and the algebra property \eqref{algebraproperty}. Further, we used control on norm growth (Lemma \ref{normgrowthregularized}), initial data regularization \eqref{initialdata}, anisotropic Sobolev embedding (Lemma \ref{anisotropicproperties}), and uniform $H^{s_1,s_2}_{\omega}$ bound (Lemma \ref{Hshbound}). \\

 Third, \begin{align*} & \int_{\mathbb{T}^2} \langle \partial_y \rangle^{s_2} \partial_x^{-1}(u_n - u_m) \langle \partial_y \rangle^{s_2}P_n((u_n-u_m)(u_n+u_m)) \\ &  \lesssim \  \|\langle \partial_y \rangle^{s_2} \partial_x^{-1}(u_n - u_m)\|_{L^2}  \|\langle \partial_y \rangle^{s_2}P_n((u_n-u_m)(u_n+u_m))\|_{L^2}.\end{align*}
        By fractional product rule, the second factor is bounded by 
        \begin{align*}
            & \|\langle \partial_y \rangle^{s_2}(u_n-u_m)\|_{L^2} (\|u_n\|_{L^{\infty}}+\|u_m\|_{L^{\infty}}) + \|u_n-u_m\|_{L^{\infty}}(\|\langle \partial_y \rangle^{s_2}u_n\|_{L^2}+\|\langle \partial_y \rangle^{s_2}u_m\|_{L^2})
            \end{align*}
        which itself is bounded by   
        \begin{align*} \|u_n-u_m\|_{H^{s_1,s_2}_{\omega}}(\|u_n\|_{H^{s_1,s_2}_{\omega}}+\|u_m\|_{H^{s_1,s_2}_{\omega}}) = o_{m\rightarrow \infty}(1)
        \end{align*}
        due to anisotropic Sobolev embedding (Lemma \ref{anisotropicproperties}), $L^{\infty}([0,T],H^{s_1,s_2}_{\omega})$ Cauchy bound (Lemma \ref{HshCauchy}) and uniform $H^{s_1,s_2}_{\omega}$ bound (Lemma \ref{Hshbound}).
\end{proof}

\subsection{Norm growth control}\label{sectionnormgrowthcontrol} 
Here, we prove an estimate controlling the growth of anisotropic Sobolev norms of solutions to the Benjamin--Ono equation (Lemma \ref{Hshnormgrowth}). We prove Proposition \ref{normgrowthX}.

Taking the limit $n\rightarrow \infty$ in Lemma \ref{normgrowthregularized} about norm growth control for the regularized equation yields the following:

\begin{lemma}\label{Hshnormgrowth}
    Suppose $s_1,s_2$ satisfy the hypothesis for the local wellposedness theory in $H^{s_1,s_2}_{\omega}$ and suppose further that $\tilde{s_1} \geq 0, 1 \geq \tilde{s_2} \geq 0$ are given. The solution to the Benjamin--Ono equation given by the local wellposedness theory in $H^{s_1,s_2}_{\omega}$ obeys the following growth bound:
    \begin{equation}
        \|u(t)\|_{H^{\tilde{s}_1,\tilde{s}_2}_{\omega}} \leq \|u_0\|_{H^{\tilde{s}_1,\tilde{s}_2}_{\omega}} \exp \big( \int_0^t \|\partial_x u(s)\|_{L^{\infty}} ds \big).
    \end{equation}
\end{lemma}

Taking $\tilde{s}_1=s_1$ and $\tilde{s}_2=s_2$ gives Proposition \ref{Hnormgrowth}. We will also need to have control over the growth of the contribution by $\|\langle \partial_y \rangle^q \partial_x^{-1}u\|_{L^2}$ to the $X$-norm.

\begin{lemma}\label{Fyqnormgrowth}
    Suppose $s_1,s_2$ satisfy the hypothesis for the local wellposedness theory in $H^{s_1,s_2}_{\omega}$ and $0 \leq q \leq s_2$. The solution to the Benjamin--Ono equation given by the wellposedness theory in $X$ satisfies the following growth bound:
   \begin{align}
   \begin{split}
\|\langle \partial_y \rangle^q \partial_x^{-1}u(t)\|_{L^2} + \|\langle \partial_y \rangle^q u(t)\|_{L^2} \leq \ & [\|\langle \partial_y \rangle^q \partial_x^{-1}u_0\|_{L^2} + \|\langle \partial_y \rangle^q u_0\|_{L^2}] \\  \cdot & \exp \bigg(\int_0^t \|\partial_xu(s)\|_{L^{\infty}} + \|u(s)\|_{L^{\infty}} ds \bigg).
\end{split}
   \end{align}
\end{lemma}

\begin{proof}[Proof of Lemma \ref{Fyqnormgrowth}]
    Denote $\|u_n\|_{\tilde{X}}:=\|\langle \partial_y \rangle^q \partial_x^{-1}u_n\|_{L^2} + \|\langle \partial_y \rangle^q u_n\|_{L^2}$. Analogous to equation (\ref{yhderivative}), we have
    \begin{equation*}
        \frac{d}{dt} \|\langle \partial_y \rangle^q u_n\|_{L^2}^2 \lesssim \|\langle \partial_y \rangle^q u_n\|_{L^2}^2 \|\partial_x u_n\|_{L^{\infty}} \leq \|u_n\|_{\tilde{X}}\|\partial_x u_n\|_{L^{\infty}}.
    \end{equation*}
  We calculate further
    \begin{align*}
        \frac{d}{dt}\|\langle \partial_y \rangle^q \partial_x^{-1}u_n\|_{L^2}^2 & = \int_{\mathbb{T}^2} \langle \partial_y \rangle^q \partial_x^{-1}u_n \langle \partial_y \rangle \big(u_n^2 - \intavg u_n^2 \big) \\ & \leq \|\langle \partial_y \rangle^q \partial_x^{-1} u_n\|_{L^2} \bigg[ \|\langle \partial_y \rangle^q u_n\|_{L^2} \|u_n\|_{L^{\infty}} + \bigg| \intavg u_n^2 \bigg| \bigg] \\ & \leq \|\langle \partial_y \rangle^q \partial_x^{-1} u_n\|_{L^2} \|\langle \partial_y \rangle^q u_n\|_{L^2} \|u_n\|_{L^{\infty}} \leq  \|u_n\|_{\tilde{X}}^2 \|u_n\|_{L^{\infty}}.
    \end{align*}
    The result follows by adding these two inequalities, invoking Gronwall, and then taking $n\rightarrow \infty$.
\end{proof}

Combining Lemma \ref{Hshnormgrowth} and Lemma \ref{Fyqnormgrowth} yields Proposition \ref{normgrowthX}.

\section{a-priori bound}\label{sectionaprioriproposition}
The entirety of this section is dedicated to the proof of Theorem \ref{aprioriformal}. 
For initial data satisfying $\|u_0\|_Y \leq \rho_0$, this a-priori estimate, together with Proposition \ref{normgrowthX}, will allow us to iterate the local wellposedness theory in $X$ on the time interval $[0,1]$.
A heuristic outline of the proof of Theorem \ref{aprioriformal} was given in the introduction of the paper.  \\

The proof proceeds by a bootstrap argument. The bootstrap hypothesis and conclusion each~split into two parts: the first part $\RomanI$ \eqref{I} consists of estimates on the gauge transformed function $w$, the second part $\RomanII$ \eqref{II} consists of estimates on the derivative $F_{xx}=u_x$ of the solution to Benjamin--Ono. 
In Section \ref{setup} we formally setup the bootstrap argument. In Section \ref{initialzebootstrap} we initialize the bootstrap.
In Sections \ref{partIbootstrap1}-\ref{paraproductestimatessection} we show that the bootstrap conclusion holds for part $\RomanI$. The main tool here is Schr\"odinger type equations for the gauge transformed function $w$ and derivatives thereof (equations (\ref{wequationsecondoccurence}), (\ref{wxequation})),
which are utilized using Strichartz estimates, see Section \ref{partIbootstrap1}.
This leaves us with various terms to be estimated, which is done in Sections \ref{partIbootstrap3} and \ref{sectionenergyestimatepartI}. 
In Section \ref{paraproductestimatessection}, we prove two paraproduct estimates required for doing so.
In Sections \ref{partIIbootstrap} and \ref{sectionerrorterms}, we show that the bootstrap conclusion holds for part $\RomanII$. The main tool here is equations (\ref{Fxequation}) and (\ref{Fxxequation}), obtained by differentiating the defining equation (\ref{gaugesecondappearance}) of the gauge transform and rearranging, thus producing an expression for $P_{+hi}(F_{xx})$  in terms of derivatives of the gauge transform $w$ and some error terms.

\subsection{Setup of the bootstrap argument}\label{setup} 
In this section, we set up the bootstrap argument that will give the desired bound on the quantity $\|\partial_xu\|_{L^4([0,T],L^{\infty}(\mathbb{T}^2))}$ appearing in the conclusion of Theorem \ref{aprioriformal}. Control over the other quantity in Theorem \ref{aprioriformal}, $\|u\|_{L^4([0,T],L^{\infty}(\mathbb{T}^2))}$, then follows by energy arguments. 
\\

First, note that it suffices to prove Theorem \ref{aprioriformal} for solutions $u \in C([0,T],X_{\omega}^{s_1,s_2})$ to the Benjamin--Ono equation \textit{with zero mean}, 
due to the Galilean symmetry $v(t,x)=u(t,x-ta)-a$ of the equation.
Eliminating the zero Fourier mode can never increase the $Y$-norm. 

Given a mean-zero solution $u \in C([0,T],X_{\omega}^{s_1,s_2})$ to the Benjamin--Ono equation, define the function $F:[0,T]\times \mathbb{T}^2 \rightarrow \mathbb{R}$ by
\begin{equation}\label{definitionF}\hat{F}(t,n) = 
      \frac{1}{i(\omega\cdot n)} \hat{u}(t,n) 1_{n \neq 0}(n) + 
      t \cdot \|u_0\|^2_{L^2(\mathbb{T}^2)}  1_{n=0}(n),\end{equation}
and define \begin{equation}\label{gaugesecondappearance}w:=P_{+hi}(e^{-iF}).\end{equation} 
Note that $F$ and $w$ are in $C([0,T],H^{s_1,s_2}_{\omega})$. The statement about $F$ is due to $$\|u\|_{X_{\omega}^{s_1,s_2}} \sim \|\partial_x^{-1}u\|_{H_{\omega}^{s_1,s_2}} + \|u\|_{H_{\omega}^{s_1,s_2}},$$ and the statemant about $w$ follows by the algebra property of $H^{s_1,s_2}_{\omega}$ allowing us to exponentiate. 
The function $F$ is real-valued, as $u$ is real-valued and the symbol $m$ satisfies $m(-n) = \overline{m(n)}$. 
Using the assumption that $u_0$ has vanishing mean, and conservation of the mean $\hat{u}(t,0)=\int_{\mathbb{T}^2}u(t)$, the function $F$ is a primitive of $u$ with respect to tangential differentiation $\partial_x$. That is, $F_x=u$. \\

Fix $\eta>0$ small enough so that $\sigma - \eta > \frac{7}{8}$ and $p<\infty$ large enough so that 
Sobolev embedding $$H^{\eta,p}(\mathbb{T}^2) \hookrightarrow L^{\infty}(\mathbb{T}^2)$$ 
holds. 
Provided $\|u_0\|_Y \leq \rho$, we will show that
$$\|\langle \nabla \rangle^{\eta}F_{xx}\|_{L^4([0,T],L^p(\mathbb{T}^2))} \lesssim \rho,$$ from which the result follows.
Indeed, it is then clear that $\|\partial_xu\|_{L^4([0,T],L^{\infty}(\mathbb{T}^2))}\lesssim \rho$ from Sobolev embedding, and the other part of the conclusion of Theorem \ref{aprioriformal} follows by Lemma \ref{Hshnormgrowth} and the embedding \eqref{anisotropicembedding}: \begin{equation*}\|u\|_{L^4([0,T],L^{\infty}(\mathbb{T}^2))} \lesssim T^{1/4}\|u\|_{L^{\infty}([0,T],H_{\omega}^{1+\sigma,\sigma})} \lesssim \|u_0\|_{H^{1+\sigma,\sigma}_{\omega}}\exp \big(\int_0^T \|\partial_xu(s)\|_{L^{\infty}} ds \big)\lesssim \|u_0\|_Y \lesssim \rho.\end{equation*} 
The reason for working with the space $H^{\eta,p}(\mathbb{T}^2)$ instead of $L^{\infty}(\mathbb{T}^2)$ is that the Littlewood-Paley projections are bounded on the former space. This will be important in Section \ref{sectionerrorterms}, see Remark \ref{whyLp}. \\

The proof proceeds via a bootstrap argument; denoting $\varepsilon = \sqrt{\rho}$, and denoting for $0<T' \leq T$ 
\begin{equation}\label{I}\RomanI(T'):= \|\langle \nabla \rangle^{\eta}w\|_{S([0,T'])} + \|\langle \nabla \rangle^{\eta}w_x\|_{S([0,T'])} + \|\langle \nabla \rangle^{\eta}w_{xx}\|_{S([0,T'])}\end{equation} and \begin{equation}\label{II}\RomanII(T'):= \|\langle \nabla \rangle^{\eta} F_{xx}\|_{L^{4}([0,T'],L^p)} + \|\langle \nabla \rangle^{r+\eta} F_{xx}\|_{L^{4/3}([0,T'],L^2)},\end{equation}
the bootstrap hypothesis is
\begin{equation}\label{H}
    \RomanI (T') \lesssim \varepsilon \ \text{and} \ \RomanII(T')\lesssim \varepsilon
\end{equation}
and the bootstrap conclusion is 
\begin{equation}\label{C}
    \RomanI(T') \lesssim \varepsilon^2 \ \text{and} \ \RomanII(T') \lesssim \varepsilon^2.
\end{equation}
Here $r$ is such that $7/8 < r < \sigma - \eta$. The functions $F$ and $w$ being in $C([0,T],H^{s_1,s_2}_{\omega})$ provides enough regularity so that the expressions $\RomanI$ and $\RomanII$ are well-defined, see Lemma \ref{anisotropicproperties} (cf. Lemma \ref{initializebootstraplemma}).
\\

Denote
\begin{equation*}
    \Omega_1 := \{0<T' \leq T \ \text{such that \eqref{H} holds for norms defined on time slab} \ [0,T']\}
\end{equation*}
and
\begin{equation*}
    \Omega_2 := \{0<T' \leq T \ \text{such that \eqref{C} holds for norms defined on time slab} \ [0,T']\}.
\end{equation*}
The continuity of $u$ as map $[0,T] \rightarrow X$ suffices to verify that $\Omega_2$ is right-closed, $\Omega_1$ is non-empty, and that if $T' \in \Omega_2$ then $\Omega_1$ contains an open neighbourhood of $T'$. This is done in Section \ref{initialzebootstrap}. Afterwards, we turn to the main task, which is to show that $\eqref{H} \implies \eqref{C}$, so that the bootstrap argument closes. It then follows that $\Omega_2$ is the entire interval $(0,T]$, completing the proof.

\subsection{Initialzing the bootstrap}\label{initialzebootstrap}
We use continuity of $u$ as a map $[0,T] \rightarrow X$ to initialize the bootstrap argument.

\begin{lemma}\label{initializebootstraplemma}
    Suppose $u:[0,T] \rightarrow X$ is continuous. Then $$\Phi:T' \mapsto \RomanI(T')+\RomanII(T')$$ is a continuous and non-decreasing function $(0,T] \rightarrow \mathbb{R}$. Further,
    $$\lim_{T' \rightarrow 0^+}\RomanI(T')+\RomanII(T')=\|\langle \nabla \rangle^{\eta}\langle \partial_x \rangle^2 w(0)\|_{L^2} \lesssim \varepsilon^2.$$
      In particular,
      \begin{itemize}
      \setlength{\itemindent}{+.2in}
        \item $\Omega_2$ is right-closed
        \item $\Omega_2$ is contained in the interior of $\Omega_1$
        \item $\Omega_1$ is non-empty.
      \end{itemize}
\end{lemma}

It thus suffices to show that the bootstrap hypothesis $\eqref{H}$ implies the bootstrap conclusion $\eqref{C}$ in order to conclude that $\Omega_2=(0,T]$ by connectedness.

\begin{proof} 
The functions $\langle \nabla \rangle^{\eta}\partial_x^jw, \ j=0,1,2$ and $\langle \nabla \rangle^{\eta}F_{xx}$ are all $C([0,T],L^{\infty}(\mathbb{T}^2))$. This follows from the fact that $F$ and $w$ belong to $C([0,T],H^{s_1,s_2}_{\omega})$, by boundedness properties of the derivative operators $\langle \nabla \rangle^{\eta}\partial_x^j, \ j=0,1,2: H^{s_1,s_2}_{\omega} \rightarrow H_{\omega}^{s_1',s_2'}$ for suitable $s_1',s_2'$, and anisotropic Sobolev embedding, see \eqref{anisotropicembedding}, \eqref{boundedness1}, \eqref{boundedness2} in Lemma \ref{anisotropicproperties}.

Now the statement about $\Phi$ being continuous and non-decreasing is clear. The right-limit $\lim_{T' \rightarrow 0^+}\RomanI(T')+\RomanII(T')=\|\langle \nabla \rangle^{\eta}\langle \partial_x \rangle^2 w(0)\|_{L^2}$ is now clear as well. This quantity is shown to be $O(\varepsilon^2)$ in Lemma \ref{initialdatawbound} below. The first two bullet points follow from $\Phi$ being continuous and non-decreasing. 
The stated right limit gives the third bullet point.
\end{proof}

\subsection{Part I of the bootstrap: equation for the gauge transformed function}\label{partIbootstrap1}
In this~\mbox{section}, we use a Schr\"odinger-type equation for the gauge transformed function $w$ and Strichartz estimates to establish a bound on part $\RomanI$ \eqref{I} of the bootstrap conclusion. This is the content of Lemma \ref{Strichartzappliedtow}. \\

We disregard for the moment questions about the regularity of the functions involved and calculate formally. Suppose $u$ is a mean-zero solution to the Benjamin--Ono equation \begin{equation}\tag{BO} u_t + Hu_{xx} = \partial_x(u^2). \end{equation} Applying the multiplier operator $\partial_x^{-1}$ with symbol $m(n):=\frac{1}{i \langle n, \omega \rangle} 1_{\{n \neq 0\}}$ to the equation and using conservation of the quantity $\int_{\mathbb{T}^2}u^2 = \|u\|_{L^2}^2$ under the Benjamin--Ono flow, we get that the function $F$ defined by (\ref{definitionF}) satisfies the equation
\begin{equation*}\label{dynamicalequationF}F_t + HF_{xx}=F_x^2 = u^2.\end{equation*}

 Calculating formally as in \cite{MR2052470} gives the following equation:
\begin{align}\label{wequationsecondoccurence}
    w_t + Hw_{xx}  =w_t - iw_{xx} \nonumber 
    & = P_{+hi}((-iF_t +iF_x^2 -F_{xx})e^{-iF}) \nonumber \\
    & = P_{+hi}((iHF_{xx} - F_{xx})e^{-iF})\\
    & = P_{+hi}((-2P_-F_{xx})e^{-iF}) \nonumber \\ &
    = -2\underbrace{P_{+hi}(P_-(F_{xx})w)}_{=:A} -2\underbrace{P_{+hi}(P_-(F_{xx})P_{lo}(e^{-iF}))}_{=:B}. \nonumber
\end{align}
We obtain Schr\"odinger-type equations for $\langle \nabla \rangle^{\eta}\partial_x^jw$, $j=0,1,2$ by applying the operators $\langle \nabla \rangle^{\eta}\partial_x^j$:
\begin{equation}\label{wxequation}
    (\partial_t - i\partial_{xx})\langle \nabla \rangle^{\eta}\partial_x^jw = -2\langle \nabla \rangle^{\eta}\partial_x^jA -2 \langle \nabla \rangle^{\eta}\partial_x^jB.
\end{equation}
\\

Recall the quasi-periodic
Strichartz estimate (\ref{Strichartz1}) from the introduction, which incurs a derivative loss of at least a quarter.

\begin{lemma}\label{Strichartzlemma}
Let $\omega \in \mathbb{R}^2$ be a frequency vector with incommensurable entries. We have \begin{equation}\label{linearStrichartzestimate}\|e^{it\partial_{xx}}q_0\|_{L_t^4([0,T'],L_{}^4(\mathbb{T}^2))} \lesssim (T')^{1/8} \|\langle \nabla \rangle^{\kappa}q_0\|_{L^2_{}(\mathbb{T}^2)} \quad \text{for} \ \kappa>\frac{1}{4}, \ T' \in (0,1].\end{equation}
\end{lemma}
This was shown in \cite[Theorem 1.1]{MR4918730} via $l^2$-decoupling. For further discussion of Strichartz estimates in the quasi-periodic setting, we refer the reader to that paper. 
This Strichartz estimate yields bounds on sufficiently regular solutions of the Schr\"odinger-type equation \begin{equation}\label{nonlinearSchroedingertype}
    q_t - iq_{xx} = N, \quad q(0)=q_0
\end{equation}
by the usual argument employing the $TT^*$-method with adjoint identity \begin{equation}\label{adjointidentity}\langle e^{is\partial_{xx}}q_0, N \rangle_{[0,T']\times \mathbb{T}^2} = \Big\langle q_0, \int_0^{T'} e^{-is\partial_{xx}}N(s)ds \Big\rangle_{\mathbb{T}^2},\end{equation}
using the Christ-Kiselev lemma, and the Duhamel formulation of the equation \begin{equation}\label{Duhamelq}q(t)= e^{it\partial_{xx}}q_0 + \int_0^t e^{i(t-s)\partial_{xx}}N(s)ds.\end{equation}

Indeed, by duality\footnote{In our setting where we put the non-homogeneity in the $L_t^1([0,T],L^2)$-norm, this also follows simply by triangle inequality for the integral.}, using the adjoint identity \eqref{adjointidentity}, and the fact the propagator $e^{is\partial_{xx}}$ acts by isometry on square-integrable functions, we have
\begin{align}\label{dualStrichartz}\Big\|\int_0^{T'} e^{-is\partial_{xx}}N(s)ds\Big\|_{L^2(\mathbb{T}^2)} 
&= \sup_{\|q_0\|_{L^2}=1} \big\langle e^{is\partial_{xx}}q_0, N\rangle_{[0,T']\times \mathbb{T}^2} \leq \|N\|_{L^1([0,T'],L^2(\mathbb{T}^2))}. \end{align}
Invoking the Strichartz estimate in Lemma \ref{Strichartzlemma} and combining with Sobolev embedding $H^{h,4}(\mathbb{T}^2) \hookrightarrow L^{\infty}(\mathbb{T}^2)$ for $h>1/2$, we thus get the inhomogeneous Strichartz estimate
\begin{equation*}
    \Big\| \int_0^{T'} e^{i(t-s)\partial_{xx}}N(s)ds \Big\|_{L^4_t([0,T'],L^{\infty}(\mathbb{T}^2))} \lesssim (T')^{1/8} \|\langle \nabla \rangle^rN\|_{L^1([0,T'],L^2(\mathbb{T}^2))} \quad \text{for} \quad r=h+\kappa>3/4.
\end{equation*}
By the arguments proving the Christ-Kiselev lemma \cite{MR1809116}, this yields the retarded Strichartz estimate
\begin{equation*}
    \Big\| \int_0^t e^{i(t-s)\partial_{xx}}N(s)ds \Big\|_{L^4_t([0,T'],L^{\infty}(\mathbb{T}^2))} \lesssim (T')^{1/8} \|\langle \nabla \rangle^rN\|_{L^1([0,T'],L^2(\mathbb{T}^2))} \ \text{for} \ r>3/4.
\end{equation*}
Plugging all this into the Duhamel formula \ref{Duhamelq}, we obtain the following estimate on solutions to equation \eqref{nonlinearSchroedingertype}:
\begin{equation}\label{Schroedingertypeequationestimate}\|q\|_{L^4([0,T'],L^{\infty}(\mathbb{T}^2))} \leq (T')^{1/8}\|\langle \nabla \rangle^{r} q_0\|_{L^2(\mathbb{T}^2)} + (T')^{1/8}\|\langle \nabla \rangle^{r} N\|_{L^1([0,T'],L^2(\mathbb{T}^2))} \quad \text{for} \ r>3/4.\end{equation}
Further,
\begin{equation}\label{Schroedingertypeequationestimate2}\|q\|_{L^{\infty}([0,T'],L^{2}(\mathbb{T}^2))} \leq \| q_0\|_{L^2(\mathbb{T}^2)} + \| N\|_{L^1([0,T'],L^2(\mathbb{T}^2))} \end{equation}
directly from Duhamel formula \eqref{Duhamelq}, triangle inequality and $L^2$-isometry of the Schr\"odinger propagator. \\

In Lemmas \ref{enoughregularity} and \ref{retardedStrichartz}
we show that the wellposedness theory in $X$ provides enough regularity to justify these calculations when $q= \langle \nabla \rangle^{\eta}\partial_x^jw$ and $N=-2\langle \nabla \rangle^{\eta}\partial_x^jA-2\langle \nabla \rangle^{\eta}\partial_x^jB$, $j=0,1,2$. Applying (\ref{Schroedingertypeequationestimate}) and \eqref{Schroedingertypeequationestimate2}
to the functions $q= \langle \nabla \rangle^{\eta}\partial_x^jw$, $j=0,1,2$ satisfying equation (\ref{wxequation}), yields Lemma \ref{Strichartzappliedtow} below. We will use this estimate in Sections \ref{partIbootstrap3} and \ref{sectionenergyestimatepartI} 
to get bounds on part $\RomanI$ of the bootstrap conclusion \eqref{C}.
\begin{lemma}\label{Strichartzappliedtow} Suppose $0<T \leq 1$ and $u:[0,T] \rightarrow X$ is a solution to the Benjamin-Ono equation with zero mean. Let $w$ and $A,B$ be defined as in \eqref{gaugesecondappearance} and \eqref{wequationsecondoccurence}, respectively. Let $j=0,1,2$ and $r>3/4$. For $T' \in (0,T]$ we have
\begin{equation}\label{estimateafterapplyingStrichartz}
    \|\langle \nabla \rangle^{\eta}\partial_x^jw\|_{S([0,T'])} \lesssim  \|\langle \nabla \rangle^{\eta+r}\partial_x^jw(0)\|_{L^2} + \|\langle \nabla \rangle^{\eta+r}\partial_x^jA\|_{L^1([0,T'],L^2)}+\|\langle \nabla \rangle^{\eta+r}\partial_x^jB\|_{L^1([0,T'],L^2)}.
\end{equation}
\end{lemma}
\hfill

We now verify that the local wellposedness theory in $X$ does provide enough regularity to justify above calculations. 
The next lemma asserts that $w$ satisfies the Duhamel formulation of equation (\ref{wequationsecondoccurence}) and derivatives thereof, equation (\ref{wxequation}).

\begin{lemma}\label{enoughregularity}
    Suppose $u \in C([0,T],X_{\omega}^{s_1,s_2})$ is a solution to the Benjamin--Ono equation with vanishing mean. Then the function $w$ defined by \eqref{definitionF} and \eqref{gaugesecondappearance} is in $C([0,T],H^{s_1,s_2}_{\omega})$ and satisfies
    \begin{equation}\label{wDuhamel}
        \langle \nabla \rangle^{\eta}\partial_x^jw(t)=\langle \nabla \rangle^{\eta}\partial_x^je^{it\partial_{xx}}w(0) + \int_0^t e^{i(t-s)\partial_{xx}}\langle \nabla \rangle^{\eta}\partial_x^j(A(s)+B(s))ds
    \end{equation}
pointwise in space and time for $j=0,1,2$. 
\end{lemma}

\begin{proof}
We begin by noting the following statement describing the regularity of the nonlinearity $A(s)+B(s)$: for $j=0,1,2$, we have that
\begin{equation}\label{regularityproperty}
    \langle \nabla \rangle^{\eta}\partial_x^j(A+B)(s) \in C_s([0,T], H^{s_1',s_2'}_{\omega}) \subseteq C([0,T], L^{\infty}(\mathbb{T}^2)), \end{equation}
where
\begin{equation}\label{sprimes} s_1' =(s_1-3) - \eta \frac{s_1-1}{s_2}, \quad s_2'=\frac{s_1-3}{s_1-1}s_2-\eta.
\end{equation}
Indeed, $u, F \in C([0,T],H^{s_1,s_2}_{\omega})$ by definition of $X$-norm, and the algebra property of $H_{\omega}^{s_1,s_2}$ enables exponentiation, so $w=P_{+hi}(e^{-iF})$ and $P_{lo}(e^{-iF})$ are $C([0,T],H^{s_1,s_2}_{\omega}) \subseteq C([0,T],H^{s_1-1,\frac{s_1-1}{s_1}s_2}_{\omega})$. 
Using the algebra property and boundedness property \eqref{boundedness1} in Lemma \ref{anisotropicproperties} gives $$A,B\in C([0,T],H^{s_1-1,\frac{s_1-1}{s_1}s_2}).$$
Now \eqref{regularityproperty} follows by invoking boundedness properties \eqref{boundedness1} and \eqref{boundedness2} in Lemma \ref{anisotropicproperties}. \\

We turn to verifying Duhamel's formula \eqref{wDuhamel}. By assumption $u$ is a continuous mapping $[0,T] \rightarrow X$ and satisfies 
\begin{equation}
    u(t) = u_0 + \int_0^t -H\partial_{xx}u(s) + \partial_x[u(s)^2] ds
\end{equation}
pointwise.
In particular, it holds in $L^2(\mathbb{T}^2)$.
Observe that the mean-zero antiderivative operator $\partial_x^{-1}$ is a closed linear operator $D(\partial_x^{-1}) \subseteq L^2 \rightarrow L^2$.
Applying Hille's Theorem (Proposition \ref{Hillestheorem}) yields that $F$ satisfies
\begin{multline*}F(t) - t \|u_0\|_{L^2}^2 =\partial_x^{-1}u(t)=\partial_x^{-1}u_0 + \int_0^t -H\partial_{xx}F(s) +u(s)^2 - \|u_0\|_{L^2}^2  ds \\ = F_0 + \int_0^t -H\partial_{xx}F(s) +[F_x(s)]^2 ds - t\|u_0\|_{L^2}^2\end{multline*}
with equality in $L^2$. Equality also holds pointwise in space and time as both sides of the equations are $C(\mathbb{T}^2)$-functions. 
The integrand $H\partial_{xx}F(s) + u(s)^2$ is $C([0,T],C(\mathbb{T}^2))$ by anisotropic Sobolev embedding so that $F$ is once continuously differentiable and satisfies $$F_t+HF_{xx} = F_x^2.$$
Hence,
\begin{align*}\label{Ftequation}
    \frac{d}{dt}e^{-iF(s)} -i\partial_{xx}e^{-iF(s)} = [-iF_t +i F_x^2 - F_{xx}]e^{-iF(s)}  = [-I+iH]F_{xx}e^{-iF} = -2P_-F_{xx}e^{-iF} 
\end{align*}
classically by chain rule. This also holds in the $L^2(\mathbb{T}^2)$-Frechet sense, as can be verified by integrating in time and using that the integrand $i\partial_{xx}e^{-iF(s)}-2P_-F_{xx}e^{-iF}$ is $C([0,T],C(\mathbb{T}^2))$.
We want to apply the backwards Schr\"odinger propagator $e^{-is\partial_{xx}}$ 
as `integrating factor' --- in order to do so, we note that 
\begin{equation}\label{integratingfactor}\frac{d}{dt}(e^{-it\partial_{xx}}e^{-iF}) = e^{-it\partial_{xx}}\big(\frac{d}{dt}e^{-iF} -i \partial_{xx}e^{-iF}\big)\end{equation}
in the $L^2(\mathbb{T}^2)$-Frechet sense. (One way to see this is by finding the $l_n^2(\mathbb{Z}^2)$-Frechet derivative of $e^{-it (\omega \cdot n)^2}\mathcal{F}({e^{iF(t)}})(n)$ using mean value theorem and dominated convergence theorem.) Thus,
 \begin{equation}\frac{d}{dt}(e^{-it\partial_{xx}}e^{-iF}) = e^{-is\partial_{xx}}[-2P_-F_{xx}e^{-iF}]\end{equation} in the $L^2(\mathbb{T}^2)$-Frechet sense. Integrating in time $s\in[0,t]$ and applying the forward propagator $e^{it\partial_{xx}}$ then yields
\begin{equation}\label{Duhameltypee-iF}
    e^{-iF} = e^{it\partial_{xx}}e^{-iF} + \int_0^t e^{(t-s)i\partial_{xx}}[-2P_{+hi}(P_-(F_{xx}(s))e^{-iF(s)})]ds
\end{equation}
with equality in $L^2(\mathbb{T}^2)$. We conclude \eqref{wDuhamel} with equality in $L^2(\mathbb{T}^2)$ by applying the operators $P_{+hi}\langle \nabla \rangle^{\eta} \partial_x^j$, $j=0,1,2$, which commute with the Bochner integral by Hille's Theorem (Proposition \ref{Hillestheorem}) applicable due to \eqref{regularityproperty}. 
Due to the frequency projection in the paraproduct $P_{+hi}(P_-(F_{xx}(s))e^{-iF(s)})$ the negative frequencies of $e^{-iF}$ can be discarded giving the expression $A(s)+B(s)$. Finally, both the left hand side and right hand side in equation \eqref{wDuhamel} are in fact continuous in space and time by \eqref{regularityproperty} and Lemma \ref{anisotropicproperties}, and thus equality holds pointwise.
\end{proof}

The Duhamel formulation \eqref{wDuhamel} of the equation for $w$ is exactly what we need to utilize Strichartz estimates. 

\begin{lemma}[retarded Strichartz estimates]\label{retardedStrichartz}
    For $j=0,1,2$, $\kappa>1/4$, $h>1/2$ and $0<T' \leq T$ we have
\begin{equation}\label{retardedStrichartzestimate}\Big\|\int_0^t e^{i(t-s)\partial_{xx}}\langle \nabla \rangle^{\eta} \partial_x^j(A(s)+B(s))ds \Big\|_{S_t([0,T'])} \lesssim (T')^{1/8} \|\langle \nabla \rangle^{\kappa+h+\eta}\partial_x^j(A+B)\|_{L^1_t([0,T'], L^2)}.\end{equation}
\end{lemma}

\begin{proof}
The arguments used in the proof of the Christ-Kiselev lemma \cite{MR1809116} give the following reduction:
    Provided \begin{equation}\label{inhomogeneousStrichartzestimate}\Big\|\int_0^{T'} e^{i(t-s)\partial_{xx}} \tilde{G}(s)ds\Big\|_{L^4_t([0,T'],L^{\infty})} \lesssim (T')^{1/8} \|\langle \nabla \rangle^{\kappa+h}\tilde{G}(t)\|_{L^1_t([0,T'],L^2)}\end{equation} holds for any $\tilde{G}(s)=\langle \nabla \rangle^{\eta}\partial_x^j(A+B)(s)1_{B}(s)$ with $B\subseteq [0,T']$ measurable, then \eqref{retardedStrichartzestimate} also holds. \\ 
    
While the statement of the Christ-Kiselev lemma in \cite{MR1809116} is for real-valued $L^p$ functions and discrete filtrations, it is mentioned that it also holds for Banach space valued functions and continuous filtration. 
We provide details as needed for the above reduction. 
Denoting
\begin{equation*}RG(t):= \int_0^{T'} e^{i(t-s)\partial_{xx}} \langle \nabla \rangle^{-\kappa-h}G(s)ds\end{equation*}
we have by assumption that $$\|RG\|_{L^4([0,T'],L^{\infty})} \lesssim (T')^{1/8}\|G\|_{L^1([0,T'],L^2)}$$
with implicit constant uniform for all $G(s)=\langle \nabla \rangle^{\kappa+h+\eta}\partial_x^j(A+B)(s)1_{B}(s)$ with $B\subseteq [0,T']$ measurable. 
For a dyadic number $2^q$ consider the finite filtration $\{[0,m2^{-q}T']:1 \leq m \leq 2^q\}$ of the interval $[0,T']$. Define the associated maximal function $$R_q^*G(t):= \sup_{0 \leq m \leq 2^q} \|R(G \cdot 1_{[0,m2^{-q}]})(t)\|_{L^{\infty}(\mathbb{T}^2)}.$$
The proof of the Christ-Kiselev lemma in \cite{MR1809116} adapted to Banach space valued functions and finite filtrations shows that $$\|R_q^*G\|_{L^4_t([0,T'])} \lesssim (T')^{1/8} \|G\|_{L^1_t([0,T'],L^2)}$$
with implicit constant independent of $q$. 
Denoting $T_q(t):=m2^{-q}T'$ for the unique $m$
such that $t \in [(m-1)2^{-q}T',m2^{-q}T')$, we therefore have that $$\|\int_0^{T_q(t)} e^{i(t-s)\partial_{xx}}\langle \nabla \rangle^{-\kappa -h}G(s)ds\|_{L^4_t([0,T'],L^{\infty})} \leq \|R_q^*G(t)\|_{L^4_t([0,T'])} \lesssim (T')^{1/8} \|G\|_{L^1_t([0,T'],L^2)}.$$
Now \begin{align*}
    \Big\| \int_0^{T_q(t)} e^{i(t-s)\partial_{xx}}\langle \nabla \rangle^{-\kappa-h}G(s)ds - \int_0^{t} e^{i(t-s)\partial_{xx}}\langle \nabla \rangle^{-\kappa-h}G(s)ds \Big\|_{L^4_t([0,T'],L^{\infty}(\mathbb{T}^2))}  \\\leq 
    \||T_q(t)-t| \cdot {\sup_{s} \|\langle \nabla \rangle^{-\kappa-h}G(s)\|_{H^{s'_1,s'_2}_{\omega}}}\|_{L^4_t([0,T'])} \lesssim_u 2^{-q} \rightarrow 0
\end{align*}
as $q \rightarrow \infty$, where $s'_1,s'_2$ are as in \eqref{sprimes}. \\

We turn to verifying the inhomogeneous Strichartz estimate \eqref{inhomogeneousStrichartzestimate}. Let $\kappa >1/4$ and $h>1/2$, $B \subseteq [0,T']$ measurable and denote $\tilde{G}(s)=\langle \nabla \rangle^{\eta}\partial_x^j(A+B)(s)1_{B}(s)$.
We calculate
\begin{align*}
     \Big\|\int_0^T e^{i(t-s)\partial_{xx}}\tilde{G}(s)ds\Big\|_{L^4_t([0,T'], L^{\infty}(\mathbb{T}^2))}  & = \Big\| e^{it\partial_{xx}}\int_0^{T'} e^{-is \partial_{xx}}\tilde{G}(s)ds\Big\|_{L_t^4([0,T'],L^{\infty}(\mathbb{T}^2))} \\ 
    & \lesssim (T')^{1/8} \Big\| \langle \nabla \rangle^{h+\kappa} \int_0^{T'} e^{-is \partial_{xx}}\tilde{G}(s)ds\Big\|_{L^2(\mathbb{T}^2)} 
     \\ & \lesssim (T')^{1/8} \|\langle \nabla \rangle^{\kappa+h}\tilde{G}\|_{L_t^1([0,T'],L^2(\mathbb{T}^2))}.
\end{align*}
Commuting the propagator with the integral in the first line is justified by Hille's Theorem (Proposition \ref{Hillestheorem}); the integrand $e^{-is\partial_{xx}}\tilde{G}(s)=e^{-is\partial_{xx}}\langle \nabla \rangle^{\eta}\partial_x^j(A+B)(s)1_B(s)$ is a $H^{s_1',s_2'}(\mathbb{T}^2)$-valued Bochner-integrable function by \eqref{regularityproperty}, and for each fixed time $t$ the propagator is a bounded linear operator on $H^{s_1',s_2'}_{\omega}(\mathbb{T}^2)$ which embeds into $L^{\infty}(\mathbb{T}^2)$. The second line uses Sobolev embedding $H^{h,4}(\mathbb{T}^2) \hookrightarrow L^{\infty}(\mathbb{T}^2)$ 
and the Strichartz estimate in Lemma \ref{Strichartzlemma}. 
The third line can be obtained by commuting the operator $\langle \nabla \rangle^{\kappa +h}$ inside the integral, which is again justified by Hille's Theorem, and then using the dual inhomogeneous Strichartz estimate \eqref{dualStrichartz} which in this case reduces to the triangle inequality for the Bochner integral.
\end{proof}

Combining the Duhamel formulation (\ref{wDuhamel}) of equation (\ref{wxequation}), the linear Strichartz estimate (\ref{linearStrichartzestimate}) together with Sobolev embedding $H^{h,4}(\mathbb{T}^2) \hookrightarrow L^{\infty}(\mathbb{T}^2), \ h>1/2$, and the retarded Strichartz estimate (\ref{retardedStrichartzestimate}) proves estimate \eqref{estimateafterapplyingStrichartz} for the $L^4([0,T'],L^{\infty})$-part of the Strichartz norm. Estimate \eqref{estimateafterapplyingStrichartz} for the $L^{\infty}([0,T'],L^2)$-part of the Strichartz norm is clear from \eqref{wDuhamel} and \eqref{regularityproperty}. This completes the proof of  Lemma \ref{Strichartzappliedtow}.

\subsection{Part I of the bootstrap: estimates on the gauge transformed function}\label{partIbootstrap3}

We prove part $\RomanI$ of the bootstrap conclusion \eqref{C} by controlling the three terms on the right hand side of estimate \eqref{estimateafterapplyingStrichartz} in Lemma \ref{Strichartzappliedtow}. 

Most terms appearing when expanding the right-hand side of \eqref{estimateafterapplyingStrichartz} can be controlled using the bootstrap hypothesis, except for one term, which will be controlled in the next section by a Gronwall argument, see Lemma \ref{wenergyestimate}. We use two paraproduct estimates \eqref{paraproductone} and \eqref{paraproducttwo}, the proof of which is the content of Section \ref{paraproductestimatessection}.
Throughout this section, we use the notation $\|\cdot\|_{L^pL^q}$ as an abbreviation for mixed Lebesgue spacetime norms $\|\cdot\|_{L^p([0,T'],L^q(\mathbb{T}^2))}$, always with respect to the time interval $[0,T']$. \\

We begin by bounding the terms arising from the error term $B$.

\begin{lemma}\label{BestimateforI} Assume the bootstrap hypothesis holds. For $j=0,1,2$ and $7/8<r<\sigma-\eta$, we have
    $$\|\langle \nabla \rangle^{r+\eta} \partial_x^j B\|_{L^1L^2} \lesssim \varepsilon^2.$$
\end{lemma}

\begin{proof}
Observe that the term $B$ is low frequency with respect to tangential derivatives, so that it suffices to prove the lemma for $j=0$. Due to the frequency projections we may further replace $P_-(F_{xx})$ by $P_-P_{lo}(F_{xx})$, that is $B=P_{+hi}(P_-P_{lo}(F_{xx})P_{lo}(e^{iF}))$. Using $L^2$ boundedness of $P_{+hi}$ and fractional chain rule, we estimate that 
    \begin{align*}
        \|\langle \nabla \rangle^{r+\eta}B\|_{L^1L^2} & \lesssim  \|\langle\nabla\rangle^{r+\eta}P_-P_{lo}(F_{xx})\|_{L^{\infty}L^2} \|P_{lo}(e^{-iF})\|_{L^{\infty}L^{\infty}} \\ &+ \|P_-P_{lo}(F_{xx})\|_{L^{\infty}L^2} \|P_{lo}(e^{-iF})\|_{L^{1}L^{\infty}} \\ & + \|P_-(F_{xx})\|_{L^4L^{\infty}}\| |\nabla|^{r+\eta}P_{lo}(e^{-iF})\|_{L^{4/3}L^2}  \lesssim \varepsilon^2.
    \end{align*}
The factors $\|P_{lo}(e^{-iF})\|_{L^{\infty}L^{\infty}}$ and $\|P_{lo}(e^{-iF})\|_{L^{\infty}L^{\infty}}$ are just $1$. The terms $\|P_-P_{lo}(F_{xx})\|_{L^{\infty}L^2}$ and $\|\langle \nabla \rangle^{r+\eta}P_-P_{lo}(F_{xx})\|_{L^{\infty}L^2}$ are $O(\varepsilon^2)$ by estimate \eqref{eq5} and $L^2$-boundedness of $\partial_x$ on low frequency functions together, thus bounding the first and second summand by $O(\varepsilon^2)$. To estimate the third summand, use fractional chain rule (Lemma \ref{chainrule}) and estimate \eqref{eq16} giving $$\||\nabla|^{r+\eta}P_{lo}(e^{-iF})\|_{L^{4/3}L^2} \lesssim \||\nabla|^{r+\eta}F\|_{L^{\infty}L^2}\|e^{-iF}\|_{L^{4/3}L^{\infty}} = \||\nabla|^{r+\eta}F\|_{L^{\infty}L^2} \lesssim \varepsilon^2,$$ while by bootstrap hypothesis $\|P_-(F_{xx})\|_{L^4L^{\infty}} = O(\varepsilon)$. Estimates \eqref{eq16} and \eqref{eq5} are obtained in Section \ref{partIIbootstrap} only assuming the bootstrap hypothesis, by norm growth control (Lemmas \ref{Hshnormgrowth} and \ref{Fyqnormgrowth}) and smallness of initial data $\|u_0\|_Y \lesssim \varepsilon^2$.
\end{proof}

Next, we estimate the terms arising from $A$.

\begin{lemma}\label{estimatesforA} Assume the bootstrap hypothesis holds. For $j=0,1,2$ and $7/8<r<\sigma-\eta$, we have
    $$\|\langle \nabla \rangle^{r+\eta} \partial_x^jA\|_{L^1L^2} \lesssim \varepsilon^2,$$
provided $\|\langle \nabla \rangle^{r+\eta}\partial_x^jw\|_{L^{\infty}L^2} \lesssim \varepsilon^2$ holds.
\end{lemma}

\begin{proof}
By Plancherel, using $\langle \xi_1 + \xi_2 \rangle^{r+\eta} \lesssim \langle \xi_1 \rangle^{r+\eta} + \langle \xi_2 \rangle^{r+\eta}$ we have that
    \begin{equation}\label{eq19}
        \|\langle \nabla \rangle^{r+\eta}\partial_{x}^jA\|_{L^1L^2} \lesssim \|\partial_{x}^jP_{+hi}([\langle \nabla \rangle^{r+\eta}P_-(F_{xx})]w)\|_{L^1L^2} + \|\partial_{x}^jP_{+hi}(P_-(F_{xx})\langle \nabla \rangle^{r+\eta}w)\|_{L^1L^2}.
    \end{equation}
We further estimate both terms on the right hand side by using the paraproduct estimates in Lemmas \ref{paraproductone} and \ref{paraproducttwo}, allowing us to move the derivatives $\partial_x^j$ on the second factor $w$, and using H\"older in time:
\begin{equation}\label{eq20}
\|\partial_{x}^j P_{+hi}([\langle \nabla \rangle^{r+\eta}P_-(F_{xx})]w)\|_{L^1L^2} \lesssim \|\langle \nabla \rangle^{r+\eta} P_-(F_{xx})\|_{L^{4/3}L^2} \|\partial_x^jw\|_{L^4L^{\infty}} \lesssim \varepsilon^2
\end{equation}
and
\begin{equation}\label{eq21}
\|\partial_{x}^jP_{+hi}(P_-(F_{xx})\langle \nabla \rangle^{r+\eta}w)\|_{L^1L^2} \lesssim \| P_-(F_{xx})\|_{L^{4}L^{\infty}} \|\langle \nabla \rangle^{r+\eta}\partial_x^jw\|_{L^{\infty}L^{2}} \lesssim \varepsilon^2.
\end{equation}
The terms $\|P_-(F_{xx})\|_{L^4L^{\infty}}$ and $\|\langle \nabla \rangle^{r+\eta}P_-(F_{xx})\|_{L^{4/3}L^2}$ and $\|\partial_x^jw\|_{L^4L^{\infty}}$ are all $O(\varepsilon)$ by the bootstrap hypothesis, while the term $\|\langle \nabla \rangle^{r+\eta}\partial_x^jw\|_{L^{\infty}L^2}$ is  $O(\varepsilon^2)$ by assumption. 
\end{proof}

That the assumption $\|\langle \nabla \rangle^{r+\eta}\partial_x^jw\|_{L^{\infty}L^2} \lesssim \varepsilon^2$ is indeed satisfied will be shown in Lemma \ref{wenergyestimate} in the next section via energy arguments using equation \eqref{wxequation} for $\langle \nabla \rangle^{\eta}\partial_x^jw$. 

\begin{remark}\label{whyparaproducts}
    The paraproduct estimate in Lemma  \ref{paraproducttwo} is of crucial importance in estimate \eqref{eq20}. We want to put the factor $\langle \nabla \rangle^{r+\eta}P_-(F_{xx})$ which features the highest order derivatives into $L^2$, while the other factor goes into $L^{\infty}$. For terms measured in $L^2$ we can hope to gain control via Gronwall arguments, as we indeed do in the next section for $\|\langle \nabla \rangle^{r+\eta}w_{xx}\|_{L^{\infty}L^2}$, which then in turn is related to $\|\langle \nabla \rangle^{r+\eta}P_-(F_{xx})\|_{L^{4/3}L^2}$ in equation (\ref{PHIFx432}) in Section \ref{partIIbootstrap}.
    In the setting on the real line in \cite{MR2052470}, the Strichartz estimate does not incur a derivative loss and it suffices to have the paraproduct \cite[Lemma 3.2]{MR2052470} corresponding in this paper to Lemma \ref{paraproductone}, which is utilized in estimate \eqref{eq21}.
\end{remark}

Finally, we bound the term arising from the linear evolution of the initial data.

\begin{lemma}\label{initialdatawbound} For $j=0,1,2$ and $7/8<r<\sigma -\eta$, we have
    $$\|\langle \nabla \rangle^{r+\eta} \partial_x^jw(0)\|_{L^2} \lesssim \varepsilon^2.$$
\end{lemma}

\begin{proof} 
The statement for $j=0,1$ follows from the statement for $j=2$ due to the frequency projection $P_{+hi}$ involved in the definition of the function $w$ (equation (\ref{gaugesecondappearance})) and $L^2$-boundedness of $\partial_x^{-1}$ on functions with frequency support bounded away from zero. \\

We turn to the case $j=2$: By fractional chain rule
   \begin{multline*} 
       \|\langle \nabla \rangle^{r+\eta} w_{xx}(0)\|_{L^2} = \|\langle \nabla \rangle^{r+\eta}P_{+hi}((-iF_{xx} -F_x^2)e^{-iF(0)})\|_{L^2} \\ \lesssim \|\langle \nabla \rangle^{r+\eta}(-iF_{xx}-F_x^2)\|_{L^2} \|e^{-iF}\|_{L^{\infty}} + \|-iF_{xx} - F_x^2\|_{L^q} \|\langle \nabla \rangle^{r+\eta} e^{-iF}\|_{L^p}.
   \end{multline*}
   Here $\frac{1}{p} + \frac{1}{q} = \frac{1}{2}$. The first term can be bounded
   \begin{align*} \|\langle \nabla \rangle^{r+\eta}(-iF_{xx}-F_x^2)\|_{L^2} & \lesssim \|\langle \nabla \rangle^{r+\eta} u_x(0)\|_{L^2} + \|u^2(0)\|_{H^{1+\sigma, \sigma}} \\ & \lesssim \|\langle \nabla \rangle^{\sigma} \partial_x u(0)\|_{L^2} + \|u(0)\|_{H^{1+\sigma,\sigma}}^2  \lesssim \|u(0)\|_Y + \|u(0)\|_Y^2 \lesssim \varepsilon^2.\end{align*}
As $\sigma-\eta/2>\sigma - \eta>7/8$ we have $\frac{\sigma}{2}+\frac{\sigma-\eta-r}{2}>\frac{1}{2}$ for small enough $r>3/4$ and can find $p,q$ with $\frac{1}{p}+\frac{1}{q}=\frac{1}{2}$ such that $$\frac{1}{q}=\frac{1}{2}-\theta \frac{\sigma}{2} \quad \textit{and} \quad \frac{1}{p}=\frac{1}{2} - \theta \frac{\sigma-\eta-r}{2} \quad \textit{for some} \ 0\leq \theta\leq 1.$$
    Thus, we can estimate by Gagliardo-Nirenberg
     $$\|F_{xx}(0)\|_{L^q} \lesssim \|\langle \nabla \rangle^{\sigma}u_x(0)\|_{L^2} \lesssim \|u(0)\|_Y \lesssim \varepsilon^2.$$
     Further, we estimate by fractional chain rule and Gagliardo-Nirenberg that \begin{align*}\|\langle\nabla \rangle^{r+\eta}e^{-iF(0)}\|_{L^p} & \lesssim \|e^{-iF(0)}\|_{L^p} + \||\nabla|^{r+\eta}e^{-iF(0)}\|_{L^p} \\ & \lesssim \|e^{-iF(0)}\|_{L^{\infty}}[1+\||\nabla|^{r+\eta}F(0)\|_{L^p}] \lesssim  1+\|\langle \nabla \rangle^{\sigma}F(0)\|_{L^2} \lesssim 1+\|u(0)\|_Y \lesssim 1.\end{align*}
     From anisotropic Sobolev embedding \eqref{anisotropicembedding},
     $$\|F_x^2(0)\|_{L^q} \lesssim \|F_x^2(0)\|_{L^{\infty}} \lesssim \|u(0)\|_{L^{\infty}}^2 \lesssim \|u(0)\|_{H^{1+\sigma,\sigma}}^2 \lesssim \|u(0)\|_Y^2 \lesssim  \varepsilon^2.$$
    Putting the previous three estimates together we obtain
    $$\|-iF_{xx}-F_x^2\|_{L^q} \|\langle \nabla \rangle^{r+\eta}e^{-iF}\|_{L^p} \lesssim \varepsilon^2,$$
    thus completing the proof of the lemma for $j=2$.
    \end{proof}

\begin{remark} This lemma dictates the threshold $\sigma>7/8$ in the Definition \ref{Yspace} of the $Y$-norm. \end{remark} 

 Putting Lemmas \ref{Strichartzappliedtow}, \ref{BestimateforI}, \ref{estimatesforA}, \ref{initialdatawbound} together, and invoking Lemma \ref{wenergyestimate} below yields part I of the bootstrap conclusion:

 \begin{lemma}
     Suppose the bootstrap hypothesis holds. Then $$\RomanI=\sum_{j=0,1,2}\|\langle \nabla \rangle^{\eta}\partial_x^jw\|_{S([0,T'])} \lesssim \varepsilon^2.$$
 \end{lemma}

\subsection{Energy estimate on derivatives of the gauge transformed function}\label{sectionenergyestimatepartI}

In this section, we prove the estimate $\|\langle \nabla \rangle^{r+\eta} \partial_x^jw\|_{L^2} \lesssim \varepsilon^2$ that was used in Lemma \ref{estimatesforA}. The proof is based on a Gronwall argument for the quantity $\|\langle \nabla \rangle^{r+\eta} \partial_x^jw\|_{L^2}$.  We demonstrate that the regularity provided by the wellposedness theory in $X$ is enough to justify the calculations.

\begin{lemma}\label{wenergyestimate}
Assume the bootstrap hypothesis holds. For $j=0,1,2$ and $7/8<r<\sigma-\eta$, we have $$\|\langle \nabla \rangle^{r+\eta}\partial_x^jw\|_{L^{\infty}L^2} \lesssim \varepsilon^2.$$
\end{lemma}
\noindent The following lemma asserts that the function $t \mapsto \langle \nabla \rangle^{r+\eta} \partial_x^j w(t)$ satisfies a Schr\"odinger-type equation in the sense of $L^2(\mathbb{T}^2)$-Frechet derivatives. This will suffice to carry out the Gronwall argument.

\begin{lemma}\label{regularityforgronwall}
Denote 
$$t \mapsto \phi(t):=\langle \nabla \rangle^{r+\eta} \partial_x^j w(t)$$
and $$t \mapsto \psi(t):= \langle \nabla \rangle^{r+\eta} \partial_x^{j}(-iw_{xx})(t) + \langle \nabla \rangle^{r+\eta} \partial_x^{j}A(t) + \langle \nabla \rangle^{r+\eta} \partial_x^{j}B(t).$$
These are both continuous mappings $[0,T] \rightarrow L^2(\mathbb{T}^2)$ and we have $$\partial_t\phi = \psi$$ in the sense of Frechet derivatives in $L^2(\mathbb{T}^2)$.
\end{lemma}

\begin{proof}
    The assertion that $\phi, \psi \in C([0,T],L^2(\mathbb{T}^2))$ follows from $w,F$ being in the algebra $H^{s_1+1,s_2}$ and Young's inequality. This also gives continuity into $L^2(\mathbb{T}^2)$ of $$t \mapsto -H\partial_{xx}w(s) + A(s) + B(s).$$ To prove the statement about Frechet derivatives, start by noting that the integral equation associated with \eqref{wequationsecondoccurence}, $$w(t) = w(0) + \int_0^t -i\partial_{xx}w(s) +A(s) + B(s) ds,$$ holds with equality in $L^2$, as a consequence of the proof of Lemma \ref{enoughregularity}. Continuity of $\psi$ in particular gives Bochner integrability so that Hille's theorem is applicable with the closed linear operator $\langle \nabla \rangle^{r+\eta} \partial_x^j$, yielding $$\phi(t)=\phi(0)+\int_0^t \psi(s) ds.$$
    From this, the statement about the Frechet derivative follows by the fundamental theorem of calculus for Bochner integrals.
\end{proof}
We turn to proving Lemma \ref{wenergyestimate}. By Lemma \ref{regularityforgronwall}, we can calculate \begin{align*}\frac{d}{dt} \|\langle \nabla \rangle^{r+\eta} \partial_x^j w(t)\|_{L^2}^2 = 2 \langle \phi(t), \partial_t \phi(t) \rangle  = 2\int_{\mathbb{T}^2} \phi(t) \psi(t). \end{align*}
Using anti-selfadjointness of $-i\partial_{xx}$ we bound this up to a fixed constant factor by
\begin{equation*}
\|\langle \nabla \rangle^{r+\eta} \partial_x^jw\|_{L^2}[\|\langle \nabla \rangle^{r+\eta} \partial_x^jA\|_{L^2} + \|\langle \nabla \rangle^{r+\eta} \partial_x^jB\|_{L^2}]. 
\end{equation*}
Dividing by $\|\langle \nabla \rangle^{r+\eta} \partial_x^jw\|_{L^2}$, using estimates \eqref{eq19}, \eqref{eq20}, \eqref{eq21} in the proof of Lemma \ref{estimatesforA}, and integrating in time gives
\begin{equation*}
 \|\langle \nabla \rangle^{r+\eta} \partial_x^jw(t)\|_{L^2} \leq \alpha(t) + \int_0^t \beta(s) \|\langle \nabla \rangle^{r+\eta} \partial_x^jw(s)\|_{L^2} ds 
\end{equation*}
where 
$$\beta(s):=\|P_-F_{xx}(s)\|_{L^{\infty}}$$
and
$$\alpha(t):=\|\langle \nabla \rangle^{r+\eta} \partial_x^jw(0)\|_{L^2} + \|\langle \nabla \rangle^{r+\eta}P_-(F_{xx})\|_{L^{4/3}([0,t],L^2)}\|w_{xx}\|_{L^4([0,t],L^{\infty})} + \|\langle \nabla \rangle^{r+\eta} \partial_x^jB\|_{L^1([0,t],L^2)}.$$
By Gronwall's inequality, we get
$$\|\langle \nabla \rangle^{r+\eta} \partial_x^jw(t)\|_{L^2} \leq \alpha(t) \exp \big( \int_0^t \beta(s) ds \big) \leq \varepsilon^2 \exp (\|P_-F_{xx}\|_{L^1L^{\infty}}) \lesssim \varepsilon^2,$$
where we used $\alpha(t) \lesssim \varepsilon^2$ and $ \exp (\|P_-F_{xx}\|_{L^1L^{\infty}})\lesssim 1$ due to the bootstrap hypothesis and Lemmas  \ref{BestimateforI} and \ref{initialdatawbound}.

\subsection{Paraproduct estimates}\label{paraproductestimatessection}
The estimates on the nonlinearity in the gauge transformed equation that we obtained in Section \ref{partIbootstrap3} depended crucially on two bounds on paraproduct terms. In this section we prove these estimates. While the first paraproduct estimate follows from 
Littlewood-Paley techniques, the second paraproduct estimate is technically more challenging, requiring us to employ the more robust machinery of paradifferential calculus.

\begin{lemma}[First paraproduct estimate]\label{paraproductone}
Suppose $s \geq 0$ and let $f, g$ be functions on the torus $\mathbb{T}^2$. We have that $$\|\partial_x^s P_{+hi}(P_-(f)g))\|_{L^2(\mathbb{T}^2)} \leq \|P_-(f)\|_{L^{\infty}(\mathbb{T}^2)}\|\partial_x^s g\|_{L^2(\mathbb{T}^2)}.$$
\end{lemma}

\begin{proof}
    This paraproduct estimate follows by adapting the proof of Lemma 3.2 in \cite{MR2052470} with the Littlewood-Paley operators replaced by Littlewood-Paley operators in the $\omega$-direction. 
    In the following, $P_n:=P_{2^n,\infty}-P_{2^{n-1},\infty}$ denotes 
    projection to Fourier modes $k$ with $2^{n-1}< |\omega \cdot k|\leq 2^n$, see \eqref{generalprojection}. For $s=0$ the statement is immediate by $L^2$-boundedness of $P_{+hi}$, so we now assume $s>0$.
Due to the projection to positive frequencies we have that 
$$P_nP_{+hi}(P_-(f)g) = P_nP_{+hi}(P_-(f)P_{\geq n-1}(g))$$ so that using $L^2$-boundedness of $P_nP_{+hi}$
$$\|P_nP_{+hi}(P_-(f)g)\|_{L_2} \leq \|P_-(f)\|_{L^{\infty}} \|P_{\geq n-1}g\|_{L^2}.$$ Multiplying by $2^{sn}$ and square summing in $n$ we further bound by Plancherel  
    $$\|\partial_x^s P_{+hi}(P_-(f)g))\|_{L_{x,y}^2} \leq \|P_-(f)\|_{L^{\infty}} \Big[ \sum_n 2^{2ns} \|P_{\geq n-1} g\|_{L^2}^2 \Big]^{\frac{1}{2}}.$$
Breaking up $\|P_{\geq n-1}g\|_{L^2_{x,y}}$ into Littlewood-Paley pieces, switching summation order, bounding geometric series and again using Plancherel, the last factor is bounded by 
\begin{equation*}\Bigg[ \sum_{j} \Big[ \sum_{n=1}^{j+1} 2^{2s(n-j)} \Big] 2^{2sj} \|P_{j}g\|_{L^2_{x,y}}^2 \Bigg]^{\frac{1}{2}}  \lesssim \|\partial_x^sg\|_{L^2_{x,y}}. \qedhere \end{equation*}
\end{proof}

\begin{lemma}[Second paraproduct estimate]\label{paraproducttwo}
    Suppose $s\geq 0$ and let $f, g$ be functions on the torus $\mathbb{T}^2$. We have that $$\||\partial_x|^s P_{+hi}(P_-(f)g))\|_{L^2(\mathbb{T}^2)} \leq \|P_-(f)\|_{L^2(\mathbb{T}^2)}\||\partial_x|^s g\|_{L^{\infty}(\mathbb{T}^2)}.$$
    In particular, if $s$ is an integer and $g$ has only positive frequency support with respect to the tangential direction, then this estimate holds with $|\partial_x|^s$ replaced by ordinary derivatives $\partial_x^s$.
\end{lemma}

The argument for Lemma \ref{paraproductone} cannot be employed here. 
When trying to place $\partial_x^sg$ in $L^{\infty}(\mathbb{T}^2)$, the reasoning breaks down due to limitations of Littlewood-Paley techniques in $L^{\infty}$. To overcome these limitations, we abandon the direct 
Littlewood-Paley approach and instead rely on paradifferential calculus. We establish the estimate in the Euclidean setting and then conclude by invoking the transference principle for bilinear multiplier operators (Proposition \ref{transferenceparaproduct}).

\begin{proof}[Proof of second paraproduct estimate]

First, observe that it suffices to prove the one-dimensional paraproduct estimate \begin{equation}\label{paraproductonline}\|\partial_x^sP_{+hi}(P_-(f)g)\|_{L^2(\mathbb{R})} \lesssim \|P_-(f)\|_{L^2(\mathbb{R})} \|\partial_x^sg\|_{L^{\infty}(\mathbb{R})}\end{equation} for functions $f$ and $g$ on the real line. Here, and throughout the remainder of the proof, $P_{+hi}$ is the Fourier multiplier operator on the real line with symbol $\psi_{+hi}(\xi)$, where $\xi \in \mathbb{R}$ is the frequency variable. (We emphasize that this operator is a one-dimensional version, distinct from the $\mathbb{T}^2$ multiplier operator defined in Section \ref{sectionnotation}.)
Indeed, the transference principle in Proposition \ref{transferenceparaproduct} is applicable to the related estimate 
\begin{equation}\label{paraproductonplane}\|\partial_{\omega}^sP_{+hi,\omega}(P_{-,\omega}(\tilde{f})\partial_{\omega}^{-s}\tilde{g})\|_{L^2(\mathbb{R}^2)} \lesssim \|\tilde{f}\|_{L^2(\mathbb{R}^2)} \|\tilde{g}\|_{L^{\infty}(\mathbb{R}^2)},\end{equation} which can be obtained from \eqref{paraproductonline} by a rotational dilation of the plane $\mathbb{R}^2$. In \eqref{paraproductonplane}, $P_{+hi,\omega}$ denotes the multiplier with symbol $\psi_{+}(\omega \cdot \xi)$, $\xi \in \mathbb{R}^2$ and $\partial_{\omega}^{-s}$ is the multiplier with symbol $\frac{1}{(\omega \cdot \xi)^s}1_{\xi \neq 0}(\xi)$. 
Note that $\partial_{\omega}^{-s}$ does not cause any issues on the left hand side as the paraproduct has frequency support bounded away from zero in the $\omega$-direction ($|\xi_1\cdot\omega|, |\xi_2 \cdot \omega| \gtrsim 1$ on the support of its symbol). 
\\

Now $P_{+hi}(P_-(f)g)$ is a paraproduct on $\mathbb{R}$ with symbol $m(\xi_1,\xi_2):=\psi_{+hi}(\xi_1+\xi_2)1_{(-\infty,0)}(\xi_1)$.
The support of the symbol is contained in the angular sector between the positive $\xi_2$-axis ($\xi_1=0$) and the portion of the line $\xi_1+\xi_2=0$ contained in the second quadrant of the $(\xi_1,\xi_2)$-plane. The symbol $m(\xi_1,\xi_2)$ is not smooth (it has a jump on the positive $\xi_2$-axis (at $\xi_1=0$)) and its derivatives do not have enough decay in the strip $1/2\leq\xi_1+\xi_2\leq 2$ as $\xi_2 \rightarrow \infty$ to be a Coifman-Meyer multiplier symbol. However, we can work around this in order to gain access to propositions about Coifman-Meyer paraproducts, which will allow us to move derivatives on the high-frequency factor in a low-high or high-high frequency interaction. \\

Denote the paraproducts $\pi_1(f,g):=P_{+hi}(f,g)$ and $\pi_2(f,g)=P_-(f)g$, and denote by $\#$ the binary operation on bilinear multiplier operators given by multiplying their symbols, which is commutative. We have that $P_{+hi}(P_-(f)g) = (\pi_1 \# \pi_2)(f,g).$
Define $\pi_3$ to be the paraproduct with symbol $$m_3(\xi_1,\xi_2):=\phi(\theta(\xi_1,\xi_2))\psi(r(\xi_1,\xi_2)) \quad \text{(with } \theta \text{ and } r \text{ standard polar coordinates)}$$ for functions $$\phi \in C^{\infty}(\mathbb{R}/2\pi\mathbb{Z},[0,1]) \ \text{with} \ \phi(\theta)=1 \ \text{for} \ \pi/2 \leq \theta \leq 3\pi/4 \ \text{and} \ \phi(\theta)=0 \ \text{for} \ -9\pi/8 \leq \theta \leq 3\pi/8,$$  and $$\psi \in C^{\infty}((0,\infty),[0,1]) \ \text{with} \ \psi(r)=1 \ \text{for} \ r \geq 1/4 \ \text{and} \ \psi(r)=0 \ \text{for} \ r \leq 1/8.$$ 
As the symbol of $\pi_3$ is constant $1$ on the support of the original paraproduct $\pi_1\#\pi_2$ we have that 
    $$\pi_1 \# \pi_2(f,g) = \pi_1 \# \pi_2 \# \pi_3 (f,g),$$
    and we may commute 
    $$P_{+hi}(P_-(f)g) = \pi_1 \# \pi_3 \# \pi_2 (f,g) = P_{+hi}(\pi_3(P_-(f),g)).$$
    By $L^2$-boundedness of the Littlewood-Paley projections $P_{+hi}$ and $P_-$, and as $P_{+hi}$ commutes with $\partial_x^s$ the problem reduces to the following estimate for the paraproduct $\pi_3$: $$\|\partial_x^s\pi_3(h, g)\|_{L^2_x} \lesssim \|h\|_{L_x^2}\|\partial_x^s g\|_{L_x^{\infty}}.$$
    \\
    
    But $\pi_3$ can be decomposed as a sum of a high-high and a low-high Coifman-Meyer paraproduct. We split $$\phi=\phi_1+\phi_2 \ \text{where} \ \phi_1(\theta)=\phi(\theta) \cdot \eta(\theta) \ \text{and} \ \phi_2(\theta)=\phi(\theta) \cdot (1-\eta(\theta))$$ for $$\eta \in C^{\infty}(\mathbb{R}/2\pi \mathbb{Z},[0,1]) \ \text{with} \ \eta(\theta)=1 \ \text{for} \ 3\pi/4\leq \theta \leq 11\pi/8 \ \text{and} \ \eta(\theta)=0 \ \text{and} \ -3\pi/8 \leq \theta \leq 5\pi/8.$$
    Then $$\pi_3 = \pi_3' + \pi_3''$$ where $\pi_3'$ and $\pi_3''$ are the paraproducts with symbol $m_3'(\xi_1,\xi_2)=\phi_1(\theta)\psi(r)$ and $m_3''(\xi_1,\xi_2)=\phi_2(\theta)\psi(r)$, respectively. Lemma \ref{claimiscoifmanmeyer} below shows that these both define Coifman-Meyer multiplier operators. Further, $\pi_3'$ is of type \textit{high-high} in that on the support of its symbol $|\xi_1|\sim|\xi_2|$. And $\pi_3''$ is of type \textit{low-high} in the sense that on the support of its symbol $|\xi_1|+|\xi_2|\sim|\xi_2|$. \\

     Now homogeneous Fourier multipliers commute with paraproducts in the following sense: \\
     
     $\bullet$ Given a homogeneous Fourier multiplier $D^s$ of order $s \geq 0$ and 
     a Coifman-Meyer paraproduct $\pi$ of type \textit{low-high}, we can write $$D^s \pi(f,g)=\pi'(f,|\partial_x|^sg)$$ for some other Coifman-Meyer paraproduct $\pi'$ of the same type, see \cite[Lemma 6.2]{Taoparaproducts}.
    In our case, we can easily verify directly that $$|\xi_1+\xi_2|^sm_3''(\xi_1,\xi_2)|\xi_2|^{-s}$$ is a Coifman-Meyer symbol. It is smooth as $m_3''$ is supported away from $\xi_2=0$. 
    By product rule,
\begin{equation*}
      \Big| \partial_{\xi_1}^{\alpha} \partial_{\xi_2}^{\beta} (|\xi_1+\xi_2|^s m_3''(\xi_1,\xi_2)|\xi_2|^{-s}) \Big|  \lesssim  \kern-1em\sum_{\substack{\alpha_1+\alpha_2+\alpha_3=\alpha  \\ \beta_1+\beta_2+\beta_3=\beta}}  \kern-0.2em\Big| \partial_{\xi_1}^{\alpha_1} \partial_{\xi_2}^{\beta_1}|\xi_1+\xi_2|^s \Big| \cdot \Big| \partial_{\xi_1}^{\alpha_2} \partial_{\xi_2}^{\beta_2}m_3''(\xi_1,\xi_2) \Big| \cdot \Big| \partial_{\xi_1}^{\alpha_3} \partial_{\xi_2}^{\beta_3}|\xi_2|^{-s} \Big|.
\end{equation*}
As $m_3''$ is of type \textit{low-high}, 
we have that $|\xi_1+\xi_2|\sim|\xi_2| \sim |\xi_1|+|\xi_2|$ on its support. We calculate that 
\begin{equation*}
    \Big| \partial_{\xi_1}^{\alpha_1} \partial_{\xi_2}^{\beta_1}|\xi_1+\xi_2|^s \Big| =|\xi_1+\xi_2|^{s-\alpha_1-\beta_1} \sim (|\xi_1|+|\xi_2|)^{s-\alpha_1-\beta_1},
\end{equation*}
and
\begin{equation*}|\partial_{\xi_1}^{\alpha_2} \partial_{\xi_2}^{\beta_2}m_3''(\xi_1,\xi_2)| \lesssim_{\alpha_2,\beta_2} (|\xi_1|+|\xi_2|)^{-\alpha_2-\beta_2} 
\end{equation*}
due to $m_3$ being Coifman-Meyer, and finally
\begin{equation}
    \Big| \partial_{\xi_1}^{\alpha_3} \partial_{\xi_2}^{\beta_3}|\xi_2|^{-s} \Big| = 
    \begin{Bmatrix}
    0 & \text{if} \  \alpha_3 > 0 \\
      |\xi_2|^{-s-\beta_3} & \text{if} \ \alpha_3=0 
    \end{Bmatrix}  \leq |\xi_2|^{-s-\beta_3-\alpha_3} \sim (|\xi_1|+|\xi_2|)^{-s-\alpha_3 - \beta_3}.
\end{equation}
Combining these three estimates yields
$\Big| \partial_{\xi_1}^{\alpha} \partial_{\xi_2}^{\beta} (|\xi_1+\xi_2|^s m_3(\xi_1,\xi_2)|\xi_2|^{-s}) \Big|  \lesssim (|\xi_1|+|\xi_2|)^{-\alpha-\beta}$ as desired.
Invoking the Coifman-Meyer theorem \cite{MR511821}\cite{CoifmanMeyerAlternative} we conclude
    \begin{equation}\label{lowhigh}
        \||\partial_x|^s \pi_3''(f,g)\|_{L^2} = \|\pi'(f,|\partial_x|^sg)\|_{L^2} \lesssim \|f\|_{L^2} \||\partial_x|^sg\|_{L^{\infty}}.
    \end{equation}

 $\bullet$ Given a homogeneous Fourier multiplier $D^s$ of order $s\geq 0$ and a Coifman-Meyer paraproduct $\pi$ of type \textit{high-high}, we can write
 $$D^s\pi(f,g)=\pi^r(f,|\partial_x|^sg)$$
for some \textit{high-high} residual paraproduct $\pi^r$, see \cite[Definition 5.1, Lemma 6.3]{Taoparaproducts}. 
Residual paraproducts define bounded bilinear maps $L^p(\mathbb{R}) \times L^q(\mathbb{R}) \rightarrow L^r(\mathbb{R})$ for $1\leq p,q,r \leq \infty$ with $\frac{1}{r}=\frac{1}{p}+\frac{1}{q}$ excluding the cases $(p,q)=(\infty,\infty),(\infty,1),(1,\infty)$, see \cite[Theorem 5.4]{Taoparaproducts}. Thus,
\begin{equation}\label{highhigh}
    \||\partial_x|^s\pi_3'(f,g)\|_{L^2} = \|\pi^r(f,|\partial_x|^sg)\|_{L^2} \lesssim \|f\|_{L^2}\||\partial_x|^sg\|_{L^{\infty}}.
\end{equation}

The proof concludes by adding (\ref{lowhigh}) and (\ref{highhigh}).
\end{proof}

\begin{lemma}\label{claimiscoifmanmeyer}
Let $\phi \in C^{\infty}(\mathbb{T})$ and $\psi \in C^{\infty}((0,\infty),[0,1])$ with $\psi(r)=1$ for $r \geq 1/4$ and $\psi(r)=0$ for $r\leq1/8$. The symbol $m(\xi_1,\xi_2):=\phi(\theta(\xi_1,\xi_2))\psi(r(\xi_1,\xi_2))$ defines a Coifman-Meyer multiplier operator, that is, it satisfies estimate (\ref{CoifmanMeyer}). 
\end{lemma}
\begin{proof}
By product rule,
$$\partial_{\xi_1}^{\alpha} \partial_{\xi_2}^{\beta}m(\xi_1,\xi_2)=\sum_{\alpha_1+\alpha_2 = \alpha, \ \beta_1+\beta_2 = \beta} C_{\alpha_1,\alpha_2,\beta_1,\beta_2} \cdot \partial_{\xi_1}^{\alpha_1}\partial_{\xi_2}^{\beta_1}(\phi(\theta)) \cdot \partial_{\xi_1}^{\alpha_2} \partial_{\xi_2}^{\beta_2}(\psi(r)),$$
where $C_{\alpha_1,\alpha_2,\beta_1,\beta_2}$ are suitable multinomial coefficients.

If any of the derivatives fall on the second factor $\psi(r)$, i.e. if $\alpha_2>0$ or $\beta_2>0$, then compactness of $\supp \ \partial_r\psi(r)$ tells us that the term $\partial_{\xi_1}^{\alpha_1}\partial_{\xi_2}^{\beta_1}\phi(\theta) \cdot \partial_{\xi_1}^{\alpha_2} \partial_{\xi_2}^{\beta_2}\psi(r)$ vanishes unless $|\xi_1|+|\xi_2| \sim r \sim 1$. Further, using compactness of $\supp \ \phi \subseteq \mathbb{T}$ as well, this term is bounded. In the non-vanishing case we can thus bound 
$$|\partial_{\xi_1}^{\alpha_1}\partial_{\xi_2}^{\beta_1}\phi(\theta) \cdot \partial_{\xi_1}^{\alpha_2} \partial_{\xi_2}^{\beta_2}\psi(r)| \lesssim 1 \sim (|\xi_1|+|\xi_2|)^{-\alpha-\beta}.$$

Otherwise, all the derivatives fall on the first factor, and we want to bound $[\partial_{\xi_1}^{\alpha} \partial_{\xi_2}^{\beta} \phi(\theta)] \psi(r)$.
The second factor $\psi(r)$ is bounded by assumption.
We can estimate
$$|\partial_{\xi_1}^{\alpha} \partial_{\xi_2}^{\beta} \phi(\theta)| \lesssim (|\xi_1|+|\xi_2|)^{-\alpha-\beta}$$
by chain rule and using the fundamental estimate \begin{equation*}|\partial_{\xi_1}^{\tilde{\alpha}}\partial_{\xi_2}^{\tilde{\beta}}\theta| \lesssim r^{-\tilde{\alpha}-\tilde{\beta}} \sim (|\xi_1|+|\xi_2|)^{-\tilde{\alpha}-\tilde{\beta}} \quad \text{for any } \tilde{\alpha},\tilde{\beta} \in \mathbb{N}.\qedhere \end{equation*}
\end{proof}

We record the following version of the fractional chain rule, which will be used in the next section.

\begin{lemma}[Fractional chain rule]\label{chainrule} Let $s \in (0,1)$ and $1<p<\infty$. For any function $F \in H^s(\mathbb{T}^2)$ with $e^{iF} \in L^{\infty}(\mathbb{T}^2)$ we have that
    \begin{equation}\label{fractionalchainruleinequality}\||\nabla|^se^{iF}\|_{L^p} \lesssim \| |\nabla|^sF\|_{L^p} \|e^{iF}\|_{L^{\infty}}.\end{equation}
    In particular, this estimate holds for any real-valued function $F \in H^s(\mathbb{T}^2)$. 
\end{lemma}

\begin{proof} 
In \cite{fractionalchainruletorus}, the Christ--Weinstein fractional chain rule \cite{MR1124294} is shown for functions on the torus by adapting the proof in \cite[Chapter 2]{Taylor2000ToolsFP} to the periodic setting. 
To apply the endpoint case of \cite[Proposition 3]{fractionalchainruletorus} to the function $v \mapsto e^{iv}$, we verify the hypothesis that $|e^{iv}-e^{iu}| \lesssim |u-v|(G(u)+G(v))$ for some function $G:\mathbb{C} \rightarrow [0,\infty)$.
Indeed, denoting by $[u,v]$ the straight line path connecting $u,v \in \mathbb{C}$, we have that $$|e^{iu} -e^{iv}| \leq  \big| \int_{[u,v]} -i e^{iz} dz \big| \leq |u-v| \sup_{t \in [0,1]} |e^{i(v+t(u-v))}| \leq |u-v|[e^{-\Re(u)}+e^{-\Re(v)}],$$
and \cite[Proposition 3]{fractionalchainruletorus} yields the result. 
\end{proof}

\subsection{Part $\RomanII$ of the bootstrap: estimates on the primitive}\label{partIIbootstrap}
We prove part $\RomanII$ \eqref{II} of the bootstrap conclusion \eqref{C}. All spacetime norms in this section are with respect to the time interval $[0,T']$. In particular, we use the notation $\|\cdot\|_{L^pL^q}$ as an abbreviation for
the mixed Lebesgue spacetime norm $\|\cdot\|_{L^p([0,T'],L^q(\mathbb{T}^2))}$.
\\

Differentiating the defining equation \eqref{gaugesecondappearance} for the gauge transformed function $w$ and rearranging as in \cite{MR2052470} gives 
\begin{equation*} iw_x = P_{+hi}(F_xe^{-iF})= [1 - (P_{lo}+P_{-hi})](P_{+HI}(F_x)e^{-iF}) + P_{+hi}([P_{-HI}+P_{LO}](F_x)e^{-iF}).\end{equation*} We obtain 
\begin{equation}\label{Fxequation}
    P_{+HI}(F_{x}) = ie^{iF}w_x + ie^{iF}E
\end{equation}
with
\begin{equation*}
    E= \underbrace{(P_{lo}+P_{-hi})(P_{+HI}(F_x)e^{-iF})}_{=:E_1} + \underbrace{P_{+hi}(P_{LO}(F_x)e^{-iF})}_{=:E_2} + \underbrace{P_{+hi}(P_{-HI}(F_x)e^{-iF})}_{=:E_3}.
\end{equation*}
Due to the frequency projections we may replace $e^{-iF}$ by $P_{-hi}(e^{-iF})$ and $P_{+hi}(e^{-iF})$ in $E_1$ and $E_3$ respectively. That is, we have $E_1=(P_{lo}+P_{-hi})(P_{+HI}(F_x)P_{-hi}(e^{-iF}))$ and $E_3=P_{+hi}(P_{-HI}(F_x)P_{+hi}(e^{-iF})).$
\\

In the following, we refer to terms involving $w$ and derivatives thereof as main terms and terms involving $E$ and derivatives thereof as error terms.
We may differentiate equation (\ref{Fxequation}) once more, yielding
\begin{equation}\label{Fxxequation}P_{+HI}(F_{xx}) = ie^{iF}w_{xx}+i\partial_x(e^{iF})w_x + ie^{iF}E_x +i\partial_x(e^{iF})E.\end{equation}

In both equations (\ref{Fxequation}) and (\ref{Fxxequation}) we have equality pointwise with both sides being continuous.  Indeed, as already observed $w,F,e^{-iF}$ are in $H^{s_1,s_2}_{\omega}$ due to the definition of $X$-norm. By boundedness property \eqref{boundedness1} in Lemma \ref{anisotropicproperties} and by the algebra property, both sides of (\ref{Fxequation}), (\ref{Fxxequation}) are in $H^{\tilde{s}_1, \tilde{s}_2}_{\omega}$ with $\tilde{s}_1=s_1-2,\tilde{s}_2=\frac{s_1-2}{s_1}s_2$. Provided $s_1$ is sufficiently large, anisotropic Sobolev embedding \eqref{anisotropicembedding} is applicable establishing pointwise equality and continuity. \\

Equations (\ref{Fxequation}) and (\ref{Fxxequation}) are central to gain control on the terms in part $\RomanII$ \eqref{II} of the bootstrap conclusion. This is done in Lemmas \ref{IIb}-\ref{IIc} below. In proving those lemmas, it will be useful to have the function space defined below and the estimates in Lemma \ref{eiFZ} available.

\begin{definition}
    Define $Z$ to be the space $$Z:=\{g:\mathbb{T}^2 \rightarrow \mathbb{C} \ | \  \|g\|_Z:=\|g\|_{H^{\eta,p}}+\|\partial_xg\|_{H^{\eta,p}} < \infty \}.$$
(Here $H^{\eta,p}(\mathbb{T}^2)$ is as usual the Sobolev space defined by the norm $\|\langle \nabla \rangle^{\eta}g\|_{L^p(\mathbb{T}^2)}.$)
\end{definition}

By Sobolev embedding $H^{\eta,p}(\mathbb{T}^2) \hookrightarrow L^{\infty}(\mathbb{T}^2)$ and fractional product rule we have that $Z$ is an algebra with $$\|g\cdot f\|_Z \lesssim \|g\|_Z \cdot \|f\|_Z.$$ Part $\RomanI$ of the bootstrap conclusion, which we have already established, implies that $\|w_{xx}\|_{L^4Z} \lesssim \varepsilon^2.$ Equation \ref{Fxequation} and the algebra property of $Z$ will allow us to conclude $\|P_{+HI}(F_{xx})\|_{L^4Z} \lesssim \varepsilon^2$ once we have established that $\|e^{iF}\|_{L^{\infty}Z} \lesssim 1$ and
$\|E\|_{L^4Z} \lesssim \varepsilon^2$. But this implies the desired bound on the first term in part $\RomanII$ of the bootstrap conclusion \eqref{C} for positive high frequencies. This automatically takes care of the negative high frequencies as well, due to $F$ being real-valued so that $P_{-HI}(F_{xx})=\overline{P_{+HI}(F_{xx})}$.
We will deal with low frequencies separately using energy bounds, see \eqref{lowfrequencies}. The next lemma establishes the desired bound $\|e^{iF}\|_{L^{\infty}Z} \lesssim 1$. The desired bound on the error term, $\|E\|_{L^4Z} \lesssim \varepsilon^2$, will be shown in Lemma \ref{errortermslemma}.

\begin{lemma}\label{eiFZ}
      Assume the bootstrap hypothesis holds and $0 \leq \eta<(2\sigma^2-1)/(2\sigma+2)$. Then $$\|e^{iF}\|_{L^{\infty}Z} \lesssim 1.$$
      More precisely,
      \begin{equation}\label{eiFZ1}\|\langle \nabla \rangle^{\eta}e^{iF}\|_{L^{\infty}L^p} \lesssim 1 \quad \text{and} \quad
         \|\langle \nabla \rangle^{\eta}\partial_xe^{iF}\|_{L^{\infty}L^p} \lesssim \varepsilon^2. 
      \end{equation}
      Further, $e^{iF}, \partial_xe^{iF} \in H^{s+\eta}$ with \begin{equation}\label{eiFZ2}
          \|\langle \nabla \rangle^{r+\eta}e^{iF}\|_{L^{\infty}L^2} \lesssim 1 \quad \text{and} \quad
          \|\langle \nabla \rangle^{r+\eta}\partial_xe^{iF}\|_{L^{\infty}L^2} \lesssim \varepsilon^2.
      \end{equation}
\end{lemma}

\begin{proof}
    The growth of the $H^{1+\sigma, \sigma}$-norm is controlled by energy arguments; Lemma \ref{Hshnormgrowth}
    gives \begin{equation}\label{Hr1rgrowth}\|u(t)\|_{H^{1+\sigma, \sigma}} \lesssim \|u_0\|_{H^{1+\sigma, \sigma}} \exp \big(\int_0^t \|\partial_x u(s)\|_{L^{\infty}} ds \big) \lesssim \|u_0\|_Y \lesssim \varepsilon^2 \lesssim 1,\end{equation}
    where we used $\int_0^t\|\partial_x u(s) \|_{L^{\infty}}ds \leq \exp (\varepsilon) \lesssim 1$ by part $\RomanII$ \eqref{II} of the bootstrap hypothesis. 
    Similarly, by Lemma \ref{Fyqnormgrowth}, we have 
    \begin{align*}
       \|\langle \partial_y \rangle^{\sigma} F(t)\|_{L^2} 
       & \lesssim [\|\langle \partial_y \rangle^{\sigma} F(0)\|_{L^2} + \|\langle \partial_y \rangle^{\sigma} u_0\|_{L^2} ] \exp \Big(\int_0^t \|\partial_xu(s)\|_{L^{\infty}} + \|u(s)\|_{L^{\infty}} ds \Big) \\ &\lesssim \|u_0\|_Y \exp (\varepsilon) \lesssim \varepsilon^2,
    \end{align*}
    again using part $\RomanII$ of the bootstrap hypothesis in the second inequality. Combining with (\ref{Hr1rgrowth}) (and $F_x=u$) yields \begin{equation}\label{eq16} \|F\|_{L^{\infty}H^{1+\sigma, \sigma}_{\omega}} \leq\|F\|_{L^{\infty}H^{2+\sigma,\sigma}_{\omega}} \lesssim \varepsilon^2 \lesssim 1.\end{equation}
    \\

We have an embedding of normed spaces $H_{\omega}^{1+\sigma,\sigma} \hookrightarrow H^{\eta,\infty} \hookrightarrow H^{\eta,p}$ by combining boundedness property (2) in Lemma \ref{anisotropicproperties} and anisotropic Sobolev embedding \eqref{anisotropicembedding}. Invoking the latter is valid by our hypothesis on $\eta$. 
By \eqref{eq16}, this embedding gives
\begin{equation}\label{est0}
\||\nabla|^{\eta}F\|_{L^{\infty}L^p}\lesssim \|\langle \nabla \rangle^{\eta}F\|_{L^{\infty}L^{p}} \lesssim \|F\|_{L^{\infty}H^{1+\sigma,\sigma}_{\omega}} \lesssim \varepsilon^2,
\end{equation}
and by \eqref{Hr1rgrowth},
\begin{equation}\label{est1}
\|F_x\|_{L^{\infty}L^{\infty}}=\|u\|_{L^{\infty}L^{\infty}} \lesssim \|\langle \nabla \rangle^{\eta}F_x\|_{L^{\infty}L^{\infty}}=\|\langle \nabla \rangle^{\eta}u\|_{L^{\infty}L^p} \lesssim \|u\|_{L^{\infty}H^{1+\sigma,\sigma}_{\omega}} \lesssim \varepsilon^2.
\end{equation}
\\

Now we prove (\ref{eiFZ1}). By fractional chain rule \eqref{fractionalchainruleinequality}, estimate \eqref{est0} gives
\begin{equation}\label{est3}\|\langle \nabla \rangle^{\eta}e^{iF}\|_{L^{\infty}L^p} \lesssim \|e^{iF}\|_{L^{\infty}L^p}+\||\nabla|^{\eta}F\|_{L^{\infty}L^p}\|e^{iF}\|_{L^{\infty}L^{\infty}}\lesssim 1 + \varepsilon^2 \lesssim 1.\end{equation}
The second estimate in (\ref{eiFZ1}) follows from fractional product rule and estimates (\ref{est1}), (\ref{est3}):
\begin{equation}\|\langle \nabla \rangle^{\eta}\partial_xe^{iF}\|_{L^{\infty}L^p} \lesssim \|F_x\|_{L^{\infty}L^{\infty}}\|\langle \nabla \rangle^{\eta}e^{iF}\|_{L^{\infty}L^p} + \|\langle \nabla \rangle^{\eta}F_x\|_{L^{\infty}L^p} \|e^{iF}\|_{L^{\infty}L^{\infty}} \lesssim \varepsilon^2.\end{equation}
\\

We turn to (\ref{eiFZ2}).
The first estimate in (\ref{eiFZ2}) reduces to $\||\nabla|^{r+\eta}F\|_{L^{\infty}L^2} \lesssim 1$ by fractional chain rule \eqref{fractionalchainruleinequality}, and this is true by (\ref{eq16}). To obtain the second estimate in (\ref{eiFZ2}), note that by the energy estimate (\ref{Hr1rgrowth}) \begin{equation}\label{eq5}\|\langle \nabla \rangle^{r+\eta}F_x\|_{L^{\infty}L^2} \lesssim \varepsilon^2.\end{equation} Thus, fractional product rule, the first estimate in (\ref{eiFZ2}), and estimate (\ref{est1}) give
\begin{equation*}\|\langle \nabla \rangle^{r+\eta}\partial_xe^{iF}\|_{L^{\infty}L^2} \lesssim \|\langle \nabla \rangle^{r+\eta}F_x\|_{L^{\infty}L^2} \|e^{iF}\|_{L^{\infty}L^{\infty}} +\|F_x\|_{L^{\infty}L^{\infty}}\|\langle \nabla \rangle^{r+\eta}e^{iF}\|_{L^{\infty}L^2} \lesssim \varepsilon^2. \qedhere \end{equation*}
\end{proof}
\hfill

We now bound the first summand in part $\RomanII$ \eqref{II} of the bootstrap conclusion \eqref{C}.
\begin{lemma}\label{IIb} Assume the bootstrap hypothesis holds and $0 \leq \eta<(2\sigma^2-1)/(2\sigma+2)$. We have
    $$\|\langle \nabla \rangle^{\eta}(F_{xx})\|_{L^4L^p} \lesssim \varepsilon^2.$$
\end{lemma}
\begin{proof}
We prove that $$\|\langle \nabla \rangle^{\eta}P_{LO}(F_{xx})\|_{L^4L^p} \lesssim \varepsilon^2 \quad \text{and} \quad \|\langle \nabla \rangle^{\eta}P_{+HI}(F_{xx})\|_{L^4L^p} \lesssim \varepsilon^2.$$
This is sufficient by Littlewood-Paley decomposition and due to the fact that $P_{-HI}(F_{xx}) = \overline{P_{+HI}(F_{xx})}$ as $F$ is real-valued.
For the low frequency part we estimate 

\begin{align}\begin{split}\label{lowfrequencies}\|\langle \nabla \rangle^{\eta}P_{LO}(F_{xx})\|_{L^4L^p}&  \lesssim \|\langle \nabla \rangle^{\eta}P_{LO}u_x\|_{L^{\infty}L^{\infty}} \\ & \lesssim \|\langle \nabla \rangle^{\eta}u\|_{L^{\infty}H^{1+\sigma-\eta\frac{1+\sigma}{\sigma}, \sigma- \eta}} \lesssim \|u\|_{L^{\infty}H^{1+\sigma,\sigma}} \lesssim\varepsilon^2.\end{split}\end{align}
We used anisotropic Sobolev embedding \eqref{anisotropicembedding}, boundedness of $\partial_xP_{LO}$ on $L^2$-based Sobolev spaces, boundedness property \eqref{boundedness2} in Lemma \ref{anisotropicproperties}, and the energy estimate (\ref{Hr1rgrowth}). We require the hypothesis on $\eta$ for anisotropic Sobolev embedding \eqref{anisotropicembedding}. \\

We turn to the high positive frequency part. Lemma \ref{eiFZ} asserts that $\|e^{iF}\|_{L^{\infty}Z} \lesssim 1$. Due to equation \ref{Fxequation} and the algebra property of $Z$, it suffices to show $\|w_x\|_{L^4Z} \lesssim \varepsilon^2$ and $\|E\|_{L^4Z} \lesssim \varepsilon^2.$
But the former statement is immediate from part $\RomanI$ \eqref{I} of the bootstrap conclusion, which we have already established.  The latter statement is the content of Lemma \ref{errortermslemma} in the next section.
\end{proof}
\hfill

We turn to the second summand in part $\RomanII$ \eqref{II} of the bootstrap conclusion \eqref{C}.
\begin{lemma}\label{IIc}
    Assume the bootstrap hypothesis holds and $0 \leq \eta<(2\sigma^2-1)/(2\sigma+2)$. We have $$\|\langle\nabla\rangle^{r+\eta}F_{xx}\|_{L^{4/3}L^2} \lesssim \varepsilon^2.$$
\end{lemma}

\begin{proof}
    Apply $\langle \nabla\rangle ^{r+\eta}$ to equation (\ref{Fxxequation}). We begin bounding the contribution by the two main terms using fractional product rule, yielding
    \begin{equation*}
        \|\langle \nabla \rangle^{r+\eta}(e^{iF}w_{xx})\|_{L^{4/3}L^2} \lesssim \|e^{iF}\|_{L^{\infty}L^{\infty}}\|\langle\nabla\rangle^{r+\eta}w_{xx}\|_{L^{\infty}L^2} + \|\langle\nabla\rangle^{r+\eta}e^{iF}\|_{L^{\infty}L^2} \|w_{xx}\|_{L^4L^{\infty}}
        \end{equation*}
   and     
        \begin{equation*} \|\langle\nabla\rangle^{r+\eta}((\partial_xe^{iF})w_{x})\|_{L^{\infty}L^2} \lesssim \|\langle\nabla\rangle^{r+\eta}\partial_xe^{iF}\|_{L^{\infty}L^2} \|w_x\|_{L^4L^{\infty}} + \|\partial_xe^{iF}\|_{L^{\infty}L^{\infty}} \|\langle \nabla \rangle^{r+\eta}w_x\|_{L^4L^{\infty}}.
    \end{equation*}
    The terms $\|\langle \nabla \rangle^{r+\eta}w_x\|_{L^{\infty}L^2}$ and $\|\langle \nabla \rangle^{r+\eta}w_{xx}\|_{L^{\infty}L^2}$ are $O(\varepsilon^2)$ by the energy estimates in Section \ref{sectionenergyestimatepartI}, Lemma \ref{wenergyestimate}. The instances where terms involving $w$ and derivatives thereof are measured in $L^4L^{\infty}$ are $O(\varepsilon^2)$ by part $\RomanI$ \eqref{I} of the bootstrap conclusion which we have already established. 
    The remaining terms are $O(1)$ by Lemma \ref{eiFZ} together with Sobolev embedding, so that the overall expression is $O(\varepsilon^2)$. \\

    Similarly, we bound the contribution by the error terms, replacing the role of $w_x$ by $E$. 
    Fractional product rule gives
\begin{equation*}
\|\langle\nabla\rangle^{r+\eta}ie^{iF}E_x\|_{L^{4/3}L^2} \lesssim \|\langle\nabla\rangle^{r+\eta}e^{iF}\|_{L^{\infty}L^2} \|E_x\|_{L^4L^{\infty}} + \|e^{iF}\|_{L^{\infty}L^{\infty}}\|\langle\nabla\rangle^{r+\eta}E_x\|_{L^{4/3}L^2}
\end{equation*}
and 
\begin{equation*}
    \|\langle \nabla \rangle^{r+\eta}((\partial_xe^{iF})E)\|_{L^{4/3}L^2} \lesssim \|\langle \nabla \rangle^{r+\eta} \partial_xe^{iF}\|_{L^{\infty}L^2}\|E\|_{L^4 L^{\infty}}  + \|\partial_xe^{iF}\|_{L^{\infty}L^{\infty}}\|\langle \nabla \rangle^{r+\eta}E\|_{L^{4/3}L^2}.
\end{equation*}
Thus, it suffices to show that $$\|E\|_{L^4L^{\infty}}+ \|E_x\|_{L^4L^{\infty}}+ \|\langle \nabla \rangle^{r+\eta}E\|_{L^{4/3}L^2}+ \|\langle \nabla \rangle^{r+\eta}E_x\|_{L^{4/3}L^2} \lesssim \varepsilon^2,$$ the other factors being the same bounded factors as before. 
But this is the content of Lemma \ref{errortermslemma} below. 
\end{proof}

\subsection{Part $\RomanII$ of the bootstrap: estimates on error terms}\label{sectionerrorterms}
We prove the error bounds that were used in the previous section. Throughout this section, we use the notation $\|\cdot\|_{L^pL^q}$ as an abbreviation for mixed Lebesgue spacetime norms $\|\cdot\|_{L^p([0,T'],L^q(\mathbb{T}^2))}$.
\begin{lemma}\label{errortermslemma}
    Assume the bootstrap hypothesis holds and $0\leq \eta <(2\sigma^2-1)/(2\sigma+2)$. Then $$\|E\|_{L^4Z} \lesssim \varepsilon^2.$$
    That is, we have the following estimates:   \begin{equation}\label{firsterrorestimate}\|\langle \nabla \rangle^{\eta}E\|_{L^{4}L^p} \lesssim \varepsilon^2. \quad \text{In particular,} \ \|E\|_{L^{4}L^{\infty}} \lesssim \varepsilon^2. \end{equation} \begin{equation}\label{seconderrorestimate}\|\langle \nabla \rangle^{\eta}E_x\|_{L^{4}L^p} \lesssim \varepsilon^2. \quad \text{In particular,} \ \|E_x\|_{L^{4}L^{\infty}} \lesssim \varepsilon^2. \end{equation}
Further, we have that $E,E_x \in L^{4/3}H^{r+\eta}$ with
\begin{equation}\label{thirderrorterm} \|\langle\nabla\rangle^{r+\eta}E\|_{L^{4/3}L^2} \lesssim \varepsilon^2 \end{equation}
\begin{equation}\label{fourtherrorterm} \|\langle\nabla\rangle^{r+\eta}E_x\|_{L^{4/3}L^2} \lesssim \varepsilon^2. \end{equation}
\end{lemma}

\begin{proof}
Recall that
\begin{equation*}
    E= \underbrace{(P_{lo}+P_{-hi})(P_{+HI}(F_x)P_{-hi}(e^{-iF}))}_{=E_1} + \underbrace{P_{+hi}(P_{LO}(F_x)e^{-iF})}_{=E_2} + \underbrace{P_{+hi}(P_{-HI}(F_x)P_{+hi}(e^{-iF}))}_{=E_3}.
\end{equation*}
For each estimate (\ref{firsterrorestimate})-(\ref{fourtherrorterm}) and each term $E_i,i=1,2,3$, the first step in below calculations is to apply fractional product rule, and $L^p$-boundedness of the Littlewood-Paley projections involved, see Remark \ref{whyLp}. Each of the resulting factors is then estimated by either the bootstrap hypothesis,
Lemma \ref{eiFZ}, Lemma \ref{furtherestimates}, or Lemma \ref{usefulestimates2} giving the overall bound $O(\varepsilon^2)$.
The required estimates on terms arising from the factors $P_{-hi}e^{-iF}$ and $P_{+hi}e^{-iF}$ in the terms $E_1$ and $E_3$ are proven in Lemma \ref{usefulestimates2} below. The required estimates on the terms arising from $P_{\pm HI}(F_x)$ and $P_{LO}(F_x)$ are collected in Lemma \ref{furtherestimates}. The terms arising from the factor $e^{-iF}$ in $E_2$ can be controlled using Lemma \ref{eiFZ}. We now provide the details.
\\

First, we prove estimate (\ref{firsterrorestimate}). We begin estimating $\|\langle \nabla \rangle^{\eta}E_1\|_{L^4L^p}$. By fractional product rule in the spatial variables and H\"older in time, we have
\begin{align*} &\|\langle \nabla \rangle^{\eta}E_1\|_{L^4L^p} \lesssim \|\langle \nabla \rangle^{\eta}P_{+HI}(F_x)\|_{L^4L^p} \|P_{-hi}e^{-iF}\|_{L^{\infty}L^{\infty}} +\|P_{+HI}(F_x)\|_{L^{4}L^{\infty}} \|\langle \nabla \rangle^{\eta}P_{-hi}e^{-iF}\|_{L^{\infty}L^p}.\end{align*}
Lemma \ref{usefulestimates2} below asserts that $\|P_{-hi}e^{-iF}\|_{L^4L^p}+\|\langle \nabla \rangle^{\eta}P_{-hi}e^{-iF}\|_{L^{\infty}L^p} = O(\varepsilon^2)$. Invoke estimate (\ref{PHIFxLp}) in Lemma \ref{furtherestimates}
to conclude that the overall expression is $O(\varepsilon^2)$.
The term $\|\langle \nabla \rangle^{\eta}E_3\|_{L^4L^p}$ can be estimated to be $O(\varepsilon^2)$ analogously, using the estimates in Lemma \ref{usefulestimates2} with $P_{-hi}(e^{-iF})$ replaced by $P_{+hi}(e^{-iF})$, and the estimate (\ref{PHIFxLp}) in Lemma \ref{furtherestimates} with $P_{+HI}$ replaced by $P_{-HI}$. Further, we estimate by fractional product rule 
\begin{equation*}
    \|\langle \nabla \rangle^{\eta}E_2\|_{L^4L^p} \lesssim \|\langle \nabla \rangle^{\eta}P_{LO}(F_x)\|_{L^4L^p} \|e^{-iF}\|_{L^{\infty}L^{\infty}} + \|P_{LO}(F_x)\|_{L^{4}L^{\infty}} \|\langle \nabla \rangle^{\eta}e^{-iF}\|_{L^{\infty}L^p}
\end{equation*}
which is $O(\varepsilon^2)$ due to estimate (\ref{LoFxLp}) in Lemma \ref{furtherestimates} and Lemma \ref{usefulestimates2}. \\

Next, we prove estimate (\ref{seconderrorestimate}). We begin with the contribution by $E_1$. By boundedness of frequency projections on $L^p$, fractional chain rule in the spatial variables and H\"older in time, we have 
\begin{align*} & \|\langle \nabla \rangle^{\eta}(E_1)_x\|_{L^4L^p} \\ & \lesssim \|\langle \nabla \rangle^{\eta}P_{+HI}(F_{xx})\|_{L^4L^p} \|P_{-hi}e^{-iF}\|_{L^{\infty}L^{\infty}} + \|P_{+HI}(F_{xx})\|_{L^4L^{\infty}} \|\langle \nabla \rangle^{\eta}P_{-hi}e^{-iF}\|_{L^{\infty}L^{p}} \\ & + \|\langle \nabla \rangle^{\eta}P_{+HI}(F_x)\|_{L^4L^p} \|\partial_xP_{-hi}(e^{-iF})\|_{L^{\infty}L^{\infty}} + \|P_{+HI}(F_x)\|_{L^4L^{\infty}} \|\langle \nabla \rangle^{\eta} \partial_xP_{-hi}(e^{-iF})\|_{L^{\infty}L^{p}}.\end{align*}
By the bootstrap hypothesis, $\|\langle \nabla \rangle^{\eta}P_{+HI}(F_{xx})\|_{L^4L^p} + \|P_{+HI}(F_{xx})\|_{L^4L^{\infty}} = O(\varepsilon)$. Each of the remaining factors is $O(\varepsilon^2)$ either by Lemma \ref{usefulestimates2} below or by (\ref{PHIFxLp}) in Lemma \ref{furtherestimates}, thus bounding the entire expression by $O(\varepsilon^2)$. The contribution of $E_3$ can be shown to be $O(\varepsilon^2)$ analogously. Finally, the contribution of $E_2$ is  \begin{align*}\|\langle \nabla \rangle^{\eta}(E_2)_x\|_{L^4L^p} & \lesssim \|\langle \nabla \rangle^{\eta}P_{LO}(F_{xx})\|_{L^4L^p}\|e^{-iF}\|_{L^{\infty}L^{\infty}} + \|P_{LO}(F_{xx})\|_{L^{4}L^{\infty}}\|\langle \nabla \rangle^{\eta}e^{-iF}\|_{L^{\infty}L^p} \\ & + \|\langle \nabla \rangle^{\eta}P_{LO}(F_{x})\|_{L^4L^p}\|\partial_xe^{-iF}\|_{L^{\infty}L^{\infty}} + \|P_{LO}(F_{x})\|_{L^{4}L^{\infty}}\|\langle \nabla \rangle^{\eta}\partial_xe^{-iF}\|_{L^{\infty}L^p} \\ & \lesssim \varepsilon^2.\end{align*}
Lemma \ref{eiFZ} asserts that all the factors arising from $e^{-iF}$ are $O(1)$, while estimates (\ref{LoFxLp}) and (\ref{LoFxxLp}) in Lemma \ref{furtherestimates} bound the remaining factors by $O(\varepsilon^2)$. \\

Next, we turn to (\ref{thirderrorterm}). By fractional product rule,
\begin{align*} & \|\langle \nabla \rangle^{r+\eta}E_1\|_{L^{4/3}L^2} \\ & \lesssim \|\langle \nabla \rangle^{r+\eta}P_{+HI}(F_x)\|_{L^{4/3}L^2} \|P_{-hi}(e^{-iF})\|_{L^{\infty}L^{\infty}} + \|P_{+HI}(F_x)\|_{L^4L^{\infty}} {\|\langle \nabla \rangle^{r+\eta} P_{-hi}(e^{-iF})\|_{L^{\infty}L^{2}}}.\end{align*}
We have $\|\langle \nabla \rangle^{r+\eta}P_{+HI}(F_x)\|_{L^{4/3}L^2} = O(\varepsilon^2)$ 
by estimate (\ref{PHIFx432}) 
and $\|P_{+HI}(F_x)\|_{L^4L^{\infty}} = O(\varepsilon^2)$
by estimate (\ref{PHIFxLp}) in Lemma \ref{furtherestimates}. The other two factors are bounded in Lemma \ref{usefulestimates2}, thus bounding the entire expression by $O(\varepsilon^2)$. Analogously, the expression $\|\langle \nabla \rangle^{\eta}E_3\|_{L^{4/3}L^2}$ is $O(\varepsilon^2)$. The contribution by $E_2$ is controlled by
$$\|\langle \nabla \rangle^{r+\eta}E_2\|_{L^{4/3}L^2} \lesssim \|\langle \nabla \rangle^{r+\eta}P_{LO}(F_x)\|_{L^{4/3}L^2}\|e^{-iF}\|_{L^{\infty}L^{\infty}} + \|P_{LO}(F_x)\|_{L^{\infty}L^{\infty}}\|\langle \nabla \rangle^{r+\eta}e^{-iF}\|_{L^{\infty}L^2}.$$
We invoke estimates (\ref{PHIFx432}), (\ref{LOFxinfty}), and Lemma \ref{eiFZ} (together with Sobolev embedding) to conclude that the overall expression is $O(\varepsilon^2)$. \\

We turn to estimate (\ref{fourtherrorterm}). First, we bound the contribution by $E_1$.
\begin{align*}& \|\langle \nabla \rangle^{r+\eta}(E_1)_x\|_{L^{4/3}L^2} \\ & \lesssim \|\langle \nabla \rangle^{r+\eta}P_{+HI}(F_{xx})\|_{L^{4/3}L^2} \|P_{-hi}e^{-iF}\|_{L^{\infty}L^{\infty}}  + \|P_{+HI}(F_{xx})\|_{L^4L^{\infty}} \|\langle \nabla \rangle^{r+\eta}P_{-hi}e^{-iF}\|_{L^{\infty}L^2} \\ & + \|\langle \nabla \rangle^{r+\eta}P_{+HI}(F_x)\|_{L^{4/3}L^2} \|\partial_xP_{-hi}(e^{-iF})\|_{L^{\infty}L^{\infty}}  + \|P_{+HI}(F_x)\|_{L^4L^{\infty}} \|\langle \nabla \rangle^{r+\eta} \partial_xP_{-hi}(e^{-iF})\|_{L^{\infty}L^{2}} \\ & \lesssim \varepsilon^2.\end{align*}
By bootstrap hypothesis, we have $\|\langle \nabla \rangle^{r+\eta}P_{+HI}(F_{xx})\|_{L^{4/3}L^2} + \|P_{+HI}(F_{xx})\|_{L^4L^{\infty}} = O(\varepsilon)$. 
Further, we again have $\|\langle \nabla \rangle^{r+\eta}P_{+HI}(F_x)\|_{L^{4/3}L^2} = O(\varepsilon^2)$ by estimate (\ref{PHIFx432}) in Lemma \ref{furtherestimates} and $\|P_{+HI}(F_x)\|_{L^4L^{\infty}}=O(\varepsilon^2)$ by estimate (\ref{PHIFxLp}) in Lemma \ref{furtherestimates}. 
The remaining factors are bounded in Lemma \ref{usefulestimates2}.
Analogously, we have $$\|\langle \nabla \rangle^{r+\eta}(E_3)_x\|_{L^{4/3}L^2} \lesssim \varepsilon^2.$$
Finally, we bound the contribution by $E_2$.
\begin{align*}
    & \|\langle \nabla \rangle^{r+\eta}(E_2)_x\|_{L^{4/3}L^2} \\ & \lesssim \|\langle \nabla \rangle^{r+\eta}P_{LO}(F_{xx})\|_{L^{4/3}L^2}\|e^{-iF}\|_{L^{\infty}L^{\infty}} + \|P_{LO}(F_{xx})\|_{L^4L^{\infty}}\|\langle \nabla \rangle^{r+\eta} e^{-iF}\|_{L^{\infty}L^2} \\ & + \|\langle \nabla \rangle^{r+\eta}P_{LO}(F_x)\|_{L^{4/3}L^2}\|\partial_x e^{-iF}\|_{L^{\infty}L^{\infty}} + \|P_{LO}(F_x)\|_{L^{\infty}L^{\infty}}\|\langle \nabla \rangle^{r+\eta}\partial_xe^{-iF}\|_{L^{\infty}L^2}  = O(\varepsilon^2).
\end{align*}
We used $\|\langle \nabla \rangle^{r+\eta}P_{LO}(F_{xx})\|_{L^{4/3}L^2}+ \|\langle \nabla \rangle^{r+\eta}P_{LO}(F_x)\|_{L^{4/3}L^2} =O(\varepsilon^2)$ by estimate (\ref{PHIFx432}) in Lemma \ref{furtherestimates} and $L^2$-boundedness of $P_{LO}\circ \partial_x$. As before, we have $\|P_{LO}(F_x)\|_{L^4L^{\infty}} \lesssim \varepsilon^2$ by (\ref{LoFxLp}) and $\|P_{LO}(F_{xx})\|_{L^4L^{\infty}} \lesssim \varepsilon^2$ by (\ref{LoFxxLp}) in Lemma \ref{furtherestimates}, respectively. 
The remaining factors are bounded by Lemma \ref{usefulestimates2}.
\end{proof}

\begin{remark}\label{whyLp}
Our argument uses boundedness of the operators $P_{+hi}$ and $P_{lo}+P_{-hi}$ appearing in the terms $E_1, E_2$ and $E_3$. This is the reason for working with the space $H^{\eta,p}(\mathbb{T}^2)$ instead of $L^{\infty}(\mathbb{T}^2)$.
\end{remark}

Lemma \ref{furtherestimates} below collects the estimates on terms arising from the factors $P_{\pm HI}(F_x)$ and $P_{LO}(F_x)$ that were used in the proof of Lemma \ref{errortermslemma}. Lemma \ref{usefulestimates2} establishes the estimates on terms arising from the factors $P_{-hi}e^{-iF}$ and $P_{+hi}e^{-iF}$. 

\begin{lemma}\label{furtherestimates}
    Assume the bootstrap hypothesis holds and $0 \leq \eta < (2\sigma^2-1)/(2\sigma+2)$. We have the following estimates:
    \begin{equation}\label{PHIFxLp}
        \|P_{+HI}(F_x)\|_{L^4L^{\infty}} \lesssim \|\langle \nabla \rangle^{\eta}P_{+HI}(F_x)\|_{L^4L^p} = O(\varepsilon^2)
    \end{equation}
    \begin{equation}\label{LoFxLp}
        \|P_{LO}(F_x)\|_{L^4L^{\infty}} \lesssim \|\langle \nabla \rangle^{\eta}P_{LO}(F_x)\|_{L^4L^p} = O(\varepsilon^2)
    \end{equation}
\begin{equation}\label{LoFxxLp}
    \|P_{LO}(F_{xx})\|_{L^4L^{\infty}}\lesssim\|\langle \nabla \rangle^{\eta}P_{LO}(F_{xx})\|_{L^4L^p} = O(\varepsilon^2)
\end{equation}
\begin{equation}\label{LOFxinfty}
    \|P_{LO}(F_x)\|_{L^{\infty}L^{\infty}}=O(\varepsilon^2)
\end{equation}
\begin{equation}\label{PHIFx432}
    \|\langle \nabla \rangle^{r+\eta}P_{+HI}(F_x)\|_{L^{4/3}L^2} = O(\varepsilon^2) \quad \text{and} \quad \|\langle \nabla \rangle^{r+\eta}P_{LO}(F_x)\|_{L^{4/3}L^2} = O(\varepsilon^2)
\end{equation}
Analogous estimates to (\ref{PHIFxLp}) and (\ref{PHIFx432}) hold when replacing the Littlewood-Paley projection $P_{+HI}$ by $P_{-HI}$.
\end{lemma}

\begin{proof} These estimates have essentially already been established as part of the calculations in Section \ref{partIIbootstrap}. Estimates (\ref{PHIFxLp}) and (\ref{LoFxLp}) are just (\ref{est1}) (and $L^p$-boundedness of Littlewood-Paley projections). Estimate (\ref{PHIFx432}) is (\ref{eq5}) (and $L^2$-boundedness of Littlewood-Paley projections). Estimate (\ref{LoFxxLp}) is just (\ref{lowfrequencies}). Estimate (\ref{LOFxinfty}) follows directly from the energy bound (\ref{Hr1rgrowth}) by $L^2$-boundedness  of $P_{LO}$ and anisotropic Sobolev embedding \eqref{anisotropicembedding}. 
\end{proof}

\begin{lemma}\label{usefulestimates2}
    Assume the bootstrap hypothesis holds and $0 \leq \eta < (2\sigma^2-1)/(2\sigma+2)$. We have the following estimates:
     \begin{equation}\label{fiftherror} \|\langle \nabla \rangle^{\eta}\partial_xP_{-hi}(e^{-iF})\|_{L^{\infty}L^p} \lesssim \varepsilon^2. \quad \text{In particular,} \ \|\partial_x P_{-hi}(e^{-iF})\|_{L^{\infty}L^{\infty}} \lesssim \varepsilon^2. \end{equation}
    \begin{equation}\label{sixtherror} \|\langle \nabla \rangle^{\eta}P_{-hi}e^{-iF}\|_{L^{\infty}L^p} \lesssim \varepsilon^2. \quad \text{In particular,} \ \|P_{-hi}e^{-iF}\|_{L^{\infty}L^{\infty}} \lesssim \varepsilon^2. \end{equation}
    Further,
    \begin{equation}\label{seventherror}\|\langle \nabla \rangle^{r+\eta}\partial_x P_{-hi}(e^{-iF})\|_{L^{\infty}L^2} \lesssim \varepsilon^2 \quad \text{and} \quad \|\langle \nabla \rangle^{r+\eta}P_{-hi} e^{-iF}\|_{L^{\infty}L^2} \lesssim \varepsilon^2. \end{equation}
    Analogous estimates hold with $P_{-hi}$ replaced by $P_{+hi}$.
\end{lemma}

\begin{proof} For both of (\ref{fiftherror}) and (\ref{sixtherror}), the second statement follows from the first by Sobolev embedding. 
By $L^p$-boundedness of $P_{-hi}$ and fractional product rule we estimate
\begin{equation*} \|\langle \nabla \rangle^{\eta}\partial_xP_{-hi}(e^{-iF})\|_{L^{\infty}L^p} \lesssim \|\langle \nabla \rangle^{\eta}F_x\|_{L^{\infty}L^p}\|e^{-iF}\|_{L^{\infty}L^{\infty}}+\|F_x\|_{L^{\infty}L^{\infty}}\|\langle \nabla \rangle^{\eta}e^{-iF}\|_{L^{\infty}L^p}. \end{equation*} and conclude (\ref{fiftherror}) by invoking estimate \eqref{est1} 
and Lemma \ref{eiFZ}.\\

We turn to (\ref{sixtherror}). 
By fractional chain rule \eqref{fractionalchainruleinequality} and using that the frequency support is bounded away from zero, we have $$\|\langle \nabla \rangle^{\eta}P_{-hi}e^{-iF}\|_{L^{\infty}L^p} \sim \||\nabla|^{\eta}P_{-hi}e^{-iF}\|_{L^{\infty}L^p} \lesssim \||\nabla|^{\eta}F\|_{L^{\infty}L^p} \|e^{-iF}\|_{L^{\infty}L^{\infty}}.$$
But this is $O(\varepsilon^2)$ by \eqref{est0}. \\

Finally, we show (\ref{seventherror}). The second statement follows from the first by $L^2$-boundedness of $1/\partial_x$ on high-frequency functions. We estimate by fractional product rule
$$\|\langle \nabla \rangle^{r+\eta}\partial_x P_{-hi}(e^{-iF})\|_{L^{\infty}L^2} \lesssim \|\langle \nabla \rangle^{r+\eta}F_x\|_{L^{\infty}L^2}\|e^{-iF}\|_{L^{\infty}L^{\infty}} + \|F_x\|_{L^{\infty}L^{\infty}}\|\langle \nabla \rangle^{r+\eta}e^{-iF}\|_{L^{\infty}L^2}.$$
By estimates (\ref{est1}) and (\ref{eq5}) in the proof of Lemma \ref{eiFZ} we have $\|\langle \nabla \rangle^{r+\eta}u\|_{L^{\infty}L^2}, \|u\|_{L^{\infty}L^{\infty}} \lesssim \varepsilon^2$. The remaining factors are $O(1)$ by Lemma \ref{eiFZ}, so that the entire expression is $O(\varepsilon^2)$.
\end{proof}

\section{Finishing the proof: constructing local solutions for initial data in $Y$}\label{sectionfinisihingtheproof}
In this section we put everything so far together and prove our main theorem, Theorem \ref{maintheorem}. Combining the a-priori bound in Theorem \ref{aprioriformal} and the bound on norm growth in Proposition \ref{normgrowthX} with the local wellposedness theory in $X$ (Theorem \ref{wellposednessinX}), we obtain solutions to the Cauchy problem for regularized initial data on a \textit{uniform} time interval. 
Theorem \ref{maintheorem} then follows by standard arguments using Gronwall's inequality and compactness arguments. We provide the details.

\begin{proof}[Proof of Theorem \ref{maintheorem}]
Given initial data $u_0 \in B_Y(0,\rho_0)$, regularize by frequency cut-off yielding a sequence of approximate initial data $u_0^N = P_{N,N}u_0 \in Y \cap X$ with $u_0^N \rightarrow u_0$ in $Y$. See \eqref{generalprojection} for the definition of $P_{N,N}$. Using the a-priori bound in Theorem \ref{aprioriformal} and the bound on norm growth in Proposition \ref{normgrowthX} we can iterate the local wellposedness theory in $X$ (Theorem \ref{wellposednessinX}) and get solutions $u^N \in C([0,1],X)$ to the Cauchy problem with initial data $u_0^N$ on the uniform time interval $[0,1]$.

\begin{lemma}\label{Cauchynessfinalapproximation}
    The sequence of emanating solutions $u^N \in C([0,1],X)$ is Cauchy in $C([0,1],H_{\omega}^{1+\tilde{\sigma},\tilde{\sigma}})$ for any $\tilde{\sigma}<\sigma$.
\end{lemma}

\begin{proof}
    First, we show that the sequence is Cauchy in $C([0,1],L_{\omega}^2)$ by an easy energy argument using boundedness of $\|\partial_xu^N\|_{L^{\infty}}$ uniformly in $N$. We calculate using anti-selfadjointness of $H\partial_{xx}$ and integration by parts as several times before that 
    \begin{align*}
\frac{d}{dt}\|u^N-u^M\|_{L^2}^2 & = \frac{1}{2} \int_{\mathbb{T}^2} (u^N-u^M)^2 \partial_x(u^N+u^M) \leq \frac{1}{2} \|u^N-u^M\|_{L^2}^2[\|\partial_xu^N\|_{L^{\infty}}+\|\partial_xu^M\|_{L^{\infty}}].
    \end{align*}
    Gronwall's inequality yields
    $$\|(u^N-u^M)(t)\|_{L^2}^2 \leq \|u^N(0)-u^M(0)\|_{L^2}^2 \exp \Big(\frac{1}{2} \int_0^t \|\partial_xu^N(s)\|_{L^{\infty}}+\|\partial_xu^M(s)\|_{L^{\infty}} ds \Big)$$
    so that by the uniform a-priori bound in Theorem \ref{aprioriformal}
    $$\|(u^N-u^M)(t)\|_{L_t^{\infty}([0,1],L^2)}^2 \lesssim \|u^N(0)-u^M(0)\|_{L^2}^2 \rightarrow 0 \quad \text{as} \quad N>M\rightarrow \infty.$$
Further, the solutions $u^N \in C([0,1],X)$ also satisfy an a-priori bound 
\begin{equation}\label{secondaprioribound}
    \|u^N\|_{L^{\infty}([0,1],H^{1+\sigma, \sigma}_{\omega})} \lesssim \|u_0^N\|_{H_{\omega}^{1+\sigma,\sigma}} \lesssim \|u_0\|_{Y_{\omega}^{\sigma}}
\end{equation}
uniformly in $N$. This follows by the same argument as the a-priori bound in Lemma \ref{Hshbound}, albeit using the actual Benjamin-Ono equation instead of a regularization and Theorem \ref{aprioriformal}. The calculations in this Gronwall argument are justified for functions in $C([0,1],X)$ by similar considerations as in Section \ref{sectionenergyestimatepartI}.
Interpolation (Lemma \ref{anisotropicproperties}) between Cauchyness $\|u^N-u^M\|_{C([0,1],L^2)} = o_{N,M \rightarrow \infty}(1)$ 
and the a-priori bound (\ref{secondaprioribound}) yields Cauchyness in $C([0,1],H_{\omega}^{1+\tilde{\sigma},\tilde{\sigma}})$ for any $\tilde{\sigma}<\sigma$.
\end{proof}

\begin{lemma}
The limit function $u:=H_{\omega}^{1+\tilde{\sigma}, \tilde{\sigma}}\textit{- }\kern-2pt lim_{N \rightarrow \infty} u^N$ belongs to the space $$C([0,1],H_{\omega}^{1+\tilde{\sigma}, \tilde{\sigma}}) \cap L^4([0,1],V^{1,\infty}_{\omega}).$$
\end{lemma}

\begin{proof}
    The closed unit ball in $L^4([0,1],V^{1,\infty}_{\omega})$ is weakly$^*$ sequentially compact, in the sense that every bounded sequence $u_n$ in $L_t^4([0,1],V^{1,\infty}_{\omega})$ has a subsequence $u_{n_k}$ such that both $u_{n_k}\rightharpoonup \tilde{u} $ and $\partial_x u_{n_k} \rightharpoonup \partial_x \tilde{u} $ weakly$^*$ in $L_t^4([0,1],L^{\infty}(\mathbb{T}^2))$ for some function $\tilde{u}\in L_t^4([0,1],V^{1,\infty}_{\omega})$. This follows by considering the map $$V^{1,\infty}_{\omega} \rightarrow L_t^4([0,1],L^{\infty}(\mathbb{T}^2) \times L^{\infty}(\mathbb{T}^2)), \ u \mapsto (u,\partial_xu)$$ and weak$^*$ compactness of the space $L_t^4([0,1],L^{\infty}(\mathbb{T}^2) \times L^{\infty}(\mathbb{T}^2))$ by Banach-Alaoglu, the latter space being the dual of the separable normed vector space $L_t^{4/3}([0,1],L^1(\mathbb{T}^2) \times L^1(\mathbb{T}^2))$. Indeed, if $u_{n_k}\rightharpoonup \tilde{u} $ and $\partial_x u_{n_k} \rightharpoonup v $ weakly$^*$ in $L_t^4([0,1],L^{\infty})$, one can verify that 
    $\partial_x\tilde{u} = v$ as distributions by testing against $\psi \in C^{\infty}_c([0,1] \times \mathbb{T}^2) \subseteq L^{4/3}([0,1],L^{\infty}(\mathbb{T}^2))$. 

    Now weak$^*$ convergence in $L^4_t([0,1],L^{\infty}(\mathbb{T}^2))$ and convergence in $C([0,1],H_{\omega}^{1+\tilde{\sigma}, \tilde{\sigma}})$ both imply convergence as distributions, so that $u = \tilde{u}$ as distributions. Thus, \begin{equation*} u \in C([0,1],H_{\omega}^{1+\tilde{\sigma}, \tilde{\sigma}}) \cap L^4([0,1], V^{1,\infty}_{\omega}).\qedhere \end{equation*}
\end{proof}

\begin{lemma}
    The limit function $u\in C([0,1],H_{\omega}^{1+\tilde{\sigma}, \tilde{\sigma}})$ 
    is a distributional solution to the Benjamin--Ono equation with initial data $u_0$.
\end{lemma}

\begin{proof}
Each $u^N$ is a classical solution, and thus in particular a distributional solution satisfying 
$$\int_{[0,1]} \int_{\mathbb{R}} u^N\cdot \partial_t \psi \ dxdt = \int_{[0,1]} \int_{\mathbb{R}} Hu^N \cdot \partial_{xx}\psi + (u^N)^2 \cdot \partial_x\psi \ dx dt$$ for any $\psi \in C_c^{\infty}([0,1] \times \mathbb{R}).$
By convergence in $C([0,1], H^{\tilde{\sigma},1+\tilde{\sigma}}_{\omega})$, boundedness of the quasi-periodic Hilbert transform $H$ on $L^2$-based Sobolev spaces, and by anisotropic Sobolev embedding \eqref{anisotropicembedding} we have convergence $u^N \rightarrow u$, $Hu^N \rightarrow Hu$ and $(u^N)^2 \rightarrow u^2$ in $C([0,1],C(\mathbb{R}))$. By bounded convergence, the limit function $u$ is a distributional solution.
\end{proof}

\begin{lemma}
    The solution is unique in $C([0,1],H_{\omega}^{1+\tilde{\sigma}, \tilde{\sigma}}) \cap L^4([0,1],V^{1,\infty}_{\omega})$ among limits of smooth solutions. 
\end{lemma}

\begin{proof}
Suppose that $\tilde{u}$ is the limit in the sense of $C([0,1],H_{\omega}^{1+\tilde{\sigma},\tilde{\sigma}}) \cap L^4([0,1],V^{1,\infty}_{\omega})$ of smooth solutions $\tilde{u}^N$ with initial data $\tilde{u}^N(0) \rightarrow u(0)$ in $H_{\omega}^{1+\tilde{\sigma}, \tilde{\sigma}}$ (in fact, continuity $\tilde{u}^N:[0,1] \rightarrow X$ suffices). Let $u \in C([0,1],H^{1+\tilde{\sigma},\tilde{\sigma}}_{\omega}) \cap L^4([0,1],V_{\omega}^{1,\infty})$ be the solution as already constructed, and let $u^N$ be the approximating sequence as in the construction. The same Gronwall argument as before yields
$$\|(\tilde{u}^N - u^N)(t)\|_{L^2} \lesssim \|\tilde{u}^N(0)-u^N(0)\|_{L^2} \cdot \exp \bigg( \frac{1}{2} \int_0^t \|\partial_x \tilde{u}^N(s)\|_{L^{\infty}} + \|\partial_x u^N(s)\|_{L^{\infty}} ds \bigg).$$
The factor involving the exponential is uniformly bounded in $N$ and $t \in [0,1]$ due to Theorem \ref{aprioriformal} and convergence of $\partial_x\tilde{u}^N$ in $L^4_t([0,1],L^{\infty}(\mathbb{T}^2)) \hookrightarrow L^1_t([0,1],L^{\infty}(\mathbb{T}^2))$. 
Interpolating with the a-priori estimate $\|u^N\|_{L^{\infty}([0,1],H^{1+\tilde{\sigma}, \tilde{\sigma}}_{\omega})} \lesssim \|u_0^N\|_{H_{\omega}^{1+\tilde{\sigma},\tilde{\sigma}}}$ (see (\ref{secondaprioribound})) gives $$\|\tilde{u}^N - u^N\|_{C([0,1],C(\mathbb{T}^2))}\lesssim\|\tilde{u}^N - u^N\|_{C([0,1],H_{\omega}^{1+\tilde{\tilde{\sigma}},\tilde{\tilde{\sigma}}})} \lesssim \|\tilde{u}^N(0)-u^N(0)\|_{H_{\omega}^{\tilde{\tilde{\sigma}},1+\tilde{\tilde{\sigma}}}}.$$ So, taking $N \rightarrow \infty$ yields $\tilde{u} = u$ pointwise. 
\end{proof}

\begin{lemma}\label{Lipschitzmaintheorem}
    The data-to-solution map is Lipschitz continuous from $B_Y(0,\rho_0)$ equipped with the $L^2$-norm to $C([0,1],L^2)$. 
    The data-to-solution map is also continuous from $B_Y(0,\rho_0)$ equipped with the $Y$-norm to $C([0,1],H_{\omega}^{1+\tilde{\sigma}, \tilde{\sigma}})$.
\end{lemma}

\begin{proof}
Let $u_0, \tilde{u}_0 \in B_Y(0,\rho_0)$, and denote by $u_N, \tilde{u}_N:[0,1] \rightarrow X$ the emanating solutions for regularized data $u_0^N, \tilde{u}_0^N$. As before, a Gronwald argument gives 
$$\|(u^N - \tilde{u}^N)(t)\|_{L^2} \lesssim \|u_0^N - \tilde{u}_0^N\|_{L^2} \cdot \exp \bigg( \frac{1}{2} \int_0^t \|\partial_xu^N(s)\|_{L^{\infty}}+\|\partial_x \tilde{u}^N(s)\|_{L^{\infty}}ds \bigg),$$
and the factor involving the exponential is uniformly bounded in $N$ and $t \in [0,1]$ due to Theorem \ref{aprioriformal}. 
Taking $N \rightarrow \infty$ yields Lipschitzness
$$\|u - \tilde{u}\|_{C([0,1],L^2)} \lesssim \|u_0 - \tilde{u}_0\|_{L^2}.$$ 
Interpolating with \eqref{secondaprioribound} yields the second statement.
\end{proof}
\noindent Lemmas \ref{Cauchynessfinalapproximation}-\ref{Lipschitzmaintheorem} together complete the proof of Theorem \ref{maintheorem}. 
\end{proof}

\section{Diophantine conditions}\label{Diophantinesection}
In this short section we discuss how imposing Diophantine conditions on the frequency vector $\omega=(\omega_1,\omega_2)$ allows us to embed ordinary Sobolev spaces $H^{s}_{\omega}$ of quasi-periodic functions into the space $Y^{\sigma}_{\omega}$. As a corollary, we obtain Theorem \ref{maincorollary}. We discuss the notion of badly approximable numbers and Roth-Liouville irrationality measure, in so far as they are relevant to obtaining embeddings. 
For a more general discussion, we refer the reader to \cite{MR2953186}. \\

 Let $\alpha \in \mathbb{R}$ be a given real number, and $\nu \geq 1$. If there exist infinitely many pairs of coprime numbers $(p,q), q>0$ such that \begin{equation}\label{RothLiouvillecondition}\Big|\alpha - \frac{p}{q}\Big| < \frac{1}{q^{\nu}},\end{equation} then the sequence $\frac{p}{q}$ as $q \rightarrow \infty$ is in some sense a good approximation of $\alpha$. The larger $\nu$, the better the approximation. The Liouville-Roth irrationality measure $\mu(\alpha)$ is defined to be the supremum of all $\nu$ for which such a sequence exists. Consequently, given any $\varepsilon>0$, we have that \begin{equation}\label{consequenceRothLiouville}\Big|\alpha - \frac{m}{n}\Big| \gtrsim_{\varepsilon} \frac{1}{n^{\mu(\alpha)+\varepsilon}} \quad \text{for all rational numbers } \ \frac{m}{n}\end{equation}

Now let $\omega=(\omega_1,\omega_2) \in \mathbb{R}^2$ be a frequency vector and assume $\alpha:=\omega_1/\omega_2$ has Roth-Liouville irrationality measure $\mu(\alpha)$. Then
\begin{equation}\label{consequence RothLiouville2}
    |\omega \cdot n|^{-1} \lesssim_{\varepsilon} \langle |n|\rangle^{\mu(\alpha)-1+\varepsilon}
\end{equation}
due to
$$|\omega \cdot n|=\Big|\alpha+\frac{n_2}{n_1}\Big| |\omega_2 n_1| \gtrsim |n_1|^{-\mu(\alpha)-\varepsilon} |n_1| \gtrsim \langle |n| \rangle^{-\mu(\alpha)+1-\varepsilon},$$
and thus
\begin{equation}\label{nablaboundRothLiouville}
    \|\partial_x^{-1}u\|_{L^2} \lesssim \|\langle \nabla \rangle^{\mu(\alpha)-1+\varepsilon}u\|_{L^2}.
\end{equation}
This gives the following embedding.

\begin{lemma}\label{embeddingirrationalitymeasure}
    Suppose $\omega=(\omega_1,\omega_2) \in \mathbb{R}^2$ is a frequency vector and the ratio of its entries $\alpha:=\omega_1/\omega_2$ is irrational and has Roth-Liouville irrationality measure $\mu(\alpha)$. Then 
    \begin{equation*}
        H^{s}_{\omega} \hookrightarrow Y^{\sigma}_{\omega}, \quad \text{whenever} \ 0 \leq \sigma<s-\mu(\alpha)+1.
    \end{equation*}
\end{lemma}
Combining this in the case $\mu(\alpha)=2$ with Theorem \ref{maintheorem}, yields Theorem \ref{maincorollary}. \\

 A few remarks are in order: A rational number has irrationality measure $1$. Every irrational number has irrationality measure at least $2$, as a consequence of Dirichlet's approximation theorem. Almost all real numbers have irrationality measure exactly $2$ \cite[Appendix E]{MR2953186}! Numbers with irrationality measure $\infty$ are called Liouville numbers. \\

\begin{remark}\label{badlyapproximable}
We comment on the notion of badly approximable numbers in relation to above discussion. An irrational number $\alpha$ is called \textit{badly approximable} iff there is some bound $M$ on all the $a_i$'s appearing in its continued fraction expansion  $\alpha = [a_0;a_1, a_2, . . . ]$. For a discussion of continued fractions see for example \cite[Chapter 3]{MR2723325}. 
The terminology stems from the following characterization.

\begin{proposition}
\cite[Proposition 3.10]{MR2723325}:
    An irrational number $\alpha$ is badly approximable if and only if \begin{equation}\label{badapproximabilitycondition}\bigg|\alpha - \frac{p}{q} \bigg| \gtrsim \frac{1}{q^2} \quad \text{for all rational numbers} \ \frac{p}{q}.\end{equation}
\end{proposition}
\noindent A badly approximable number has Roth-Liouville irrationality measure $2$. Indeed, if (\ref{RothLiouvillecondition}) were to hold for infinitely many coprime $(p,q),q>0$ for some $\nu>2$, then (\ref{badapproximabilitycondition}) would fail for any implicit constant. While the badly approximable numbers form a set of measure zero \cite[Section 3.3]{MR2723325}, almost all real numbers have Roth-Liouville irrationality measure exactly $2$ (as mentioned before). 

But based on \eqref{badapproximabilitycondition} one can say slightly more than based on just $\mu(\alpha)=2$: the estimates \eqref{consequenceRothLiouville}, \eqref{consequence RothLiouville2} and \eqref{nablaboundRothLiouville} hold for $\mu(\alpha)=2$ even \textit{without} the $\varepsilon$-derivative loss, and we have the embedding $$H^{1+\sigma}_{\omega} \hookrightarrow Y^{\sigma}_{\omega}.$$
This is a slight improvement over the statement of Lemma \ref{embeddingirrationalitymeasure} with $\mu(\alpha)=2$; The strict inequality $\sigma<s-\mu(\alpha)+1$ in Lemma \ref{embeddingirrationalitymeasure} is necessitated by the $\varepsilon$-derivative loss in (\ref{nablaboundRothLiouville}). 
However, as the threshold $\sigma>7/8$ in our main theorem (Theorem \ref{maintheorem}) is a strict inequality, this does not yield any improvement with regards to Theorem \ref{maincorollary}. 
\end{remark}

\printbibliography 

\end{document}